\documentclass[11pt,twoside]{article}
\usepackage{amsmath}
\usepackage{amssymb}
\usepackage{amsthm}
\usepackage{mathrsfs}
\usepackage{simplewick}
\usepackage{times}
\usepackage{color}
\usepackage{hep}
\usepackage{leftidx}
\usepackage{graphicx}
\usepackage{float}
\usepackage{enumerate}
\usepackage{titletoc}
\usepackage{epstopdf}
\usepackage{ulem}
\usepackage{appendix}
\usepackage{enumitem}
\usepackage{hyperref}

\allowdisplaybreaks

\def\cQ{\mathcal Q}

\def\cX{\mathcal X}
\def\cY{\mathcal Y}

\def\N{\mathop{\mathbb N\kern 0pt}\nolimits}
\def\Z{\mathop{\mathbb Z\kern 0pt}\nolimits}
\def\Q{\mathop{\mathbb Q\kern 0pt}\nolimits}
\def\R{\mathop{\mathbb R\kern 0pt}\nolimits}
\def\T{\mathop{\mathbb T\kern 0pt}\nolimits}
\def\C{\mathop{\mathbb C\kern 0pt}\nolimits}

\def\ds{\displaystyle}
\def\f{\frac}

\def\p{\partial}
\def\eps{\epsilon}
\def\ve{\varepsilon}

\def\ls{\lesssim}

\newcommand{\w}[1]{\langle {#1} \rangle}

\newcommand{\pk}{\dot{P}_k}
\newcommand{\Pk}{P_k}
\newcommand{\q}{\dot{Q}}
\newcommand{\abs}[1]{\left|#1\right|}

\newcommand{\dS}{\mathrm{d}\sigma_y}
\newcommand{\Id}{\mathrm{Id}}

\newcommand{\Ltw}[1]{\|#1\|_{L^2(w)}}
\newcommand{\Atwo}[1]{[#1]_{A_2}}
\newcommand{\Rj}[1]{R_{#1}}

\DeclareMathOperator{\Div}{div}

\DeclareMathOperator{\supp}{supp}
\newcommand{\Sph}{\mathbb{S}^2}
\hypersetup{colorlinks=true,linkcolor=blue,citecolor=red,urlcolor=cyan}

\theoremstyle{plain}
\newtheorem{theorem}{Theorem}[section]

\newtheorem{lemma}[theorem]{Lemma}
\newtheorem{cor}[theorem]{Corollary}

\theoremstyle{definition}
\newtheorem{definition}[theorem]{Definition}

\newtheorem{remark}{Remark}[section]

\numberwithin{equation}{section}

\title{Long time smooth solutions of 3D cubic quasilinear wave systems with small weakly decaying initial data}

\begin{document}
\author{
  Gao Mu$^{1}$ \quad Li Jun$^{1}$ \quad Yin Huicheng$^{2,}$
   \footnote{Gao Mu (\texttt{876117149@qq.com}) and Li Jun (\texttt{lijun@nju.edu.cn})
    are supported by the National key research and development program of China (No.2024YFA1013301).
   Yin Huicheng (\texttt{huicheng@nju.edu.cn}, \texttt{05407@njnu.edu.cn}) are supported by the NSFC (No.12331007, No.12101304).}
    \\[0.5cm]
  \small
    $^{1}$ School of Mathematics, Nanjing University, Nanjing 210009, China\\
  \small
   $^{2}$ School of Mathematical Sciences, Nanjing Normal University, Nanjing 210023, China
}
\date{}
\maketitle

\thispagestyle{empty}
\begin{abstract}
For the 3D cubic quasilinear wave system $\square_{c_i} u^i=G^i(u,\p u,\p^2u)
=\ds\sum_{\substack{0\le|\alpha|,|\beta|,|\gamma|\le1 \\ 1\le j,k,l \le m}} g_{\alpha\beta\gamma}^{ijkl}\p^{\alpha}\p u^j$ $\p^{\beta}u^k\p^{\gamma}u^l$,
it is well known that global solution $u$ exists when the small smooth initial data
$(u,\p_tu)|_{t=0}$ $=(u_0(x), u_1(x))$ are compactly supported
or decay rapidly at spatial infinity. However, when $(u_0, u_1)\in (H^{s+1}, H^s)$ with $s>\f{5}{2}$ are small, it remains unknown
whether $u$ exists globally or not. In this paper, we show that if $\|u_{0}\|_{H^{N+1}}+\|u_{1}\|_{H^N}\le\ve$
($N\ge 6$) is small, then the almost global solution $u$ exists in $[0, T_{\ve}]$ with $T_{\ve}\ge e^{C\ve^{-1}}$
for the general $G(u,\p u,\p^2u)$ depending on $u$ and $T_{\ve}\ge e^{C\ve^{-2}}$ for the nonlinearity $G(\p u,\p^2u)$ independent of $u$, respectively.
In addition, if $\ds\sum_{|a|\le 5}\|\w{x}^{\mu}\p^a_x(u _0,u_1)\|_{L^2}\le\ve$ holds for any fixed constant $\mu\in (0,1)$, then
the solution $u$ exists globally and meanwhile the scattering property of $u$ is derived. Our main ingredients consist in establishing a series of new weighted $L^\infty-L^2$ estimates and Strichartz estimates based on the strong Huygens' principle for 3D linear wave equations.
	\vskip 0.2 true cm
	\noindent

\qquad \quad  Huygens' principle, scattering
	\vskip 0.2 true cm
	\noindent
\end{abstract}

\vskip 0.2 true cm

\addtocontents{toc}{\protect\thispagestyle{empty}}
\tableofcontents

\section{Introduction}

\subsection{Main results and remarks}

In this paper, we are concerned with the Cauchy problem of the cubic quasilinear wave system with multiple speeds
\begin{equation}\label{eq:wave}
\left\{
\begin{aligned}
\square_{c_i} u^i = &G^i(u,\p u,\p^2 u), \qquad\quad (t,x)\in [0,\infty)\times\R^3,\ i = 1,\cdots,m, \\
(u,\p_tu)&(0,x)=(u_{0},u_{1})(x),\qquad x\in\R^3,
\end{aligned}
\right.
\end{equation}
where $t=x_0$, $x=(x_1, x_2, x_3)\in\mathbb{R}^3$, $m\in\Bbb Z_+$, $\partial=(\partial_0, \partial_1, \p_2, \partial_3)
=(\partial_t, \partial_{x_1}, \p_{x_2}, \partial_{x_3})$, $\partial_x=\nabla=(\partial_1, \p_2 , \partial_3)$,
$\square_{c_i}=\p_t^2-c_i^2\Delta$ with $c_i>0$ for $1\le i\le m$ and $\Delta=\p_1^2+\p_2^2+\p_3^2$, $G=(G^1, \cdots, G^m)$,
$u=(u^1, \cdots, u^m)$
and $u_{r}=(u^1_{r},\cdots,u^m_{r})\in H^{N+1-r}(\Bbb R^3)$ for integers $N\ge 4$ and $r \in \{0,1\}$.
In addition, the nonlinearities
\begin{equation}\label{eq:c}
\begin{split}
G^i(u,\p u,\p^2 u)=&
\sum_{j,k,l=1}^m\big[\sum_{\alpha,\beta,\gamma,\delta=0}^3 Q^{\alpha\beta\gamma\delta}_{1ijkl}\partial_{\alpha\beta}^2 u^j\partial_\gamma u^k \partial_\delta u^l+\sum_{\alpha,\beta,\gamma=0}^3 Q^{\alpha\beta\gamma}_{2ijkl}\partial_{\alpha\beta}^2 u^j\partial_\gamma u^k u^l\\
&+\sum_{\alpha,\beta=0}^3 Q^{\alpha\beta}_{3ijkl} \partial_{\alpha\beta}^2 u^ju^ku^l
+\sum_{\alpha,\beta,\gamma=0}^3 S^{\alpha\beta\gamma}_{1ijkl} \partial_{\alpha} u^j\partial_\beta u^k\partial_\gamma u^l\\
&+\sum_{\alpha,\beta=0}^3 S^{\alpha\beta}_{2ijkl} \partial_{\alpha} u^j\partial_\beta u^k u^l
+\sum_{\alpha=0}^3 S^{\alpha}_{3ijkl} \partial_{\alpha} u^j u^k u^l\big]
\end{split}
\end{equation}
with $Q^{\alpha\beta\gamma\delta}_{1ijkl},\ Q^{\alpha\beta\gamma}_{2ijkl},\ Q^{\alpha\beta}_{3ijkl},\ S^{\alpha\beta\gamma}_{1ijkl},\ S^{\alpha\beta}_{2ijkl}$ and $S^{\alpha}_{3ijkl}$ being constants. Meanwhile, the following symmetric condition is imposed $\text{for any } \alpha, \beta, \gamma, \delta = 0, 1, \cdots, 3$ and $i,j,k,l =1,\cdots, m,$
\begin{equation}\label{eq:symmetric condition}
Q^{\alpha\beta\gamma\delta}_{1ijkl} = Q^{\beta\alpha\gamma\delta}_{1ijkl} = Q^{\alpha\beta\gamma\delta}_{1jikl},
\quad Q^{\alpha\beta\gamma}_{2ijkl} = Q^{\beta\alpha\gamma}_{2ijkl} = Q^{\alpha\beta\gamma}_{2jikl},
\quad Q^{\alpha\beta}_{3ijkl} = Q^{\beta\alpha}_{3ijkl} = Q^{\alpha\beta}_{3jikl}.
\end{equation}
Our main results can be stated as follows.
\begin{theorem}\label{thm:1}
Let integer $N \ge 4$. There is a constant $\kappa_0>0$ such that if $\ve>0$ is small and
\begin{equation}\label{initial:data}
\|u_{0}\|_{H^{N+1}(\R^3)}+\|u_{1}\|_{H^N(\R^3)}\le\ve,
\end{equation}
then \eqref{eq:wave} admits an almost global solution $u\in C([0,T_\ve];H^{N+1}(\R^3))\cap C^1([0,T_\ve];H^N(\R^3))$ with $T_\ve \ge e^{\kappa_0\ve^{-1}}-e$. In particular, if $G^i(u,\p u,\p^2 u)$ ($1\le i\le m$) are independent of $u$,
that is,
\begin{equation}\label{eq:C independent u}
G^i(u,\p u,\p^2 u)=\ds\sum_{j,k,l=1}^m \big[\ds\sum_{\alpha,\beta,\gamma,\delta=0}^3 Q^{\alpha\beta\gamma\delta}_{1ijkl} \partial_{\alpha\beta}^2 u^j\partial_\gamma u^k\partial_\delta u^l+\ds\sum_{\alpha,\beta,\gamma=0}^3 S^{\alpha\beta\gamma}_{1ijkl}\partial_{\alpha} u^j \partial_\beta u^k\partial_\gamma u^l\big],
\end{equation}
then $T_\varepsilon$ can be improved to $T_\varepsilon\ge e^{\kappa_0 \varepsilon^{-2}}-e$.
\end{theorem}

\vskip 0.1cm
\begin{theorem}\label{thm:2}
Let $\mu\in(0,1)$ be any fixed number and integer $N\ge \lceil4+2\mu\rceil$. Then there is a constant $\ve_0>0$ such that
for $\ve\in(0,\ve_0)$,
when $(u_{0},u_{1})$ satisfies
\begin{equation}\label{initial:data2}
\|u_{0}\|_{H^{N+1}(\R^3)}+\|u_1\|_{H^N(\R^3)}+\sum_{|a|\le 5}\|(1+|x|^2)^{\mu/2}\p^a_x(\nabla u _0,u_1)\|_{L^2(\R^3)}\le\ve,
\end{equation}
\eqref{eq:wave} admits a global solution $u\in C([0, \infty);H^{N+1}(\R^3))\cap C^1([0, \infty);H^N(\R^3))$
with the following estimates
\begin{equation}\label{YHCCC-1}
\|\p u\|_{L^\infty([0,\infty);H^N(\R^3))}\le C\varepsilon,
\end{equation}
\begin{equation}\label{YHCCC-2}
(1+t)^{\mu^-}\big(\|u(t,\cdot)\|_{L^\infty(\R^3)}+\|\p u(t,\cdot)\|_{L^\infty(\R^3)}\big)\le C\varepsilon,
\end{equation}
\begin{equation}\label{YHCCC-3}
\|(1+t+|x|)^{\f{1}{2}\mu^-}u\|_{L^{2}([0,\infty); L^\infty(\R^3))}
+\sum_{|b|\le 1}\|(1+t+|x|)^{\f{1}{2}\mu^-}\p_x^b\p u\|_{L^{2}([0,\infty); L^\infty(\R^3))}\le C\varepsilon,
\end{equation}
where $\mu^-$ denotes any fixed positive constant less than $\mu$.
Moreover, the solution $u$ scatters to a free solution: there exist functions $u_{r}^\infty = (u^{\infty,1}_{r},\cdots,u^{\infty,m}_{r})$
for $r = 0,1$ with $(\nabla u_0^\infty, u_1^\infty) \in H^{N-1}(\R^3)$ such that the solution $u^\infty=(u^{\infty,1},\cdots,u^{\infty,m})$ of the linear wave system $\square_{c_i} u^{\infty,i }= 0$ for $i=1,\cdots, m$ with the initial data $(u_0^\infty, u_1^\infty)$ fulfills
	\begin{equation}\label{scatter to a free solution}
	\lim_{t \to +\infty} \|\partial (u(t,\cdot) - u^\infty(t,\cdot))\|_{H^{N-1}(\R^3)} = 0.	
	\end{equation}
\end{theorem}

\begin{remark}
	Consider the following more general 3D fully nonlinear wave equation system
	\[
	\begin{aligned}
	\square_{c_i} u^i =& G^i(u,\partial u, \partial^2 u) \\
	=&\sum_{j,k,l=1}^m\big[\sum_{\alpha,\beta,\gamma,\delta,\mu,\nu=0}^3 F_{1ijkl}^{\alpha\beta\gamma\delta\mu\nu} \partial_{\alpha\beta}^2 u^j \partial_{\gamma\delta}^2 u^k\partial_{\mu\nu}^2 u^l+
	 \sum_{\alpha,\beta,\gamma,\delta,\mu=0}^3 F_{2ijkl}^{\alpha\beta\gamma\delta\mu} \partial_{\alpha\beta}^2 u^j\partial_{\gamma\delta}^2 u^k \partial_\mu u^l\\
	&+\sum_{\alpha,\beta,\gamma,\delta=0}^3 Q^{\alpha\beta\gamma\delta}_{1ijkl} \partial_{\alpha\beta}^2 u^j\partial_\gamma u^k\partial_\delta u^l
	+\sum_{\alpha,\beta,\gamma=0}^3 Q^{\alpha\beta\gamma}_{2ijkl} \partial_{\alpha\beta}^2 u^j\partial_\gamma u^ku^l\\
	&+\sum_{\alpha,\beta=0}^3 Q^{\alpha\beta}_{3ijkl} \partial_{\alpha\beta}^2 u^j u^k  u^l
	+\sum_{\alpha,\beta,\gamma=0}^3 S^{\alpha\beta\gamma}_{1ijkl} \partial_{\alpha} u^j\partial_\beta u^k\partial_\gamma u^l\\
	&+\sum_{\alpha,\beta=0}^3 S^{\alpha\beta}_{2ijkl} \partial_{\alpha} u^j\partial_\beta u^k  u^l
	+\sum_{\alpha=0}^3 S^{\alpha}_{3ijkl} \partial_{\alpha} u^j  u^k u^l \big],
	\end{aligned}
	\]
where $F_{1ijkl}^{\alpha\beta\gamma\delta\mu\nu}$, $F_{2ijkl}^{\alpha\beta\gamma\delta\mu}$, $Q^{\alpha\beta\gamma\delta}_{1ijkl},\ Q^{\alpha\beta\gamma}_{2ijkl},\ Q^{\alpha\beta}_{3ijkl},\ S^{\alpha\beta\gamma}_{1ijkl},\ S^{\alpha\beta}_{2ijkl}$
and $S^{\alpha}_{3ijkl}$ are constants; meanwhile, the analogous symmetric conditions in \eqref{eq:symmetric condition} are imposed.
Then by the quasilinearization of the fully nonlinear wave equations (see Remark 4 on page 116
of \cite{Hormander}), the similar conclusions as in Theorems \ref{thm:1}--\ref{thm:2} still hold.
\end{remark}

\begin{remark}\label{rmk:1.4}
For $(u_0, u_1)(x)\in (H^{N+1}, H^N)(\R^3)$
with  $N\in\Bbb \Z_+$ being sufficiently large, when
\begin{equation}\label{YHCC-0}
\begin{split}
\||x|(|\nabla| u_0, u_1)\|_{H^2}+\||\nabla|(|x|^2 (|\nabla|u_0, u_1))\|_{H^1} +\|u_0\|_{H^{N+1}}+
\|u_1\|_{H^{N}}\le\ve
\end{split}
\end{equation}
is small, the authors in \cite{Pusateri} have established the global solution $u$ for the quasilinear wave system in \eqref{eq:wave} with the quadratic null form nonlinearities or cubic nonlinearities by the space-time resonance method. It is pointed out that
the weights $|x|$ and $|x|^2$ in \eqref{YHCC-0} play essential roles in the proof procedure on the global existence of
$u$ (see the norm in (3.6) of \cite{Pusateri}). However, for the lower weighted Sobolev space in \eqref{initial:data2}
with the weight $\w{x}^{\mu}$ ($\mu>0$ is any fixed constant), it seems difficult for us to apply the corresponding methods and results
in \cite{Pusateri} directly. In the present paper, we establish a series of new weighted Strichartz estimates instead of the
obtained $L^{\infty}$ and weighted $H^k$ estimates in (3.5)--(3.6) of \cite{Pusateri} to derive the uniform
estimates \eqref{YHCCC-1}--\eqref{YHCCC-3}.
\end{remark}

\begin{remark}\label{rmk:1.4-1}
Let $u$ solve the Cauchy problem of the quasilinear Klein-Gordon equation
\begin{equation}\label{KG-01}
\left\{
\begin{aligned}
&\Box u+u=F(u,\p u,\p^2u),\quad(t,x)\in[0,\infty)\times\R^d,\\
&(u,\p_tu)(0,x)=\ve(u_0,u_1)(x),
\end{aligned}
\right.
\end{equation}
where $x=(x_1,\cdots,x_d)$, $d\ge1$, $\ve>0$ is small, $(u_0,u_1)\in (H^{s+1}, H^s)$ with $s>0$ being a suitably large constant,
and the smooth nonlinearity $F(u,\p u,\p^2u)$ is quadratic
and is linear in $\p^2u$. When $F(u,\p u,\p^2u)=F(u,\p u)$ satisfies the null condition, the authors in \cite{DelortFang}
proved that the solution $u\in C([0,T_\varepsilon); H^{s+1})\cap  C^1([0,T_\varepsilon); H^{s})$
exists, where $T_\varepsilon\ge Ce^{C\varepsilon^{-\mu}}$ for $\mu=1$ if $d\ge3$ and $\mu=2/3$ if $d=2$.
For $d=1$, $T_{\ve}\ge \frac{C}{\ve^4 |\ln\ve|^6}$ was shown in \cite{Delort97}.
Recently, without the restriction of null condition for the general $F(u,\p u, \p^2 u)$ in \eqref{KG-01}, the authors
in \cite{HouYin2} have established $u\in C([0,T_\varepsilon); H^{s+1}(\Bbb R^d))
\cap  C^1([0,T_\varepsilon); H^{s}(\Bbb R^d))$ with $T_\varepsilon=+\infty$ if $d\ge3$,
$T_\varepsilon\ge e^{C\varepsilon^{-2}}$ if $d=2$ and $T_\varepsilon\ge \frac{C}{\ve^4}$ if $d=1$.
Meanwhile, for $d=2$ and any fixed number $\mu>0$, if
$\|u_0\|_{H^{s+1}(\R^2)}+\|u_1\|_{H^{s}(\R^2)}
+\|\w{x}^\mu u_0\|_{L^2(\R^2)}+\|\w{x}^\mu u_1\|_{L^2(\R^2)}\le 1$,
then it was proved in \cite{HouYin2} that \eqref{KG-01} admits a  global solution
$u\in C([0,\infty);H^{s+1}(\R^2))\cap C^1([0,\infty);H^{s}(\R^2))$.
In addition, for $d=1$ and the quadratic nonlinearity $F(u,\p u,\p^2u)=F(u,\p u)$,
when $\|u_0\|_{H^{s+1}(\R)}+\|u_1\|_{H^s(\R)}
+\ds\sum_{k=0}^{14}\|\w{x}^\mu\p^k(u_0,u_1)\|_{L^2(\R)}\le 1$ for any $\mu>\f12$,
then the authors in \cite{HouTaoYin} proved $u\in C([0,T_\varepsilon); H^{s+1}(\Bbb R))
\cap  C^1([0,T_\varepsilon); H^{s}(\Bbb R))$ with the optimal $T_\varepsilon\ge e^{C\varepsilon^{-2}}$.
Note that the symbols  $\pm\sqrt{1+|\xi|^2}$ ($\xi=(\xi_1,\xi_2,\xi_3)\in\Bbb R^3$) coming from the
spatial Fourier transformation on the Klein-Gordon operator $\p_t^2-\Delta +1$ in \eqref{KG-01} are obviously smooth
near $\xi=0$ while the symbols $\pm|\xi|$ coming from the Laplacian operator $-\Delta$
in the wave equation of \eqref{eq:wave} are singular near $\xi=0$. This implies that the phase
of the Fourier integral operator from the solution expression of $\square w+w=f$ can be analyzed more delicately  near $\xi=0$
and some kinds of better space-time decay estimates than those for the wave equation $\square\tilde w=\tilde f$ are obtained (see \cite{DengPusateri,Zheng1,Zheng2,HouYin2,HouTaoYin} and Chapters 6--7 of \cite{Hormander}). Based on this, the almost global or global solutions to problem \eqref{KG-01} with weakly decaying initial data
mentioned above are established. However, these analyses for the Klein-Gordon equation are not suitable for
the wave equation due to different phase functions near $\xi=0$,
so far there are few corresponding long time existence results for the nonlinear wave equations with the small initial data
in the lower weighted Sobolev spaces.
\end{remark}

\subsection{Applications of main results}

The quasilinear wave systems in \eqref{eq:wave} include several interesting physical or geometric models.
\begin{itemize}
	\item \textit{3D relativistic membrane equation}
	\begin{equation}\label{eq:1.6}
	\partial_t\big( \frac{\partial_t u}{\sqrt{1-(\partial_t u)^2 + |\nabla u|^2}} \big)
	- \Div\big( \frac{\nabla u}{\sqrt{1-(\partial_t u)^2 + |\nabla u|^2}} \big) = 0,
	\end{equation}
	where $\nabla u = (\partial_1 u,\p_2 u,\partial_3 u)$. Note that \eqref{eq:1.6} is actually the Euler-Lagrange equation
	of the volume functional $\int_{\mathbb{R}\times\mathbb{R}^3} \sqrt{1-(\partial_t u)^2 + |\nabla u|^2}dtdx$ for the embedding of
	$(t,x) \mapsto (t,x,u(t,x))$ in the Minkowski spacetime.
	It is easy to check that \eqref{eq:1.6} is equivalent to the following nonlinear equation for the $C^2$ solution $u$:
	\begin{equation*}
	\Box u=(|\partial_t u|^2-|\nabla u|^2)(\partial_t^2 u-\Delta u)-(\partial_t u)^2 \partial_t^2 u
+ 2 \partial_t u\nabla u \cdot  \partial_t\nabla u - \f12\nabla u \cdot \nabla(|\nabla u|^2).
	\end{equation*}
	This implies that the nonlinear equation \eqref{eq:1.6} admits a cubic nonlinearity in \eqref{eq:wave}.
	\item \textit{3D nonlinear membrane equation}
	\begin{equation}\label{eq:1.7}
	\partial_t^2 u - \Div\big(\frac{\nabla u}{\sqrt{1+|\nabla u|^2}}\big) = 0,
	\end{equation}
	where $u(t,x)$ stands for the position of the membrane at the point $(t,x)$. It follows from \eqref{eq:1.7} and direct computation that
	for $u\in C^2$,
\begin{equation*}\label{eq:1.7-1}
\Box u =|\nabla u|^2\Delta u-\ds\sum_{i,j=1}^3\p_iu\p_ju\p_{ij}^2u-\f{3}{2}|\nabla u|^2\p_t^2u+O(|\nabla u|^4)\p_t^2u.
\end{equation*}
Obviously, \eqref{eq:1.7} is a second order quasilinear wave equation with cubic nonlinearity.

\item \textit{3D nonlinear wave equation}
\begin{equation}\label{eq:1.8}
-\left(1+(\partial_t u)^2\right)\partial_t^2 u + \Delta u = 0,
\end{equation}
which can be regarded as a model equation from the nonlinear version of
Maxwell's equations in the electromagnetic theory (see \cite{MiaoYu}, Section 1.3), where $u$ represents the scalar potential
function. It is easy to check that \eqref{eq:1.8} is a cubic quasilinear wave equation included in \eqref{eq:wave}.

\item \textit{3D scalar nematic liquid crystal equation}
\begin{equation}\label{eq:1.6-0}
\partial_tu^2-c(u)\Div(c(u)\nabla u)=0,
\end{equation}
where $c(u)=\alpha \cos^2u+\beta \sin^2 u$, the constants $\alpha>0$ and $\beta>0$ with $\alpha \neq \beta$.
With respect to the derivation of equation \eqref{eq:1.6-0}, one can see \cite{Crit-1}, \cite{Crit-2} and \cite{Crit-3}.
In this case, $c(u)=\alpha+(\beta-\alpha)u^2+O(u^3)$ for small $|u|$. It follows from \eqref{eq:1.6-0} and direct computation that
for $u\in C^2$,
\begin{equation*}
	\Box_\alpha u = 2\alpha(\beta-\alpha)(u^2\Delta u+ u|\nabla u|^2)+O(u^3)(\Delta u +|\nabla u|^2).
\end{equation*}
Then \eqref{eq:1.6-0} admits the cubic nonlinearity
imposed in \eqref{eq:wave}.

	\item \textit{3D wave maps system}

	Let $(\mathcal{M},g)$ be a $m$-dimensional compact Riemannian manifold without boundary. A wave map is a continuous
function from the Minkowski space $\mathbb{R}^{1+3}$ into $\mathcal{M}$:
	$$
	u=(u^1, ..., u^m) : \mathbb{R} \times \mathbb{R}^3 \longrightarrow \mathcal{M},
	$$
	which corresponds to a critical point of the following functional
	$F(u) = \int_{\mathbb{R}^{1+3}} \langle \eta^{\alpha\beta} \partial_\beta u, \partial_\alpha u \rangle_g dt dx$
	with $\eta=\mathrm{diag}(-1,1,1,1)$ being the standard Minkowski metric on $\mathbb{R}^{1+3}$ in rectangular coordinates. When the local coordinates on $\mathcal{M}$ are introduced, the equations of $u^i$ ($1\leq i\leq m$) can be written as
\begin{equation}\label{eq:wave maps}
\Box u^i
=\sum_{j,k=1}^3\Gamma_{jk}^i(u)\big(\partial_t u^j\partial_t u^k-\sum_{l=1}^3\partial_{x_l}u^j\partial_{x_l}u^k\big), \quad 1\leq i\leq m,
\end{equation}
where $\Gamma_{jk}^i(u)$'s are the Christoffel symbols of the metric $g$. For more detailed derivation on the wave maps system \eqref{eq:wave maps}, one may refer to 2.1.2 of Chapter 2 in \cite{Shatah}. When we consider the small data solution of \eqref{eq:wave maps}, by carrying out the Taylor expansion on the function $\Gamma_{jk}^i(u)$ under Riemann normal coordinates, the long time existence of \eqref{eq:wave maps} can be attributed to the following system of semilinear wave equations
\begin{equation}\label{eq:wave maps2}
\Box u^i = \sum_{j,k=1}^3 \sum_{l=1}^3 C_{jkl}^i \big(\partial_t u^j \partial_t u^k - \sum_{p=1}^3 \partial_{x_p} u^j \partial_{x_p} u^k \big)u^l, \quad 1 \leq i \leq m.
\end{equation}
It is easy to know that the 3D wave maps system \eqref{eq:wave maps} or \eqref{eq:wave maps2}
admits the cubic nonlinearity as in problem \eqref{eq:wave}.
\end{itemize}

In terms of Theorems \ref{thm:1}--\ref{thm:2}, we have
\begin{cor}
\begin{itemize}
\item[(A)]  Under the conditions of Theorem \ref{thm:1},
the equations \eqref{eq:1.6}, \eqref{eq:1.7} and  \eqref{eq:1.8} have almost global solutions
$u\in C([0,T_\ve];H^{N+1}(\R^3))\cap C^1([0,T_\ve];H^N(\R^3))$ with $T_\ve\ge e^{C\ve^{-2}}$,
while the equation \eqref{eq:1.6-0} or \eqref{eq:wave maps} has a solution $u\in C([0,T_\ve];H^{N+1}(\R^3))\cap C^1([0,T_\ve];H^N(\R^3))$ with $T_\ve\ge e^{C\ve^{-1}}$.
\item[(B)]  Under the conditions of Theorem \ref{thm:2},
the equations \eqref{eq:1.6}--\eqref{eq:wave maps}  have the global solutions
$u\in C([0,\infty);H^{N+1}(\R^3))\cap C^1([0,\infty);H^N(\R^3))$.
\end{itemize}
\end{cor}

\begin{remark}\label{Rem-5A}
{\it It is pointed out that there have been remarkable results
for the local and global existence of low regularity solutions to the $d$-dimensional wave
maps system in  $\R^{1+d}$
($d\ge 2$) due to the null forms of  \eqref{eq:wave maps}. For example,
when the initial data $(u, \p_tu)(0,x)\in H^{s-1}(\Bbb R^d)\times H^{s-1}(\Bbb R^d)$ with $s>\f{d}{2}$,
the local solution $u\in C([0, T]; H^{s}(\Bbb R^d))\cap C^1([0, T]; H^{s-1}(\Bbb R^d))$ exists (see \cite{K-01,K-02});
when $(u, \p_tu)(0,x)$ is small in the homogeneous Besov space $\dot{B}_{2,1}^{d/2}(\Bbb R^d)\times\dot{B}_{2,1}^{d/2-1}(\Bbb R^d)$, the author in \cite{Tataru98,Tataru01} has proved the global
solution of \eqref{eq:wave maps} in $C([0, \infty); \dot{B}_{2,1}^{d/2}(\Bbb R^d))\cap C^1([0, \infty); \dot{B}_{2,1}^{d/2-1}(\Bbb R^d))$;
when $(u, \p_t u)(0,x)$ is small in the critical Sobolev space $\dot{H}^{d/2}(\Bbb R^d)\times\dot{H}^{d/2-1}(\Bbb R^d)$,
the global solution of \eqref{eq:wave maps} in $C([0, \infty); \dot{H}^{d/2}(\Bbb R^d))
\cap C^1([0, \infty); \dot{H}^{d/2-1}(\Bbb R^d))$
is established in \cite{TaoIMRN,TaoCMP}.}
\end{remark}

\subsection{Previous results}

We now briefly recall some basic results related to our work. Consider the
Cauchy problem of the second order nonlinear wave equation
\begin{equation}\label{eq:1.9}
\begin{cases}
\square u = F(u,\partial u,\partial^2 u), \quad (t,x) \in [0,\infty)\times\mathbb{R}^d,\\
(u,\partial_t u)(0,x) = (\varepsilon u_0,\varepsilon u_1)(x),
\end{cases}
\end{equation}
where $F(u,\partial u,\partial^2 u)=O(|u^2|+|\partial u|^2+|\partial^2 u|^2)$, $d\geq 2$, $\varepsilon>0$ is sufficiently small.
If the smooth functions $(u_0,u_1)(x)$ are compactly supported or decay rapidly at infinity, there have been
systematic global existence results as follows.
\vskip 0.1 true cm
{\bf $\bullet$ When $d \geq 4$},  \eqref{eq:1.9}
has a global small data smooth solution $u$ (see \cite{Hormander}, \cite{KlainermanPonce}, \cite{LiChen}).
\vskip 0.1 true cm
{\bf $\bullet$ When $d = 3$}, if  $F(u,\partial u,\partial^2 u)=F(\partial u,\partial^2 u)$ satisfies the related null condition,
then the global existence of small solution $u$ to \eqref{eq:1.9} has been shown in \cite{Christodoulou} and \cite{Klainerman2};
if $F(u,\partial u,\partial^2 u)$ fulfills the weak null condition, it is also shown that the global solution of \eqref{eq:1.9}
still exists (see \cite{Alinhac2}, \cite{DingLiuYin} and \cite{Lindblad}).
\vskip 0.1 true cm
{\bf $\bullet$ When $d=2$}, if $F(u,\partial u,\partial^2 u)=F(\partial u,\partial^2 u)$ satisfies both the first and second null conditions, the author in \cite{Alinhac1} proved that \eqref{eq:1.9} has a global smooth solution $u$.
\vskip 0.1 true cm
On the other hand, when $(u_0,u_1)(x)$ are not compactly supported or decay at infinity with the higher order weight
$\w{x}^{\mu}$ for $\mu\ge 1$ in \eqref{initial:data2}, there are numerous global existence results of small solution
$u$ to the general nonlinearity in \eqref{eq:1.9}, one can see \cite{Asa}, \cite{CaiLei}, \cite{HouYin1} and \cite{Pusateri}.
Analogous long time existence results in higher order space-weighted Sobolev spaces have been obtained for the 3D multi-speed wave systems, see \cite{Kat, Kji,Sideris,Pusateri}. However, so far there are few long time existence results for the problem \eqref{eq:wave}
with general cubic nonlinearities in lower-order space-weighted Sobolev spaces except the result in \cite{HouYin3} for 4D cubic quasilinear wave equations with the initial data in the Sobolev space $H^s(\Bbb R^4)$ ($s>1$ is suitably large).

\subsection{Sketch of the proofs}

We now give an outline of the proofs on Theorems \ref{thm:1}--\ref{thm:2}. When
$(u_0,u_1)(x)$ decay suitably fast at infinity, the global existence of smooth solution $u$ to \eqref{eq:wave} with $d = 3$ and $m=1$ has been established in \cite{Klainerman3}. The key points of the related analyses in \cite{Klainerman3} are based on the standard energy estimates for wave
equations, the smallness of $\|u(t,\cdot)\|_{H^{\lfloor \frac{d}{2} \rfloor+2}(\mathbb{R}^d)}$ and the following Klainerman-Sobolev inequality:
\begin{equation}\label{eq:1.14}
|\partial u(t,x)| \leq \frac{C}{(1+|t-r|)^{\frac{1}{2}}(1+t)^{\frac{d-1}{2}}}
\sum_{|I| \leq \lfloor \frac{d}{2} \rfloor+1} \|Z^I \partial u(t,x)\|_{L_x^2(\mathbb{R}^d)},
\end{equation}
where $Z \in \{\partial_t,x_i\partial_j - x_j\partial_i,x_i\partial_t + t\partial_i,1 \leq i,j \leq d,t\partial_t + \sum\limits_{k=1}^d x_k\partial_k\}$ and
$r = |x| = \sqrt{(x_1)^2+\cdots+(x_d)^2}$.
In order to apply the vector field method and \eqref{eq:1.14} to study problem \eqref{eq:wave}, the weighted norm
\begin{equation}\label{YHCCC-20}
\begin{split}
\|(1+|x|)^2(\nabla u_0,u_1)\|_{L^2(\mathbb{R}^3)}<\infty
\end{split}
\end{equation}
should be satisfied. However, for the initial data $(u_0,u_1)(x)$ in \eqref{initial:data2} with $\mu\in (0,1)$,
\eqref{YHCCC-20} may be violated in general. Meanwhile,
these  vector fields in \eqref{eq:1.14} are not suitable for the wave systems with different wave speeds.

On the other hand, the following KSS-type estimates (see \cite[Prop. 2.1]{KSS02jam} or (1.7)--(1.8) on page 190 of \cite{MetcalfeSogge06})
\begin{equation}\label{KSSestimate-1}
\begin{split}
&(\ln(2+t))^{-1/2}\big(\|\w{x}^{-1/2}\p v\|_{L_t^2L_x^2([0,t]\times\R^3)}+\|\w{x}^{-3/2}v\|_{L_t^2L_x^2([0,t]\times\R^3)}\big)\\
&\ls\|\p v(0,x)\|_{L^2(\R^3)}+\int_0^t\|\Box v(s,x)\|_{L^2(\R^3)}ds
\end{split}
\end{equation}
or
\begin{equation}\label{KSSestimate-2}
\begin{split}
&\|\w{x}^{-1/2^-}\p v\|_{L_t^2L_x^2([0,t]\times\R^3)}+\|\w{x}^{-3/2^-}v\|_{L_t^2L_x^2([0,t]\times\R^3)}\\
&\ls\|\p v(0,x)\|_{L^2(\R^3)}+\int_0^t\|\Box v(s,x)\|_{L^2(\R^3)}ds
\end{split}
\end{equation}
cannot be utilized to study problem \eqref{eq:wave} directly, where $\w{x}=(1+|x|^2)^{1/2}$.
Indeed,  by virtue of \eqref{KSSestimate-2} and Theorem \ref{thm:Weighted Strichartz estimate} below, when the weight $\mu$
in \eqref{initial:data2} is less than $1/2$ or $\mu=0$,
note that the terms in the nonlinearity of \eqref{eq:wave}
\begin{equation*}\label{YHCC-21}
\begin{split}
\p^{\le 1} u\p^{\le 1} u\p u
=
\underbrace{\w{x}^{1/4^+}\p^{\le 1} u}_{\substack{\notin L^{4}([0,t];L^{\infty})\\ \text{due to Theorem \ref{thm:Weighted Strichartz estimate}}}}
\underbrace{\w{x}^{1/4^+}\p^{\le 1} u}_{\substack{\notin L^{4}([0,t];L^{\infty})\\ \text{due to Theorem \ref{thm:Weighted Strichartz estimate}}}}
\underbrace{\w{x}^{-1/2^-}\p u}_{\in L^2([0,t];L^2)}
\notin  L^1([0,t];L^2(\R^3))\quad\text{hold},
\end{split}
\end{equation*}
then it seems difficult for us to obtain the uniform energy estimate of the solution $u$ to \eqref{eq:wave} when $\mu$ in \eqref{initial:data2} is sufficiently small or $\mu=0$.

To overcome those difficulties mentioned above, instead of the usual $L^\infty-L^2$ type estimate in \eqref{eq:1.14} and the KSS estimates in \eqref{KSSestimate-1}--\eqref{KSSestimate-2}, we establish a series of new $L^\infty-L^2$ estimates (see Subsection \ref{sec:Linfty-L^2 estimate})
and weakly space-time weighted Strichartz estimates (see Theorem \ref{thm:Weighted Strichartz estimate})
based on the strong Huygens' principle for the 3D linear wave equation.
Without loss of generality, in order to illustrate the proofs of Theorems \ref{thm:1}--\ref{thm:2} more explicitly, it suffices to
take the cubic semilinear form $\Box u = \partial u \cdot \partial u \cdot \partial u $ in \eqref{eq:wave} for instance.
At first, we start with the proof of Theorem \ref{thm:1}. To obtain the long time $H^N$ norm
of solution derivatives $\partial u$, by the standard energy estimates, the following condition is required
\begin{equation}\label{YHCC-22}
	\int_0^t \|\partial u(s)\cdot \partial u(s)\cdot \partial u (s)\|_{H_x^N} ds<\infty.
\end{equation}
In this case, one expects that
\begin{equation}\label{YHCC-23}
	\underbrace{\partial u}_{L_t^2L_x^{\infty}}\underbrace{\partial u}_{L_t^2L_x^{\infty}}
	\underbrace{\partial u}_{L_t^\infty H_x^N}\in  L_t^1H_x^N.
\end{equation}
However, the endpoint $L^2_t L^\infty_x$ Strichartz estimate fails for the 3D linear wave equation. This means that \eqref{YHCC-23} may not hold for all $t\in [0,\infty)$, and only almost global solutions may be obtained for the  initial data \eqref{initial:data}
with small Sobolev norms.

Secondly, we turn to the proof of Theorem \ref{thm:2}. In this process, as mentioned above, it is essential to verify condition \eqref{YHCC-22}.
Although the classical endpoint $L^2_t L^\infty_x$ Strichartz estimate is absent, inspired by \cite[Lemma 5.3]{HouYin2},
one may derive some weighted Strichartz estimates instead (see Theorem \ref{thm:Weighted Strichartz estimate} below).
In this case, condition \eqref{YHCC-22} follows from
\begin{equation*}
	\begin{split}
		|\partial u\partial u\partial u|\le
		\underbrace{|A_{\frac{1}{2}\mu^-} \partial u|}_{L^{2}_t L^{\infty}_x}
		\underbrace{|A_{\frac{1}{2}\mu^-} \partial u|}_{L^{2}_t L^{\infty}_x}
		\underbrace{|\partial u|}_{L^{\infty}_t H_x^{N}}
		\in  L_t^1H_x^N,
	\end{split}
\end{equation*}
where the space-time weight $A_{\nu}(t,x)= (1+t+|x|)^\nu$ for $\nu>0$.
Therefore, the key step is to show that $A_{\frac{1}{2}\mu^-} \partial u \in L^2_t L^\infty_x$. In this procedure,
it is essential to establish (see the details in Sections \ref{sec:3}--\ref{Section 4})
\begin{equation}\label{YHCC-26}
	\int_0^t \|A_{\mu}(s,x)\partial u\cdot \partial u\cdot \partial u\|_{L_x^2} ds<\infty.
\end{equation}
This will be achieved by
\begin{equation*}
A_{\mu}\partial u\partial u\partial u
=\underbrace{A_{\frac{1}{2}\mu^-}\partial u}_{\in L_t^{2^+}L_x^{\infty^-}}
\underbrace{A_{\frac{1}{2}\mu^-}\partial u}_{\in L_t^{2^+}L_x^{\infty^-}}
\underbrace{A_{0^+}\partial u}_{\in L_t^{\infty^-} L_x^{2^+}}\in  L^1([0,t];L^2(\mathbb{R}^3))
\end{equation*}
and the conclusions of $A_{\frac{1}{2}\mu^-} \partial u \in L_t^{2^+} L_x^{\infty^-}$ and $A_{0^+} \partial u \in L_t^{\infty^-} L_x^{2^+}$
derived from Theorem \ref{thm:Weighted Strichartz estimate}.

Finally, for the cubic quasilinear problem \eqref{eq:wave}, we can obtain such estimates (see Section \ref{Section 4}):
\begin{equation}\label{YHCC-30}
\|\partial u\|_{L^\infty_tH^{N}_x} \lesssim \|\partial u(0)\|_{H^{N}_x}
+\| \partial^{\le1} u\|_{A_0\tilde{L}^2_t B^{1+\delta}_{\infty,1} } ^2 \|\partial u\|_{L^\infty_tH^{N}_x},
\end{equation}
\begin{equation}\label{YHCC-31}
\begin{split}
\|\partial^{\le1} u\|_{A_{\f12\mu^-}\tilde{L}^2_t B^{1+\delta}_{\infty,1} }
\ls
\|\w{x}^{\mu}\p^{\le3}_x\p u(0)\|_{L^2_x}
+\| \partial u\|_{A_{0^+}\tilde{L}^{\infty^-}_t B^{3+\mu^-}_{2^+,1} }
\| \partial^{\le1} u\|_{A_{\f12\mu^-}\tilde{L}^{2^+}_t B^{2+\mu^-}_{\infty^-,1} }^2,
\end{split}
\end{equation}
\begin{equation}\label{YHCC-32}
	\begin{split}
		\| \partial u\|_{A_{0^+}\tilde{L}^{\infty^-}_t B^{3+\mu^-}_{2^+,1} }
		\ls
		\|\w{x}^{\mu}\p^{\le4}_x\p u(0)\|_{L^2_x}
		+
		\| \partial^{\le1} u\|_{A_{0^+}\tilde{L}^2_t B^{1+\delta}_{\infty,1} }^2
		\|\p u\|_{L^\infty_tH^{N}_x},
	\end{split}
\end{equation}
\begin{equation}\label{YHCC-33}
	\begin{split}
		\| \partial^{\le1} u\|_{A_{\f12\mu^-}\tilde{L}^{2^+}_t B^{2+\mu^-}_{\infty^-,1} }
		\ls
		\|\w{x}^{\mu}\p^{\le5}_x\p u(0)\|_{L^2_x}
		+
		\| \partial^{\le1} u\|_{A_{\f12\mu^-}\tilde{L}^2_t B^{1+\delta}_{\infty,1} }^2
		\|\p u\|_{L^\infty_tH^{N}_x},
	\end{split}
\end{equation}
where $\mu\in(0,1)$, $N\ge 4+2\mu$, $\delta\in(0,10^{-4}\mu)$, and
$\|f\|_{A_\nu \tilde{L}^p_t B^s_{r,1}}=\ds\sum_{k\ge-1}2^{sk}\|A_\nu(s,x)P_k f\|_{L^p([0,t];L^r)}$.
It follows from \eqref{YHCC-30}--\eqref{YHCC-33}, $\|\partial^{\le1} u\|_{A_0\tilde{L}^2_t B^{1+\delta}_{\infty,1}}\le \|\partial^{\le1} u\|_{A_{\f12\mu^-}\tilde{L}^2_t B^{1+\delta}_{\infty,1}}$  and the initial data \eqref{initial:data2} that
the global uniform smallness of $\|\partial^{\le1} u\|_{A_{\f12\mu^-}\tilde{L}^2_t B^{1+\delta}_{\infty,1}}$, $\|\p^{\le1} u\|_{A_0\tilde{L}^2_t B^{1+\delta}_{\infty,1}}$ and $\|\partial u\|_{L^\infty_tH^{N}_x}$ is obtained.
From this, the global existence result in Theorem \ref{thm:2} can be shown by the continuity argument.

The paper is organized as follows. In Section \ref{Section 2}, some preliminaries including the notations for Littlewood-Paley decomposition, the strong Huygens' principle for 3D linear wave equation, classical linear dispersive estimates and Strichartz estimates are presented.
In Section \ref{sec:3}, we establish some crucial space-time weighted Strichartz estimates, localized $L^\infty-L^2$ estimates and weighted $L^2-L^2$ estimates.  The uniform energy estimate on the solution $u$ to problem \eqref{eq:wave} is derived in Section \ref{Section 4}.
Based on the estimates in Sections \ref{Section 2}--\ref{Section 4}, the proofs of Theorem \ref{thm:1} and Theorem \ref{thm:2}
are shown in Section \ref{Section 5} and Section \ref{Section 6}, respectively.
In addition, some basic results and related  techniques on the $A_2$ weight are collected in Appendix A for the reader's convenience.

\subsection{Notations}
\begin{description}
	\item [$\blacktriangleright $]$\p = (\p_t,\nabla) = (\p_t,\p_1,\p_2,\p_3)$ .
	\item [$\blacktriangleright $]$\p_x^a =\p_1^{a_1}\p_2^{a_2}\p_3^{a_3}$ for $a\in \N_0^3$.
	\item [$\blacktriangleright $]For $f = (f^1,\cdots,f^m)$, $\p f=   (\p f^1,\cdots,\p f^m)$, $\ds|\p^{\le j}f|=\big(\sum_{0\le|a|
           \le j}|\p^af|^2\big)^\frac12$.
	\item [$\blacktriangleright $]$\p u(0) = (u_1,\nabla u_0) = (u_1^1,\cdots,u_1^m,\nabla u_0^1,\cdots,\nabla u_0^m)$ .
	\item [$\blacktriangleright $]If $\|\cdot\|$ is a norm and $f = (f^1,\cdots,f^m)$, then $\|f\|=\ds\sum_{i=1}^{m}\|f^i\|$.
	\item [$\blacktriangleright $] $\|f\|_{L^p(\omega)}=\|f\|_{L^p(\R^d,\omega(x)dx)} = (\int_{\R^d} |f(x)|^p\omega(x)dx)^{1/p}$ for $p \in [1,\infty)$ and $\omega(x)\ge 0$.
	\item [$\blacktriangleright $]$\w{x}= (1+|x|^2)^{1/2}$.
	\item [$\blacktriangleright $]$A_{\nu}(t,x)=(1+t+|x|)^\nu$ for $\nu>0$.
	\item [$\blacktriangleright $]$\mu^-$ denotes any constant slightly less than $\mu$, $\mu^+$ denotes any constant slightly larger than $\mu$.
	\item [$\blacktriangleright $]$k_+=\max\{0,k\},\ k_-=\min\{0,k\}$.
	\item [$\blacktriangleright $]$\Pk$ (resp. $\pk$) is the nonhomogeneous (resp. homogeneous) Littlewood-Paley projection onto frequency $2^k$ (see the details in \eqref{YHCCC-001}).
	\item [$\blacktriangleright $] \begin{equation*}
		\begin{split}
			\cY_k&=\cY_k^1\cup\cY_k^2,\\
			\cY_k^1&=\{(k_1,k_2,k_3)\in\Z^3: |\max_{i=1,2,3}\{k_i\}-k|\le4,k_1,k_2,k_3\ge-1\},\\
			\cY_k^2&=\{(k_1,k_2,k_3)\in\Z^3: \max_{i=1,2,3}\{k_i\}\ge k+4,
			\max_{i=1,2,3}\{k_i\}-{\rm med}_{i=1,2,3}\{k_i\}\le4,k_1,k_2,k_3\ge-1\}.
		\end{split}
	\end{equation*}
	\item [$\blacktriangleright $]$Q_j$ (resp. $\q_j$) is the dyadic decomposition in the physical space $\R^3$ (see \eqref{def:Q_j}
below).
	\item [$\blacktriangleright $]For $f = (f^1,\cdots,f^m)$, $P_k f = (P_k f^1,\cdots,P_k f^m)$, $\pk f = (\pk f^1,\cdots,\pk f^m)$ and $Q_j f = (Q_j f^1,\cdots,Q_j f^m)$.
	\item [$\blacktriangleright $]$R=(R_1,R_2,R_3)=\ds\f{\nabla}{|\nabla|}$ is the Riesz transformation.
	\item [$\blacktriangleright $]$S(x,t) = \{y\in \R^3: |x-y|=t\}$.
	\item [$\blacktriangleright $]$B(x,R) = \{y\in \R^3: |x-y|\le R\} $.
	\item [$\blacktriangleright $] $C(x,t,\eps) = \{y\in \R^3: t-\eps\le|x-y|\le t+\eps\}$.
	\item [$\blacktriangleright $] For $t\ge 1$, the cutoff function $\Xi_{S(x,t)}$ is defined as
	\begin{equation*}
		\Xi_{S(x,t)} (y) = 	\left\{
		\begin{aligned}
			&1,\quad\quad\quad\quad\quad\quad\quad t-0.01\le|y-x|\le t+0.01,\\
			&\in[0,1], \quad\text{smooth, \qquad otherwise},\\
			&0,\quad\quad\quad\quad\quad\quad\quad |y-x|\le t-0.02\quad\text{or}\ |y-x|\ge t+0.02;
		\end{aligned}
		\right.
	\end{equation*}
for $t<1$, $\Xi_{S(x,t)}$ is defined as
	\begin{equation*}
		\Xi_{S(x,t)} (y) = 	\left\{
		\begin{aligned}
			&1,\quad\quad\quad\quad\quad\quad\quad |y-x|\le t+0.01,\\
			&\in[0,1],\quad\text{smooth, \quad otherwise,}\\
			&0,\quad\quad\quad\quad\quad\quad\quad |y-x|\ge t+0.02.
		\end{aligned}
		\right.
	\end{equation*}
\item [$\blacktriangleright $] For the non-negative quantities $f$ and $g$,
	$f\ls g$ means $f\le C_1g$ and $f\approx g$ means $C_1g\le f\le C_2g$ with $C_1$ and $C_2$ being generic positive constants.
\end{description}

\section{Preliminaries}\label{Section 2}
	\subsection{Notations for dyadic decomposition}\label{section 2.1}
	For the function $f$ on $\mathbb{R}^3$, define its Fourier transformation as
	\begin{align*}
		\hat{f}(\xi)=\mathscr{F}f(\xi)=\int_{\mathbb{R}^3} e^{-i x\cdot\xi}f(x)dx\quad\text{with $i=\sqrt{-1}$ and $\xi\in\Bbb R^3$}.
	\end{align*}
	Choosing a smooth cutoff function $\psi\colon\mathbb{R}\rightarrow[0,1]$, which equals 1 on $[-5/4,5/4]$ and vanishes outside
	$[-8/5,8/5]$. Set
	\begin{align*}
		&\dot{\psi}_{k}(x)=\psi(|x|/2^{k})-\psi(|x|/2^{k-1}),\quad k\in\mathbb{Z},\\
		&\psi_{k}(x)=\dot{\psi}_k(x),\quad k\in\mathbb{N}_0,\\
		&\psi_{-1}(x)=1-\sum_{k\geq 0}\psi_{k}(x)=\psi(2|x|),\\
		&\dot{\psi}_{I}=\sum_{k\in I\cap\mathbb{Z}}\dot{\psi}_{k}, \quad {\psi}_{I}=\sum_{k\in I\cap\mathbb{Z}\cap[-1,+\infty)}{\psi}_{k},
	\end{align*}
	where $I$ is any subset of $\mathbb{R}$. Note that $\psi_0=1$ on $\{x\in\R^3:|x|\in[4/5,5/4]\}$ and $\supp\psi_0\subseteq\{x\in\R^3:|x|\in [5/8,8/5]\}$. Let $\pk$ be the homogeneous
	Littlewood-Paley projection onto frequency $2^k$:
	\begin{equation*}
	\mathscr{F}(\pk f)(\xi)=\dot{\psi}_{k}(\xi)\mathscr{F}f(\xi),\quad k\in\mathbb{Z}.
	\end{equation*}
	In addition, for any subset $I$, define
	\begin{align*}
		\dot{P_I}f&=\sum_{k\in I\cap\mathbb{Z}}\pk f,\\
		\dot{P}_{[[k]]}f&=\sum_{l\in[k-1,k+1]\cap\mathbb{Z}}\dot{P}_l f.
	\end{align*}
	Let $\Pk$ be the nonhomogeneous
	Littlewood-Paley projection onto frequency $2^k$:
	\begin{equation}\label{YHCCC-001}
		\begin{split}
		\mathscr{F}(\Pk f)(\xi)={\psi}_{k}(\xi)\mathscr{F}f(\xi),\quad k\in\mathbb{Z}\cap[-1,+\infty).
	\end{split}
	\end{equation}
For any subset $I$, set
	\begin{align*}
		&P_I f=\sum_{k\in I\cap\mathbb{Z}\cap\lbrack-1,+\infty)}\Pk f,\\
		&P_{[[k]]}f= \dot{P}_{[[k]]}f,\ \ \ k\geq 0,\\
		&P_{[[-1]]}f= (P_{-1}+P_{0})f.
	\end{align*}
	It is easy to check that $P_{[[k]]}P_k=P_k$ and $\dot{P}_{[[k]]}\dot{P}_k=\dot{P}_k$.
	Let $\psi_{[[k]]}$ and $\dot{\psi}_{[[k]]}$ denote the symbols of $P_{[[k]]}$ and $\dot{P}_{[[k]]}$, respectively. Note that $\psi_{[[k]]} = \dot{\psi}_{[[k]]}$ for $k\in\N_0$.

	\begin{definition}
		Denote by $\mathcal{S}_h'(\mathbb{R}^3)$ the space of tempered distributions $u\in\mathcal{S}'(\R^3)$ such that
		\[
		\lim_{\lambda \to \infty} \left\| \theta(\lambda D) u \right\|_{L^\infty} = 0 \quad \text{for any } \theta\in  C_0^{\infty}(\mathbb{R}^3).
		\]
	\end{definition}
For $f\in\mathcal{S}'(\R^3)$, $f=\ds\sum_{k\ge-1}P_kf$ in $\mathcal{S}'(\R^3)$ holds;
for $f\in\mathcal{S}'_h(\R^3)$, $f=\ds\sum_{k\in\Z}\pk f$ in $\mathcal{S}'(\R^3)$ holds.
	\begin{lemma}
		Let $u \in \mathcal{S}'(\R^3)$. If $P_{-1}u\in L^p(\R^3)$ for $p\in[1,\infty)$ or $\hat{u}\in L^1_{loc}(\R^3)$, then $u \in \mathcal{S}'_h(\R^3)$.
	\end{lemma}
	\begin{proof}
	Since its proof can be found in \cite[Chapter 1]{Chemin}, we omit the details here.
	\end{proof}
\begin{definition}
Let \( s \in \mathbb{R} \) and \( 1 \leq p, r \leq \infty \). The nonhomogeneous Besov space \( B^s_{p,r}(\R^3) \) consists of all  \( u\in \mathcal{S}' \) with $\|u\|_{B^s_{p,r}(\R^3)}= \big\| \big( 2^{ks} \|P_k u\|_{L^p(\R^3)} \big)_{k \in \mathbb{Z}\cap[- 1,+\infty)} \big\|_{\ell^r(\mathbb{Z}\cap[- 1,+\infty))} < \infty$.  The homogeneous Besov space \( \dot{B}^s_{p,r}(\R^3)\) consists of \( u\in \mathcal{S}'_h \) with $\|u\|_{\dot{B}^s_{p,r}(\R^3)}= \big\|\big( 2^{ks} \|\pk  u\|_{L^p(\R^3)}\big)_{k \in \mathbb{Z}}\big\|_{\ell^r(\mathbb{Z})}<\infty$.
\end{definition}
Note that $H^s(\R^3)=B^s_{2,2}(\R^3)$. For $s>0,\ p\in[1,\infty),\ r\in[1,\infty]$, one has $B^s_{p,r}(\R^3) = \dot{B}^s_{p,r}(\R^3)\cap L^p(\R^3)$.
In addition, we set
	\begin{equation*}
		\begin{split}
			\cX_k&=\cX_k^1\cup\cX_k^2,\\
			\cX_k^1&=\{(k_1,k_2)\in\Z^2: |\max_{i=1,2}\{k_i\}-k|\le4,k_1,k_2\ge-1\},\\
			\cX_k^2&=\{(k_1,k_2)\in\Z^2: \max_{i=1,2}\{k_i\}\ge k+4,
			\max_{i=1,2}\{k_i\}-\min_{i=1,2}\{k_i\}\le4,k_1,k_2\ge-1\};\\
					\cY_k&=\cY_k^1\cup\cY_k^2,\\
			\cY_k^1&=\{(k_1,k_2,k_3)\in\Z^3: |\max_{i=1,2,3}\{k_i\}-k|\le4,k_1,k_2,k_3\ge-1\},\\
			\cY_k^2&=\{(k_1,k_2,k_3)\in\Z^3: \max_{i=1,2,3}\{k_i\}\ge k+4,
			\max_{i=1,2,3}\{k_i\}-{\rm med}_{i=1,2,3}\{k_i\}\le4,k_1,k_2,k_3\ge-1\}.
		\end{split}
	\end{equation*}
As in \cite[page 799]{IP13}, if $P_k\big(\prod_{\iota=1}^2 P_{k_\iota}f_\iota\big)\neq0$,
then $(k_1,k_2)\in\cX_k$; if $P_k\big(\prod_{\iota=1}^3 P_{k_\iota}f_\iota\big)\neq0$, then $(k_1,k_2,k_3)\in\cY_k$.
Obviously, one has $2^k\ls 2^{\max\{k_1,k_2,k_3\}}$ for $(k_1,k_2,k_3)\in\cY_k$.
\vskip 0.1 true cm
We also introduce the dyadic decomposition in the physical space $\R^3$ as follows:
	\begin{equation}\label{def:Q_j}
		\begin{split}
			\q_jf(x)=&\dot{\psi}_j(x)f(x),\quad j\in\Z,\\
			Q_jf(x)=&\psi_j(x)f(x), \quad j\in\Z\cap[-1,+\infty),\\
		\end{split}
	\end{equation}
	Throughout this article, we denote $\w{x} = (1+|x|^2)^{1/2}$. Then one has the following lemma.
	\begin{lemma}\label{lem:Qj prop}
		For $\beta\in(-3/2,3/2)$, it holds
		\begin{equation}\label{ineq:Znorm local}
			2^{j\beta}\|Q_jP_kf\|_{L^2(\R^3)}\ls\|\w{x}^{\beta}f\|_{L^2(\R^3)}
			\ls\big\|2^{j\beta}\|Q_jP_kf\|_{L^2(\R^3)}\big\|_{\ell^1_k\ell^1_j}.
		\end{equation}
	\end{lemma}
	\begin{proof}
It is obvious that $P_k$ is a bounded operator from $L^2(\R^3)$ to $L^2(\R^3)$ with $\|P_kf\|_{L^2(\R^3)}\ls\|f\|_{L^2(\R^3)}$.
Since $\w{x}^{2\beta}$ belongs to the $A_2(\R^3)$ class (see Lemma \ref{lem:w{x} in A2}),
one achieves
 \begin{equation*}
 \|\w{x}^{\beta}P_kf\|_{L^2(\R^3)}\ls\|\w{x}^{\beta}f\|_{L^2(\R^3)}.
 \end{equation*}
This yields
		\begin{equation*}
			2^{j\beta}\|Q_jP_kf\|_{L^2(\R^3)}\ls\|\w{x}^{\beta}P_kf\|_{L^2(\R^3)}
			\ls\|\w{x}^{\beta}f\|_{L^2(\R^3)}.
		\end{equation*}
Thus, the first inequality in \eqref{ineq:Znorm local} is proved.
The second inequality in \eqref{ineq:Znorm local} is obtained by Minkowski's inequality
\begin{equation*}
\|\w{x}^{\beta}f\|_{L^2(\R^3)}\ls\|\sum_{k\ge-1}\w{x}^{\beta}P_kf\|_{L^2(\R^3)}
\ls\sum_{k\ge-1}\|\w{x}^{\beta}P_kf\|_{L^2(\R^3)}
\ls\sum_{k\ge-1}\sum_{j\ge-1}2^{j\beta}\|Q_jP_kf\|_{L^2(\R^3)}.
\end{equation*}
\end{proof}

	\subsection{Linear dispersive estimates and Strichartz estimates}\label{sec:Linear dispersive estimate and Strichartz estimate}

	In this subsection, we present some dispersive estimates and Strichartz estimates for the 3D wave equation. Denote
	\begin{equation}
		\hat{\sigma}(x) = \widehat{\sigma_{\mathbb{S}^2}}(x)=\int_{\mathbb{S}^2}e^{i x\cdot\xi}\;\sigma_{\mathbb{S}^2}\left(d\xi\right).
	\end{equation}
	Then one can write
	\begin{equation*}
	\hat{\sigma}(x)=\int_{\mathbb{S}^2}e^{i |x|\xi_3}\;\sigma_{\mathbb{S}^2}\left(d\xi\right)=e^{i|x|}\omega_{+}(|x|)+e^{-i|x|}\omega_{-}(|x|),
	\end{equation*}
	where $\omega_{\pm}(r)=\int_{\mathbb{S}^2}e^{i r(\xi_3\mp 1)}\chi_\pm(\xi)\sigma_{\mathbb{S}^2}\left(d\xi\right)$,
and $\chi_\pm$ is a partition of unity on $\Sph$ such that $(0,0,\pm1)\in\supp\chi_\pm$.
	\begin{lemma}\label{lem:unit sphere measure}
		$\omega_{\pm}$ are smooth and satisfy
		\begin{equation}\label{CYC-1}
			|\partial_{r}^{k}\omega_{\pm}(r)|\,\le\,{ C}_{k}\,\,(1+r)^{-1-k}\quad\text{for $r\ge 0$}.
		\end{equation}
	\end{lemma}
	\begin{proof}
		For $r\ge 1$, \eqref{CYC-1} is proved in \cite[Corollary 2.37]{Schlag}; for $r<1$, the bound
in \eqref{CYC-1} follows directly from the definitions of $\omega_{\pm}$.
	\end{proof}

	Note that for $\iota\in\{0,\pm1\}$ and $M\in \{0,\cdots,50\}$,
	\begin{equation}\label{def:KklM}
		\begin{split}
			\pk e^{\pm it|\nabla|}|\nabla|^{-\iota}(\Id-\Delta)^{-M}f(x)&=(2\pi)^{-3}\int_{\R^3}K_k^{\iota,M}(t,x-y)P_kf(y)dy,\\
			K_k^{\iota,M}(t,x)&=\int_{\R^3}e^{i(x\cdot\xi\pm t|\xi|)}|\xi|^{-\iota}(1+|\xi|^2)^{-M}\dot{\psi}_{[[k]]}(\xi)d\xi.\\
		\end{split}
	\end{equation}

	\begin{lemma}\label{lem:estimate of KklM}
		For \( k \in \Z \), it holds that for $x\in \R^3$,
		\begin{itemize}
			\item[(1)] when $t\ge 0$, $|K_k^{\iota,M}(t,x)|\ls 2^{(3-\iota)k-2Mk_+}$;
			\item[(2)] when $t>0$, $|x|\le |t|/2$,  $N\in\Bbb N_0$,
			$|K_k^{\iota,M}(t,x)|\ls_N 2^{(3-\iota)k-2Mk_+}(2^kt)^{-N}\ls_N 2^{(3-\iota)k-2Mk_+}(2^k|x|)^{-N}$;
			\item[(3)] when $t>0$, $|t|/2\le |x|\le 2|t|$, $|K_k^{\iota,M}(t,x)|\ls 2^{(3-\iota)k-2Mk_+}(2^kt)^{-1}\approx 2^{(3-\iota)k-2Mk_+}(2^k|x|)^{-1}$;
			\item[(4)] when $t>0$, $|x|\ge 2|t|$, $N\in\Bbb N_0$,
			$|K_k^{\iota,M}(t,x)|\ls_N 2^{(3-\iota)k-2Mk_+}(2^k|x|)^{-N}\ls_N 2^{(3-\iota)k-2Mk_+}(2^kt)^{-N}$,
		\end{itemize}
		where $k_+=\max\{0,k\}$. Therefore, for $t\ge0$ and $x\in\R^3$, one has
		\begin{equation}\label{ineq:estimate of KklM}
			|K_k^{\iota,M}(t,x)|\ls 2^{(3-\iota)k-2Mk_+}(1+2^kt)^{-1}.
		\end{equation}
	\end{lemma}
	\begin{proof}
		Note that
		\begin{equation}\label{2.4}
			K_k^{\iota,M}(t,x)=2^{(3-\iota)k}\int_{\R^3}e^{i2^k(x\cdot\xi\pm t|\xi|)}|\xi|^{-\iota}(1+2^{2k}|\xi|^2)^{-M}\dot{\psi}_{[[0]]}(\xi)d\xi.
		\end{equation}

		(1) comes from \eqref{2.4} directly.\\

		(2) We write  \eqref{2.4} as
		\begin{equation*}
			\begin{split}
				K_k^{\iota,M}(t,x)&=2^{(3-\iota)k}\int_{\R^3}e^{i2^kt\Phi_1(\xi;t,x)}|\xi|^{-\iota}(1+2^{2k}|\xi|^2)^{-M}\dot{\psi}_{[[0]]}(\xi)d\xi,\\
				\Phi_1(\xi;t,x)&=\f{x}{t}\cdot\xi \pm |\xi|.
			\end{split}	
		\end{equation*}
		For $t>0$ and $|x|\le |t|/2$, one arrives at
		\begin{equation*}
			\begin{split}
				&\nabla_\xi \Phi_1(\xi;t,x)=\f{x}{t} \pm \f{\xi}{|\xi|},\qquad |\nabla_\xi \Phi_1(\xi;t,x)|\ge 1-\f{|x|}{t}\ge \f{1}{2},\\
				&|\dot{\psi}_{[[0]](\xi)}\nabla_\xi^n \Phi_1(\xi;t,x)|\ls \dot{\psi}_{[[0]]}(\xi)|\xi|^{1-n}\ls C_n,\qquad n\ge 2,\\
				&|\dot{\psi}_{[[0]]}(\xi)\nabla_\xi^n (1+2^{2k}|\xi|^2)^{-M}|\ls 2^{-2Mk_+},\qquad\qquad n\ge 0,
			\end{split}
		\end{equation*}
		where $C_n$ is a positive constant depending on $n$.

Then $|K_k^{\iota,M}(t,x)|\ls_N 2^{(3-\iota)k-2Mk_+}(2^kt)^{-N}\ls_N 2^{(3-\iota)k-2Mk_+}(2^k|x|)^{-N}$ follows by
a standard stationary phase method (see \cite[Chapter 8]{Chemin}).\\

		(3) For $t>0$ and $|t|/2\le |x|\le 2|t|$, it follows from \eqref{2.4} that
		\begin{equation*}
			K_k^{\iota,M}(t,x)=2^{(3-\iota)k}\int_{0}^{+\infty}\big(\int_{\mathbb{S}^{2}}e^{i2^kx\cdot\rho\omega}d\omega\big)e^{\pm i 2^k t\rho}\dot{\psi}_{[[0]]}(\rho)(1+2^{2k}\rho)^{-M} \rho^{2-\iota}d\rho,
		\end{equation*}
		\begin{equation*}
			\begin{split}
				|K_k^{\iota,M}(t,x)|&\ls 2^{(3-\iota)k}\int_{0}^{+\infty}|\hat{\sigma}(2^k\rho x)|\dot{\psi}_{[[0]]}(\rho) (1+2^{2k}\rho)^{-M} \rho^{2-\iota}d\rho\\
				&\ls  2^{(3-\iota)k}\int_{0}^{+\infty}|2^k\rho x|^{-1}\dot{\psi}_{[[0]]}(\rho) (1+2^{2k}\rho)^{-M} \rho^{2-\iota}d\rho\\
				&\ls 2^{(3-\iota)k-2Mk_+}(2^k|x|)^{-1}\approx 2^{(3-\iota)k-2Mk_+}(2^kt)^{-1}.
			\end{split}
		\end{equation*}

		(4) For $t>0$ and $|x|\ge 2|t|$, \eqref{2.4} can be rewritten as
		\begin{equation*}
			\begin{split}
				K_k^{\iota,M}(t,x)&=2^{(3-\iota)k}\int_{\R^3}e^{i2^k|x|\Phi_2(\xi;t,x)}|\xi|^{-\iota}(1+2^{2k}|\xi|^2)^{-M}\dot{\psi}_{[[0]]}(\xi)d\xi,\\
				\Phi_2(\xi;t,x)&=\f{x}{|x|}\cdot\xi \pm \f{t}{|x|}|\xi|.
			\end{split}	
		\end{equation*}
		Note that
		\begin{equation*}
			\begin{split}
				&\nabla_\xi \Phi_2(\xi;t,x)=\f{x}{|x|} \pm \f{t}{|x|}\f{\xi}{|\xi|},\qquad |\nabla_\xi \Phi_2(\xi;t,x)|\ge 1-\f{t}{|x|}\ge \f{1}{2},\\
				&|\dot{\psi}_{[[0]]}(\xi)\nabla_\xi^n \Phi_2(\xi;t,x)|\ls \f{t}{|x|}|\xi|^{1-n}\dot{\psi}_{[[0]]}(\xi)\ls C_n,\qquad n\ge 2.
			\end{split}
		\end{equation*}
		Then $|K_k^{\iota,M}(t,x)|\ls_N 2^{(3-\iota)k-2Mk_+}(2^k|x|)^{-N}\ls_N 2^{(3-\iota)k-2Mk_+}(2^kt)^{-N}$
is derived by the standard stationary phase method.

 Collecting (1)--(4) yields \eqref{ineq:estimate of KklM}.
 		\end{proof}
On the other hand, one has
	\begin{equation}\label{def:K-1}
		\begin{split}
			P_{-1}e^{\pm it|\nabla|}f(x)&=(2\pi)^{-3}\int_{\R^3}K_{-1}(t,x-y)P_{-1}f(y)dy,\\
			K_{-1}(t,x)&=\int_{\R^3}e^{i(x\cdot\xi\pm t|\xi|)}\psi_{[[-1]]}(\xi)d\xi.
		\end{split}
	\end{equation}

	\begin{lemma}\label{lem:estimate of K-1}
		It holds that for $x\in \R^3$,
		\begin{itemize}
			\item[(1)] when $t\ge 0$, $|K_{-1}(t,x)|\ls 1$;
			\item[(2)] when $t>0$ and $|x|\le |t|/2$, $|K_{-1}(t,x)|\ls t^{-3}\ls |x|^{-3}$;
			\item[(3)] when $t>0$ and $|t|/2\le |x|\le 2|t|$,  $|K_{-1}(t,x)|\ls t^{-1}\approx |x|^{-1}$;
			\item[(4)] when $t>0$ and $|x|\ge 2|t|$,  $|K_{-1}(t,x)|\ls |x|^{-3}\ls t^{-3}$.
		\end{itemize}
		Therefore, for $t\ge 0$ and $x\in\R^3$, one has
		\begin{equation}\label{ineq:estimate of K-1}
			|K_{-1}(t,x)|\ls (1+t)^{-1}.
		\end{equation}
	\end{lemma}

	\begin{proof}
		$K_{-1}(t,x)$ can be rewritten as
		\begin{equation*}
			\begin{split}
		K_{-1}(t,x)=\int_{0}^{+\infty}\big(\int_{\mathbb{S}^{2}}e^{ix\cdot\rho\omega}d\omega\big)e^{\pm it\rho}\psi_{[[-1]]}(\rho) \rho^2d\rho,
			\end{split}
		\end{equation*}
which yields (1). In addition, it follows from Lemma \ref{lem:unit sphere measure} that
		\begin{equation}\label{YHCCC-45}
			K_{-1}(t,x)=\sum_{\mu\in\{\pm\}}\int_{0}^{+\infty}e^{i\rho(\pm t+\mu|x|)}\omega_{\mu}(\rho|x|)\psi_{[[-1]]}(\rho) \rho^2 d\rho.
		\end{equation}
		By applying integration by parts three times to \eqref{YHCCC-45}, we arrive at
		\begin{equation}\label{2.200}
			\begin{split}
				K_{-1}(t,x)=\sum_{\mu\in\{\pm\}}\big(\f{1}{i(\pm t+\mu|x|)}\big)^3\big[-2\omega_{\mu}(0)-
				\int_{0}^{+\infty}e^{i\rho(\pm t+\mu|x|)} \p_\rho^3\big(\omega_{\mu}(\rho|x|)\psi_{[[-1]]}(\rho) \rho^2\big)d\rho\big].
			\end{split}
		\end{equation}
		Note that $0\notin\supp\p_\rho\psi_{[[-1]]}$ and
		\begin{equation*}
			\omega_{+}(0)+\omega_{-}(0)=4\pi,\quad |\omega_{\mu}(\rho |x|)|\le 1,
\quad |\p_\rho^k(\omega_{\mu}(\rho |x|))|\ls_k|x|^k(1+\rho|x|)^{-1-k}.
		\end{equation*}
		Then one has
		\begin{equation*}
			\begin{split}
				&\sum_{k=1}^{3}\int_{0}^{+\infty}\p_\rho^k\left(\omega_{\mu}(\rho|x|)\right)\p_\rho^{3-k}(\rho^2)\psi_{[[-1]]}(\rho) d\rho\\
				&\ls
				\sum_{k=1}^{3}\int_{0}^{+\infty}|x|^k(1+\rho|x|)^{-1-k} \rho^{k-1} \psi_{[[-1]]}(\rho) d\rho\\
				&\ls
				\int_{0}^{+\infty}(1+\rho|x|)^{-2} \psi_{[[-1]]}(\rho) d(\rho|x|)<\infty,\\
			\end{split}
		\end{equation*}
		which implies
		\begin{equation*}
			|K_{-1}(t,x)|\ls\f{1}{|t-|x||^3}.
		\end{equation*}
	Therefore, for $|x|\le |t|/2$ or $|x|\ge 2|t|$, (2) and (4) in Lemma \ref{lem:estimate of K-1} are obtained.
		For $|t|/2\le |x|\le 2|t|$, as in the proof for (3) of Lemma \ref{lem:estimate of KklM} with $l=M=0$,
(3) can be shown. Collecting (1)--(4) yields \eqref{ineq:estimate of K-1}.
	\end{proof}
	Note that
	\begin{equation}\label{def:T-1}
		\begin{split}
			P_{-1}|\nabla|^{-2}e^{\pm it|\nabla|}f(x)&=(2\pi)^{-3}\int_{\R^3}T_{-1}(t,x-y)P_{-1}f(y)dy,\\
			T_{-1}(t,x)&=\int_{\R^3}e^{i(x\cdot\xi\pm t|\xi|)}|\xi|^{-2}\psi_{[[-1]]}(\xi)d\xi.
		\end{split}
	\end{equation}
	\begin{lemma}\label{lem:estimate of T-1}
		For $t\ge 0$ and $x\in\R^3$, it holds that
		\begin{equation}\label{ineq:estimate of T-1}
			|T_{-1}(t,x)|\ls (1+t)^{-1}\ln(e+t).
		\end{equation}
	\end{lemma}
	\begin{proof}
		$T_{-1}(t,x)$ can be rewritten as
		\begin{equation*}
			\begin{split}
				T_{-1}(t,x)=\int_{0}^{+\infty}\big(\int_{\mathbb{S}^{2}}e^{ix\cdot\rho\omega}d\omega\big)e^{\pm it\rho}\psi_{[[-1]]}(\rho) d\rho.
			\end{split}
		\end{equation*}
		As in Lemma \ref{lem:unit sphere measure}, we have
		\begin{equation*}
			T_{-1}(t,x)=\sum_{\mu\in\{\pm\}}\int_{0}^{+\infty}e^{i\rho(\pm t+\mu|x|)}\omega_{\mu}(\rho|x|)\psi_{[[-1]]}(\rho)  d\rho.
		\end{equation*}
		It follows from the integration by parts that
		\begin{equation*}
			\begin{split}
				T_{-1}(t,x)=\sum_{\mu\in\{\pm\}}\f{1}{i(\pm t+\mu|x|)}\big[-\omega_{\mu}(0)-
				\int_{0}^{+\infty}e^{i\rho(\pm t+\mu|x|)}\p_\rho\big(\omega_{\mu}(\rho|x|)\psi_{[[-1]]}(\rho)\big)d\rho\big].
			\end{split}
		\end{equation*}
		Due to $\omega_{+}(0)+\omega_{-}(0)=4\pi$, $|\omega_{\mu}(\rho |x|)|\le 1$ and $|\p_\rho^k(\omega_{\mu}(\rho |x|))|\ls_k|x|^k(1+\rho|x|)^{-1-k}$, then
		\begin{equation}\label{2.201}
			|T_{-1}(t,x)|\ls\f{1}{|t-|x||}.
		\end{equation}
		Therefore, for $|x|\le |t|/2$ or $|x|\ge 2|t|$, \eqref{ineq:estimate of T-1} is proved.
		For $|t|/2\le |x|\le 2|t|$, by $|\omega_{\mu}(\rho |x|)|\ls(1+\rho|x|)^{-1}$, one has
		\begin{equation}\label{2.202}
			\begin{split}
				|T_{-1}(t,x)|&\ls\int_{0}^{+\infty}  (1+\rho|x|)^{-1}\psi_{[[-1]]}(\rho)   d\rho\ls |x|^{-1}\ln(1+|x|)\ls t^{-1}\ln(e+t).
			\end{split}
		\end{equation}
		Combining $|T_{-1}(t,x)|\ls1$ with \eqref{2.201} and \eqref{2.202} yields \eqref{ineq:estimate of T-1}.
	\end{proof}
	We now give the following dispersion estimates for the 3D linear wave equation.
	\begin{lemma}[Linear dispersive estimates]\label{lem:linear dispersive}
		For $t\ge0$ and $k\in \Z$, it holds that
		\begin{equation}\label{ineq:linear dispersive}
			\|\pk e^{\pm it|\nabla|}f\|_{L^\infty(\R^3)} \lesssim 2^{3k}(1+2^kt)^{-1}\|\pk f\|_{L^1(\R^3)},
		\end{equation}
		\begin{equation}\label{ineq:linear dispersive 2}
			\|P_{-1}e^{\pm it|\nabla|}f\|_{L^\infty(\R^3)} \lesssim (1+t)^{-1}\|P_{-1} f\|_{L^1(\R^3)},
		\end{equation}
		\begin{equation}\label{ineq:linear dispersive 3}
			\|P_{-1}|\nabla|^{-2}e^{\pm it|\nabla|}f\|_{L^\infty(\R^3)} \lesssim (1+t)^{-1}\ln(e+t)\|P_{-1} f\|_{L^1(\R^3)}.
		\end{equation}
	\end{lemma}
	\begin{proof}
		Note that
		\begin{equation*}
			\pk e^{\pm it|\nabla|}f(x)=(2\pi)^{-3}\int_{\R^3}K_k^{0,0}(t,x-y) \pk f(y)dy,
		\end{equation*}	
		\begin{equation*}
			P_{-1} e^{\pm it|\nabla|}f(x)=(2\pi)^{-3}\int_{\R^3}K_{-1}(t,x-y) P_{-1} f(y)dy
		\end{equation*}		
		and
		\begin{equation*}
			P_{-1} |\nabla|^{-2}e^{\pm it|\nabla|}f(x)=(2\pi)^{-3}\int_{\R^3}T_{-1}(t,x-y) P_{-1} f(y)dy,
		\end{equation*}		
		where $K^{0,0}_k$, $K_{-1}$ and  $T_{-1}$ are defined in \eqref{def:KklM}, \eqref{def:K-1} and \eqref{def:T-1}, respectively.
		Therefore, by Lemmas \ref{lem:estimate of KklM}--\ref{lem:estimate of T-1}, we have
		\begin{equation*}
			|\pk e^{\pm it|\nabla|}f(x)|
			\ls\int_{\R^3}|K_k^{0,0}(t,x-y)||\pk f(y)|dy
			\lesssim 2^{3k}(1+2^kt)^{-1}\|\pk f\|_{L^1(\R^3)},
		\end{equation*}
		\begin{equation*}
			|P_{-1}e^{\pm it|\nabla|}f(x)|
			\ls\int_{\R^3}|K_{-1}(t,x-y)||P_{-1}f(y)|dy
			\lesssim (1+t)^{-1}\|P_{-1} f\|_{L^1(\R^3)}
		\end{equation*}
		and
		\begin{equation*}
			|P_{-1}|\nabla|^{-2}e^{\pm it|\nabla|}f(x)|
			\ls\int_{\R^3}|T_{-1}(t,x-y)||P_{-1}f(y)|dy
			\lesssim (1+t)^{-1}\ln(e+t)\|P_{-1} f\|_{L^1(\R^3)}.
		\end{equation*}
		Then Lemma \ref{lem:linear dispersive} is proved.
	\end{proof}
	The following Strichartz estimates for the 3D wave equation can be established.
	\begin{lemma}[Linear Strichartz estimates I]\label{lem:strichartz L2Linfty}
		For $p,\ r\in [2,\infty]$ with $\ \f1p+\f1r\le \f12$, \( t \geq t_0 \geq 0 \) and integer \( k \geq -1 \), when $(p,r) \neq (2,\infty)$, it holds that
		\begin{equation}\label{ineq:Strichartz'}
			\|P_ke^{\pm is|\nabla|}f\|_{L^p([t_0,t];L^r(\R^3))}
			\ls{2^{(\f32-\f1p-\f3r) k}}\|P_kf\|_{L^2(\R^3)}.
		\end{equation}
		Moreover, when $(p,r) = (2,\infty)$, one has
		\begin{equation}\label{ineq:strichartz L2Linfty}
			\|P_k e^{\pm is|\nabla|} f\|_{L^2([t_0,t];L^\infty(\mathbb{R}^3))} \lesssim 2^{k}\ln^{\f{1}{2}}(1+2^kt) \|P_k f\|_{L^2(\mathbb{R}^3)}.
		\end{equation}
		In addition, \eqref{ineq:Strichartz'}--\eqref{ineq:strichartz L2Linfty} remain valid when $P_k$ is replaced by $\pk$ for $k\in\Z$.
	\end{lemma}
	\begin{proof}
		For $p\ge 2$, $r\in[1,\infty]$, set
		\begin{equation}\label{2.18}
			\mathcal{B}=\{\phi\in C_0^\infty((t_0,t)\times\R^3): \|\phi\|_{L^{p'}_tL^{r'}_x}\le 1\},
		\end{equation}
		where $1/p+1/p'=1$ and $1/r+1/r'=1$. Then
		\begin{equation}\label{2.19}
			\begin{split}
				\|P_k e^{\pm is|\nabla|} f\|_{L^p([t_0,t];L^r)}
				=&\sup_{\phi\in \mathcal{B}}\int_{t_0}^{t}\big(P_k e^{\pm is|\nabla|} f,\phi(s)\big)_{L^2(\R^3)}ds\\
				=&\sup_{\phi\in \mathcal{B}}\big( P_kf,\int_{t_0}^{t}P_{[[k]]} e^{\mp is|\nabla|}\phi(s)ds\big)_{L^2(\R^3)}\\
				\le& \sup_{\phi\in \mathcal{B}}\|\int_{t_0}^{t}P_{[[k]]} e^{\mp is|\nabla|}\phi(s)ds\|_{L^2(\mathbb{R}^3)}\|P_k f\|_{L^2(\mathbb{R}^3)}.
			\end{split}
		\end{equation}
		We next treat $\|\int_{t_0}^{t}P_{[[k]]} e^{\mp is|\nabla|}\phi(s)ds\|_{L^2(\mathbb{R}^3)}$.
It follows from H\"older's inequality  that
		\begin{equation}\label{2.20}
			\begin{split}
				&\|\int_{t_0}^{t}P_{[[k]]} e^{\mp is|\nabla|}\phi(s)ds\|_{L^2(\mathbb{R}^3)}^2\\
				&=\big(
				\int_{t_0}^{t}P_{[[k]]} e^{\mp is|\nabla|}\phi(s)ds,
				\int_{t_0}^{t}P_{[[k]]} e^{\mp i\tau|\nabla|}\phi(\tau)d\tau
				\big)_{L^2(\mathbb{R}^3)}\\
				&=
				\int_{t_0}^{t}\int_{t_0}^{t}\big( \phi(s),
				P_{[[k]]}^2 e^{\pm i(s-\tau)|\nabla|}\phi(\tau)
				\big)_{L^2(\mathbb{R}^3)}d\tau ds\\
				&\le
				\|\phi(s)\|_{L^{p'}([t_0,t];L^{r'})}
				\big\|\int_{t_0}^{t}
				\|P_{[[k]]}^2 e^{\pm i(s-\tau)|\nabla|} \phi(\tau) \|_{L^{r}(\mathbb{R}^3)} d\tau\big\|_{L^p_s[t_0,t]}.
			\end{split}
		\end{equation}
By \eqref{ineq:linear dispersive}--\eqref{ineq:linear dispersive 2}, the $L^2$ boundedness of operator $P_{[[k]]}^2 e^{\pm i(s-\tau)|\nabla|}$ and the Riesz-Thorin interpolation theorem, one has
		\begin{equation}\label{2.21}
			\|P_{[[k]]}^2 e^{\pm i(s-\tau)|\nabla|} \phi(\tau) \|_{L^r(\mathbb{R}^3)}\lesssim (2^{3k}(1 + 2^k |\tau-s|)^{-1})^{1-\f 2r} \|\phi(\tau) \|_{L^{r'}(\mathbb{R}^3)}.
		\end{equation}
		Substituting \eqref{2.21} into \eqref{2.20} and applying the refined Young's inequality (see \cite[Theorem 1.5]{Chemin}) yield
		\begin{equation}\label{2.22}
			\begin{split}
				&\big\|\int_{t_0}^{t}
				\|P_{[[k]]}^2 e^{\pm i(s-\tau)|\nabla|} \phi(\tau) \|_{L^r(\mathbb{R}^3)} d\tau\big\|_{L^p_s[t_0,t]}\\
				&\ls 2^{3(1-\f 2r)k} \| \int_{\mathbb{R}}(1+2^k|\tau-s|)^{\f 2r-1}\chi_{[t_0-t\le \tau-s\le t-t_0]}\|\phi(\tau)\|_{L^{r'}{(\R^3)}}\chi_{[t_0\le\tau\le t]}d\tau\|_{L^p_s{[t_0,t]}}\\
				&\ls
				\left\{
				\begin{aligned}
				&2^{3k} \|(1 + 2^k |\cdot|)^{-1}\|_{L^{1}[-t,t]}\|\phi\|_{L^{2}([t_0,t];L^1)},\quad\quad if\ p=2,\ r = \infty,\\
				&2^{3(1-\f 2r)k} \|(1 + 2^k |\cdot|)^{\f 2r-1}\|_{L^{\f{p}2,\infty}(\R)}\|\phi\|_{L^{p'}([t_0,t];L^{r'})},\quad\quad if\ p\in (2,\infty],\ \f1p+\f1r\le \f12,
				\end{aligned}
				\right.\\
				&\ls
				 \left\{
				\begin{aligned}
					&2^{2k}\ln(1+2^kt) ,\quad\quad if\ p = 2,\ r = \infty,\\
					&2^{(3-\f6r-\f2p)k},\quad\quad if\  p\in (2,\infty],\ \f1p+\f1r\le \f12,\\
				\end{aligned}
				\right.
			\end{split}
		\end{equation}
		where $\chi_{[a,b]}$ is the characteristic function of $[a,b]$, $\|(1 + 2^k |\cdot|)^{\f 2r-1}\|_{L^{\f{p}2,\infty}(\R)} = \ds\sup_{\lambda>0} \lambda\cdot|\{x\in\R:(1 + 2^k |x|)^{\f 2r-1}>\lambda\}|^{\f2p}\ls \ds\sup_{\lambda\in(0,1)}\lambda(2^{-k}\lambda^{\f{r}{2-r}})^{\f2p}\ls 2^{-\f2p k}$. Combining \eqref{2.18}--\eqref{2.20} and \eqref{2.22} yields \eqref{ineq:Strichartz'}--\eqref{ineq:strichartz L2Linfty}.
	\end{proof}
	\begin{lemma}[Linear Strichartz estimate II]\label{lem:strichartz L2Linfty2}
		For \( t \geq t_0 \geq 0 \) and integer \( k \geq -1 \), it holds that
		\begin{equation}\label{ineq:strichartz L2Linfty 2}
			\|P_k |\nabla|^{-1}e^{\pm is|\nabla|} f\|_{L^2([t_0,t];L^\infty(\mathbb{R}^3))} \lesssim \ln(e+2^kt)\|P_k f\|_{L^2(\mathbb{R}^3)}.
		\end{equation}
	\end{lemma}
	\begin{proof}
		For $k\ge 0$, \eqref{ineq:strichartz L2Linfty 2} follows directly from \eqref{ineq:strichartz L2Linfty} and Bernstein's inequality.
Next we show \eqref{ineq:strichartz L2Linfty 2} with $k = -1$. Denote
		\begin{equation}\label{2.118}
			\mathcal{B} = \{\phi \in C_0^\infty((t_0,t)\times\R^3): \|\phi\|_{L^{2}_tL^{1}_x}\le 1\}.
		\end{equation}
		Then one has
		\begin{equation}\label{2.119}
			\begin{split}
				&\|P_{-1} |\nabla|^{-1}e^{\pm is|\nabla|} f\|_{L^2([t_0,t];L^\infty)}\\
				&=\sup_{\phi\in \mathcal{B}}\int_{t_0}^{t}\big(P_{-1} |\nabla|^{-1} e^{\pm is|\nabla|} f,\phi(s)\big)_{L^2(\R^3)}ds\\
				&=\sup_{\phi\in \mathcal{B}}\big( P_{-1}f,\int_{t_0}^{t}P_{[[-1]]} |\nabla|^{-1}e^{\mp is|\nabla|}\phi(s)ds\big)_{L^2(\R^3)}\\
				&\le \sup_{\phi\in \mathcal{B}}\|\int_{t_0}^{t}P_{[[-1]]}|\nabla|^{-1} e^{\mp is|\nabla|}\phi(s)ds\|_{L^2(\mathbb{R}^3)}\|P_{-1} f\|_{L^2(\mathbb{R}^3)}.
			\end{split}
		\end{equation}
		Next we treat $\|\int_{t_0}^{t}P_{[[-1]]}|\nabla|^{-1} e^{\mp is|\nabla|}\phi(s)ds\|_{L^2(\mathbb{R}^3)}$.
Applying H\"older's inequality  yields
		\begin{equation}\label{2.120}
			\begin{split}
				&\|\int_{t_0}^{t}P_{[[-1]]}|\nabla|^{-1} e^{\mp is|\nabla|}\phi(s)ds\|_{L^2(\mathbb{R}^3)}^2\\
				&=\big(
				\int_{t_0}^{t}P_{[[-1]]}|\nabla|^{-1} e^{\mp is|\nabla|}\phi(s)ds,
				\int_{t_0}^{t}P_{[[-1]]}|\nabla|^{-1} e^{\mp i\tau|\nabla|}\phi(\tau)d\tau
				\big)_{L^2(\mathbb{R}^3)}\\
				&=
				\int_{t_0}^{t}\int_{t_0}^{t}\big( \phi(s),
				P_{[[-1]]}^2|\nabla|^{-2} e^{\pm i(s-\tau)|\nabla|}\phi(\tau)
				\big)_{L^2(\mathbb{R}^3)}d\tau ds\\
				&\le
				\|\phi(s)\|_{L^{2}([t_0,t];L^{1})}
				\big\|\int_{t_0}^{t}
				\|P_{[[-1]]}^2|\nabla|^{-2} e^{\pm i(s-\tau)|\nabla|} \phi(\tau) \|_{L^{\infty}(\mathbb{R}^3)} d\tau\big\|_{L^2_s[t_0,t]}.
			\end{split}
		\end{equation}
		Hence it follows from \eqref{ineq:linear dispersive 3} that
		\begin{equation}\label{2.121}
			\|P_{[[-1]]}^2|\nabla|^{-2} e^{\pm i(s-\tau)|\nabla|} \phi(\tau) \|_{L^\infty(\mathbb{R}^3)}\lesssim (e + |\tau-s|)^{-1} \ln(e+ |\tau-s|)\|\phi(\tau) \|_{L^{1}(\mathbb{R}^3)}.
		\end{equation}
		Substituting \eqref{2.121} into \eqref{2.120} and applying Young's inequality yield
		\begin{equation}\label{2.122}
			\begin{split}
				&\big\|\int_{t_0}^{t}
				\|P_{[[-1]]}^2 |\nabla|^{-2}e^{\pm i(s-\tau)|\nabla|} \phi(\tau) \|_{L^\infty(\mathbb{R}^3)} d\tau\big\|_{L^2_s[t_0,t]}\\
				&\ls \| \int_{\mathbb{R}}(e+|\tau-s|)^{-1}\ln(e+ |\tau-s|)\chi_{[t_0-t\le \tau-s\le t-t_0]}\|\phi(\tau)\|_{L^{1}{(\R^3)}}\chi_{[t_0\le\tau\le t]}d\tau\|_{L^2_s{[t_0,t]}}\\
				&\ls
				\|(e + |\cdot|)^{-1}\ln(e+ |\cdot|)\|_{L^{1}(-t,t)}\|\phi\|_{L^{2}([t_0,t];L^1)}\\
				&\ls
				\ln^2(e+t).
			\end{split}
		\end{equation}
		Therefore, collecting \eqref{2.118}--\eqref{2.120} and \eqref{2.122} implies \eqref{ineq:strichartz L2Linfty 2} for $k = -1$.
	\end{proof}

	\subsection{The strong Huygens' principle}
	Consider the Cauchy problem of the 3D linear wave equation
	\begin{equation}\label{eq:linear wave}
		\left\{
		\begin{aligned}
			\square u = &F(t,x),\quad (t,x)\in[0,+\infty)\times\R^3,\\
			(u,\p_t&u)(0,x)=(u_{0},u_{1})(x),\quad x\in\R^3.
		\end{aligned}
		\right.
	\end{equation}
    It follows from the spatial Fourier transformation on \eqref{eq:linear wave} and Kirchhoff's formula that
	\begin{equation*}
		\begin{aligned}
			u(t,x) &= \cos(t|\nabla|)u_0(x) + \frac{\sin(t|\nabla|)}{|\nabla|}u_1(x) + \int_0^t \frac{\sin\big((t-s)|\nabla|\big)}{|\nabla|} F(s,x)ds \\
			&=\frac{\partial}{\partial t}\big[ \frac{1}{4\pi t} \int_{\abs{x-y}=t} u_0(y)\dS\big] + \frac{1}{4\pi t} \int_{\abs{x-y}=t} u_1(y)\dS \\
			&\quad + \int_0^t \frac{1}{4\pi (t-s)} \int_{\abs{x-y}=t-s} F(s,y)\dS ds.
		\end{aligned}
	\end{equation*}
	Meanwhile,
	\begin{equation*}
		\begin{aligned}
			\partial_t u(t,x) &= -\sin(t|\nabla|)|\nabla|u_0(x) + \cos(t|\nabla|) u_1(x) + \int_0^t \cos\big((t-s)|\nabla|\big) F(s,x)ds\\
		\end{aligned}
	\end{equation*}
	and
	\begin{equation*}
		\begin{aligned}
			\nabla u(t,x) &=\cos(t|\nabla|)\nabla u_0(x) + \frac{\sin(t|\nabla|)}{|\nabla|}\nabla u_1(x) + \int_0^t \frac{\sin\big((t-s)|\nabla|\big)}{|\nabla|} \nabla F(s,x)ds. \\
		\end{aligned}
	\end{equation*}
	By virtue of the strong Huygens' principle for the 3D linear wave equation, for any fixed point $(t_0,x_0)\in \R_+\times \R^3$,
	\begin{equation}\label{eq:Huygens1}
		\begin{split}
			\frac{\sin(t_0|\nabla|)}{|\nabla|}f(x_0)
			=& \frac{1}{4\pi t_0} \int_{S(x_0,t_0)} f(y)\dS\\
			=&  \frac{1}{4\pi t_0} \int_{S(x_0,t_0)} \Xi_{S(x_0,t_0)}(y)f(y)\dS \\
			=& \frac{\sin(t_0|\nabla|)}{|\nabla|}(\Xi_{S(x_0,t_0)}f)(x_0),\\
		\end{split}
	\end{equation}
	\begin{equation}\label{eq:Huygens2}
		\begin{split}
			\cos(t_0|\nabla|)f(x_0)
			=&\frac{\partial}{\partial t}\big[ \frac{1}{4\pi t} \int_{\abs{x-y}=t} f(y)\dS \big]\big|_{(t,x) = (t_0,x_0)}\\
			=&\frac{\partial}{\partial t}\big[ \frac{t}{4\pi} \int_{\Sph} f(x+tz) \mathrm{d} \sigma_z \big]\big|_{(t,x) = (t_0,x_0)}\\
			=&\frac{1}{4\pi t_0^2} \int_{S(x_0,t_0)} f(y)\dS+\frac{1}{4\pi t_0} \int_{S(x_0,t_0)} \f{\p f}{\p n}(y)\dS\\
			=&	\cos(t_0|\nabla|)(\Xi_{S(x_0,t_0)}f)(x_0), \\
		\end{split}
	\end{equation}
	\begin{equation}\label{eq:Huygens3}
		\begin{split}
			-\sin(t_0|\nabla|)|\nabla|f(x_0)
			=& -\sin(t_0|\nabla|)|\nabla|(\Xi_{S(x_0,t_0)}f)(x_0), \\
		\end{split}
	\end{equation}
	where for $t\ge 1 $ the cutoff function $\Xi_{S(x,t)}$ is defined as
	\begin{equation*}
		\Xi_{S(x,t)} (y) = 	\left\{
		\begin{aligned}
			&1,\quad\quad\quad\quad\quad\quad\quad t-0.01\le|y-x|\le t+0.01,\\
			&\in[0,1],\quad \text{smooth, \quad otherwise},\\
			&0,\quad\quad\quad\quad\quad\quad\quad |y-x|\le t-0.02\quad\text{or}\ |y-x|\ge t+0.02,
		\end{aligned}
		\right.
	\end{equation*}
	and for $t< 1$,
	\begin{equation*}
		\Xi_{S(x,t)} (y) = 	\left\{
		\begin{aligned}
			&1,\quad\quad\quad\quad\quad\quad\quad |y-x|\le t+0.01,\\
			&\in[0,1], \quad \text{smooth, \quad otherwise},\\
			&0,\quad\quad\quad\quad\quad\quad\quad |y-x|\ge t+0.02.
		\end{aligned}
		\right.
	\end{equation*}

	\section{Weighed linear $L^\infty-L^2$ estimates and Strichartz estimates}\label{sec:3}

	In this section, we establish the following weighted Strichartz estimates.
	\begin{theorem}[Weighted Strichartz estimates]\label{thm:Weighted Strichartz estimate}
		For $\beta_1\in(0,1)$, $\beta_2\in(\beta_1,\min\{\f32\beta_1,1\})$, $t\ge t_0\ge 0$ and integer $k\ge -1$,
		\begin{itemize}
			\item[(1)] when $p\in[2,\infty),\ r\in(2,\infty]$ with $\f1p+\f1r=\f12$, one has
			\begin{equation}\label{thm:Weighted Strichartz estimate 1}
				\|(1+s+|x|)^{\f{1}{p}\beta_1}P_{k}e^{\pm is|\nabla|}f\|_{L^p([t_0,t];L^r(\R^3))}
				\ls2^{\f2p(1+\beta_2)k }\|\w{x}^{\f2p \beta_2}P_kf\|_{L^2(\R^3)};
			\end{equation}
			\item[(2)] when $p\in[2,2+2\beta_2-\beta_1),\ r\in(2+\f{4}{2\beta_2-\beta_1},\infty]$ with $\f1p+\f1r=\f12$, one has
			\begin{equation}\label{thm:Weighted Strichartz estimate 2}
				\|(1+s+|x|)^{\f{1}{p}\beta_1}P_{k}|\nabla|^{-1}e^{\pm is|\nabla|}f\|_{L^p([t_0,t];L^r(\R^3))}
				\ls2^{(\f{2+2\beta_2}{p}-1)k }\|\w{x}^{\f2p \beta_2}P_kf\|_{L^2(\R^3)}.
			\end{equation}
		\end{itemize}
	\end{theorem}
	\vskip 0.2 true cm
	The proof procedure of Theorem \ref{thm:Weighted Strichartz estimate} will be divided into three parts, which include Subsections
\ref{sec:Linfty-L^2 estimate}-\ref{Proof-1} below.

	\subsection{Localized $L^\infty-L^2$ estimates}\label{sec:Linfty-L^2 estimate}

	In this subsection, by the strong Huygens' principle,  some useful $L^\infty-L^2$ estimates will be established.
	\begin{lemma}\label{lem:Linfty-L2 estimate 0}
	For any function $f$ with $\supp f \subseteq B(0,R)$ and $R>0$, $t\ge0$,  $k\in \Z$, integer $\iota\in\{0,1\}$,
    it holds that
		\begin{equation}\label{ineq:Linfty-L2 estimate0}
			\|\pk|\nabla|^{-\iota} e^{\pm it|\nabla|}f\|_{L^\infty(\R^3)} \lesssim 2^{(3-\iota)k}(1+2^kt)^{-1}R^{\f32}\|f\|_{L^2(\R^3)}.
		\end{equation}
	\end{lemma}
	\begin{proof}
	\eqref{ineq:Linfty-L2 estimate0} follows directly from Lemma \ref{lem:linear dispersive} and H\"older's inequality.
	\end{proof}
	\begin{lemma}\label{lem:Linfty-L2 estimate}
		For any function $f$ with $\supp f \subseteq B(0,R)$ and $R>0$, it holds that for $t>0$,
		\begin{equation}\label{ineq:Linfty-L2 estimate1}
			\|\f{\sin(t|\nabla|)}{|\nabla|}f\|_{L^{\infty}(\R^3)} \lesssim \min\{1,(1+t)^{-1}R\}\| f\|_{H^{4}(B(0,R))},
		\end{equation}
		\begin{equation}\label{ineq:Linfty-L2 estimate2}
			\|\cos(t|\nabla|)f\|_{L^{\infty}(\R^3)} \lesssim  \min\{1,(1+t)^{-1}R\}\| f\|_{H^{4}(B(0,R))}.
		\end{equation}
	\end{lemma}
	\begin{proof}
		Due to $e^{\pm it|\xi|}|\xi|^{-\iota}\in L^1_{loc}(\R^3)$ for $\iota\in\{0,1\}$, then we have
		\begin{equation*}
			e^{\pm it|\nabla|}|\nabla|^{-\iota}(\Xi_{S(x,t)}f)(x)
			=
			(2\pi)^{-3}\sum_{k\in \Z}\int_{\R^3}K_k^{\iota,M}(t,x-y)(\Id-\Delta_y)^M(\Xi_{S(x,t)}f)(y)dy,
		\end{equation*}		
		where $K^{\iota,M}_k$ is defined in \eqref{def:KklM} .
		Therefore, by Lemma \ref{lem:estimate of KklM}, one can arrive at
		\begin{equation}\label{2.36}
			\begin{split}
				&|e^{\pm it|\nabla|}|\nabla|^{-\iota}(\Xi_{S(x,t)}f)(x)|\\
				&\ls
				\sum_{k\in \Z}\int_{\R^3}|K_k^{\iota,M}(t,x-y)||(\Id-\Delta_y)^M(\Xi_{S(x,t)}f)(y)|dy\\
				&\ls
				\sum_{k\in \Z}2^{(3-\iota)k-2Mk_+}(1+2^kt)^{-1}\|(\Id-\Delta_y)^M(\Xi_{S(x,t)}f)(y)\|_{L^1_y(\R^3)}.
			\end{split}
		\end{equation}
		Note that for $M=2$ and $\iota \in\{0,1\}$,
		\begin{equation*}
			\begin{split}
				\sum_{k<0}2^{(3-\iota)k-2Mk_+}(1+2^kt)^{-1}
				&=\sum_{k<0}2^{(2-\iota)k}(2^{-k}+t)^{-1}
				\ls(1+t)^{-1},\\
				\sum_{k\ge0}2^{(3-\iota)k-2Mk_+}(1+2^kt)^{-1}
				&\ls\sum_{k\ge0}2^{(3-\iota)k-4k}(1+t)^{-1}
				\ls(1+t)^{-1}.
			\end{split}
		\end{equation*}
		Together with \eqref{2.36}, this yields
		\begin{equation}\label{2.37}
			|e^{\pm it|\nabla|}|\nabla|^{-\iota}(\Xi_{S(x,t)}f)(x)|
			\ls (1+t)^{-1}\|(\Id-\Delta_y)^2(\Xi_{S(x,t)}f)(y)\|_{L^1_y(\R^3)}.
		\end{equation}
		Since $\Xi_{S(x,t)}$ is a smooth cutoff with uniformly bounded derivatives (independent of $x$ and $t$),
		it follows from Leibniz's rule and H\"older's inequality that
		\begin{equation}\label{2.38}
			\|(\Id-\Delta_y)^2(\Xi_{S(x,t)}f)(y)\|_{L^1_y(\R^3)} \ls |A(x,t,R)|^{\f12} \|f\|_{H^4(B(0,R))},
		\end{equation}
 		where we have used that the support of $\Xi_{S(x,t)}f$ is contained in $A(x,t,R)=\supp \Xi_{S(x,t)}\cap B(0,R)$.
 		Note that when $t<1$,
 		\begin{equation}\label{CYC-2}
 			|A(x,t,R)|\ls \min\{|\supp \Xi_{S(x,t)}|,|B(0,R)|\} \ls \min\{1,R^3\} \ls \min\{1,R^2\}.
 		\end{equation}
 		When $t \ge 1$, one can obtain
 		\begin{equation}\label{CYC-3}
 			|A(x,t,R)|\ls |\supp \Xi_{S(x,t)}| \ls t^2,
 		\end{equation}
 		\begin{equation}\label{CYC-4}
 			|A(x,t,R)|\ls \int_{t-0.02}^{t+0.02}\mathcal{H}^2(S(x,s)\cap B(0,R))ds \ls \int_{t-0.02}^{t+0.02}4\pi R^2ds \ls R^2,
 		\end{equation}
 		where $\mathcal{H}^2$ denotes the 2D surface measure. Collecting \eqref{CYC-2}--\eqref{CYC-4}, we have
		\begin{equation}\label{YHCCC-100}
			|A(x,t,R)|\ls \min\{(1+t)^2,R^2\}.
		\end{equation}
		Combining \eqref{2.37}--\eqref{YHCCC-100} yields
		\begin{equation}\label{2.39}
			\|e^{\pm it|\nabla|}|\nabla|^{-l}(\Xi_{S(x,t)}f)(x)\|_{L_x^\infty(\R^3)}
			\ls \min\{1,(1+t)^{-1}R\}\|f\|_{H^4(B(0,R))}.
		\end{equation}
		Therefore, by the strong Huygens' principle \eqref{eq:Huygens1}, \eqref{eq:Huygens2} and \eqref{2.39}, one has
		\begin{equation*}
			\|\f{\sin(t|\nabla|)}{|\nabla|}f\|_{L^{\infty}(\R^3)} =\|\f{\sin(t|\nabla|)}{|\nabla|}(\Xi_{S(x,t)}f)(x)\|_{L_x^\infty(\R^3)}  \lesssim \min\{1,(1+t)^{-1}R\}\| f\|_{H^{4}(B(0,R))},
		\end{equation*}
		\begin{equation*}
			\|\cos(t|\nabla|)f\|_{L^{\infty}(\R^3)} =\|\cos(t|\nabla|)(\Xi_{S(x,t)}f)(x)\|_{L_x^\infty(\R^3)}\lesssim  \min\{1,(1+t)^{-1}R\}\| f\|_{H^{4}(B(0,R))}.
		\end{equation*}
		Then the proof of Lemma \ref{lem:Linfty-L2 estimate} is finished.
	\end{proof}
\begin{remark}\label{rmk:3.1}
We emphasize that due to the strong Huygens' principle for the 3D wave equation, the coefficient factor on the right-hand sides of \eqref{ineq:Linfty-L2 estimate1} and  \eqref{ineq:Linfty-L2 estimate2} is $\min\{1, (1+t)^{-1}R\}$
rather than the usual quantity $(1+t)^{-1}R^{3/2}$ in Lemma \ref{lem:Linfty-L2 estimate 0}. This improvement is crucial
in the proof of Theorem \ref{thm:2}. Note that \eqref{ineq:Linfty-L2 estimate1} and  \eqref{ineq:Linfty-L2 estimate2} also improve
such an estimate (A.1) in \cite{Pusateri}
	\[
	\|e^{it|\nabla|}f\|_{L^\infty}\ls \f{1}{t}\big\|\w{x}^{\f{3}{2}^+}|\nabla|^2 f\big\|_{L^2(\R^3)}.
	\]
\end{remark}

		Note that
		\begin{equation*}
			\big[\pk \frac{\sin(t|\nabla|)}{|\nabla|} f\big](x) \neq \big[\pk \frac{\sin(t|\nabla|)}{|\nabla|} \big( \Xi_{S(x,t)} f \big)\big](x),
		\end{equation*}
		\begin{equation*}
			\big[\pk \cos(t|\nabla|) f\big](x) \neq \big[\pk \cos(t|\nabla|) \big( \Xi_{S(x,t)} f \big)\big](x).
		\end{equation*}
		Therefore, the estimates of $\pk \frac{\sin(t|\nabla|)}{|\nabla|} f$ and $\pk \cos(t|\nabla|) f$ should be specially treated rather than derived directly from Lemma \ref{lem:Linfty-L2 estimate}.

	\begin{lemma}\label{lem:local Linfty-L2 estimate1}
		For any function $f$ with $\supp f \subseteq B(0,R)$ and $R>0$, $t>0$, $\delta\in(0,\f12)$, $k\in \Z$, it holds that
		\begin{equation}\label{ineq:local Linfty-L2 estimate1}
			\|\pk\f{\sin(t|\nabla|)}{|\nabla|}f\|_{L^\infty(\R^3)} \ls 2^{\f12k}(1+2^kt)^{-1} \w{2^k R}^{1+\delta}\|f\|_{L^2(\R^3)},
		\end{equation}
		\begin{equation}\label{ineq:local Linfty-L2 estimate2}
			\|\pk\cos(t|\nabla|)f\|_{L^\infty(\R^3)} \ls 2^{\f32k}(1+2^kt)^{-1} \w{2^k R}^{1+\delta}\|f\|_{L^2(\R^3)}.
		\end{equation}
	\end{lemma}
	\begin{proof}
		Owing to Lemma \ref{lem:Linfty-L2 estimate} and Lemma \ref{lem:Qj prop}, for $\delta\in(0,\f12)$, we have
		\begin{equation*}
			\begin{split}
				\|\dot{P}_0\f{\sin(t|\nabla|)}{|\nabla|}f\|_{L^\infty(\R^3)}
				&\ls \sum_{j\ge -1} \|\f{\sin(t|\nabla|)}{|\nabla|}Q_j \dot{P}_0f\|_{L^\infty(\R^3)} \\
				&\ls \sum_{j\ge -1} (1+t)^{-1} 2^j \|Q_j \dot{P}_0f\|_{H^4(\R^3)}\\
				&\ls \sum_{j\ge -1} \sum_{|a+b|\le 4}(1+t)^{-1} 2^j \|\p_x^a(\psi_j) \p_x^b\dot{P}_0f\|_{L^2(\R^3)}\\
				&\ls \sum_{j\ge -1} \sum_{|a+b|\le 4}(1+t)^{-1} 2^{(1-|a|)j} \|\p_x^b\dot{P}_0f\|_{L^2(\supp \psi_j)}\\
				&\ls \sum_{j\ge -1} 2^{-\delta j} \sum_{|b|\le 4}(1+t)^{-1} \|\w{x}^{1+\delta}\p_x^b\dot{P}_0f\|_{L^2(\R^3)}\\
				&\ls \sum_{|b|\le 4}(1+t)^{-1} \|\w{x}^{1+\delta}\p_x^b\dot{P}_0f\|_{L^2(\R^3)}.\\
			\end{split}
		\end{equation*}
		By Lemma \ref{lem:weighted bernstein} and $\w{x}^{2+2\delta}\in A_2(\R^3)$, one has
		\begin{equation*}
			\|\dot{P}_0\f{\sin(t|\nabla|)}{|\nabla|}f\|_{L^\infty(\R^3)}
			\ls
			(1+t)^{-1} \|\w{x}^{1+\delta}f\|_{L^2(\R^3)}
			\ls (1+t)^{-1} \w{R}^{1+\delta}\|f\|_{L^2(\R^3)}.
		\end{equation*}
		Similarly, it holds that
		\begin{equation*}
			\|\dot{P}_0\cos(t|\nabla|)f\|_{L^\infty(\R^3)}
			\ls (1+t)^{-1} \w{R}^{1+\delta}\|f\|_{L^2(\R^3)}.
		\end{equation*}
		Note that
		\begin{equation*}
			\begin{split}
				\|\dot{P}_k\f{\sin(t|\nabla|)}{|\nabla|}f\|_{L^\infty(\R^3)}
				&=\|\dot{P}_0\f{\sin(2^kt|\nabla|)}{2^k|\nabla|}(f(2^{-k}\cdot))\|_{L^\infty(\R^3)} \\
				&\ls 2^{-k}(1+2^kt)^{-1} \w{2^k R}^{1+\delta}\|f(2^{-k}\cdot)\|_{L^2(\R^3)}\\
				&\ls 2^{\f12k}(1+2^kt)^{-1} \w{2^k R}^{1+\delta}\|f\|_{L^2(\R^3)}.
			\end{split}
		\end{equation*}
		Analogously, one can arrive at
		\begin{equation*}
			\|\dot{P}_k\cos(t|\nabla|)f\|_{L^\infty(\R^3)}
			\ls 2^{\f32k}(1+2^kt)^{-1} \w{2^k R}^{1+\delta}\|f\|_{L^2(\R^3)}.
		\end{equation*}
		Then the proof of Lemma \ref{lem:local Linfty-L2 estimate1} is completed.
	\end{proof}
	\begin{cor}\label{cor:local Linfty-L2 estimate2}
		For any function $f$ with $\supp f \subseteq B(0,R)$ and $R>0$, $t>0$, $\delta\in(0,\f12]$, $k\in \Z$, one has
		\begin{equation}\label{YHCCC-50}
			\|\pk\f{\sin(t|\nabla|)}{|\nabla|}f\|_{L^\infty(\R^3)} \ls 2^{(\f32+\delta) k}(1+2^kt)^{-1}  R^{1+\delta}\|f\|_{L^2(\R^3)},
		\end{equation}
		\begin{equation}\label{YHCCC-51}
			\|\pk\cos(t|\nabla|)f\|_{L^\infty(\R^3)} \ls 2^{(\f52+\delta) k}(1+2^kt)^{-1} R^{1+\delta}\|f\|_{L^2(\R^3)}.
		\end{equation}
	\end{cor}
	\begin{proof}
		For $2^k R \le 1$, due to $(2^k R)^{\f32}\le (2^k R)^{1+\delta}$, then Lemma \ref{lem:Linfty-L2 estimate 0} can be applied;
for $2^kR>1$, due to $\w{2^k R}^{1+\delta}\ls (2^k R)^{1+\delta}$, then  Lemma \ref{lem:local Linfty-L2 estimate1} can be used.
Based on this, for $\delta\in(0,\f12]$, one obtains \eqref{YHCCC-50}--\eqref{YHCCC-51}.
\end{proof}
	\begin{cor}\label{cor:local Linfty-L2 estimate3}
		For $t>0$, $\delta\in(0,\f12]$, integer $j\ge -1$ and $k\in \Z$, it holds that
		\begin{equation}\label{ineq:local Linfty-L2 estimate4}
			\|\dot{P}_{k}\f{\sin(t|\nabla|)}{|\nabla|}Q_j f\|_{L^\infty(\R^3)} \lesssim 2^{(\f32+\delta) k}(1+2^kt)^{-1} 2^{(1+\delta)j}\| Q_j f\|_{L^2(\R^3)},
		\end{equation}
		\begin{equation}\label{ineq:local Linfty-L2 estimate5}
			\|\dot{P}_{k}\cos(t|\nabla|)Q_j f\|_{L^\infty(\R^3)} \lesssim   2^{(\f52+\delta) k}(1+2^kt)^{-1} 2^{(1+\delta)j}\| Q_j f\|_{L^2(\R^3)}.
		\end{equation}
		Meanwhile, \eqref{ineq:local Linfty-L2 estimate4} and \eqref{ineq:local Linfty-L2 estimate5} remain valid when $Q_j$ is replaced by $\q_j$ for $j\in\Z$.
	\end{cor}
	\begin{proof}
	Using  $Q_j f$ or $\q_j f$ instead of $f$ in Corollary \ref{cor:local Linfty-L2 estimate2}, we immediately derive
 \eqref{ineq:local Linfty-L2 estimate4}--\eqref{ineq:local Linfty-L2 estimate5}.
	\end{proof}
	\begin{cor}\label{cor:2.10}
		For $t>0$, $\delta\in(0,\f12]$, integers $\iota\in\{0,1\}$ and $j,\ k\ge -1$, it holds that
			\begin{equation}\label{ineq:local Linfty-L2 estimate7}
				\|P_k|\nabla|^{-\iota}e^{\pm it|\nabla|}Q_j f\|_{L^\infty(\R^3)} \lesssim 2^{(\f52+\delta-\iota) k}(1+2^kt)^{-1} 2^{(1+\delta)j}\| Q_j f\|_{L^2(\R^3)}.
			\end{equation}
	In addition, \eqref{ineq:local Linfty-L2 estimate7} remains valid when $P_k$ is replaced by $\pk$ or $Q_j$ is replaced by $\q_j$ for $j,\ k\in\Z$.
	\end{cor}
	\begin{proof}
		It follows from $e^{\pm it|\nabla|} = \cos(t|\nabla|)\pm i|\nabla|\f{\sin(t|\nabla|)}{|\nabla|}$,
	$e^{\pm it|\xi|}|\xi|^{-\iota}\in L^1_{loc}(\R^3)$ for $\iota\in\{0,1\}$, Bernstein's inequality and
Corollary \ref{cor:local Linfty-L2 estimate3} that
		\begin{equation*}
			\begin{split}
				&\|P_{-1}|\nabla|^{-\iota}e^{\pm it|\nabla|}Q_j  f\|_{L^\infty(\R^3)} \\
				&\ls
				\sum_{k\le-1}\|\dot{P}_{k}|\nabla|^{-\iota}(\cos(t|\nabla|)\pm i|\nabla|\f{\sin(t\nabla)}{|\nabla|})Q_j  f\|_{L^\infty(\R^3)} \\
				&\ls
				\sum_{k\le-1}2^{(\f52+\delta-\iota) k}(1+2^kt)^{-1} 2^{(1+\delta)j}\| Q_j f\|_{L^2(\R^3)}\\
				&\ls
				(1+t)^{-1}2^{(1+\delta)j}\| Q_j f\|_{L^2(\R^3)}.
			\end{split}
		\end{equation*}
		The  cases for $k\ge 0$ can be shown analogously or even more easily.
	\end{proof}
	\begin{cor}\label{cor:3.8}
		For $t>0$, $\delta\in(0,\f12]$, $\theta\in[0,1]$, integers $\iota\in\{0,1\}$ and $j,\ k\ge -1$, it holds that
		\begin{equation}\label{ineq:cor 3.8}
			\|P_{k}|\nabla|^{-\iota}e^{\pm it|\nabla|} f\|_{L^\infty(\R^3)}
			\ls
			2^{(\f32-\iota) k+\theta(1+\delta)k}(1+t)^{-\theta}\| \w{x}^{\theta(1+2\delta)} P_kf\|_{L^2(\R^3)}.
		\end{equation}
	\end{cor}
	\begin{proof}
		Due to $e^{\pm it|\xi|}|\xi|^{-\iota}\in L^1_{loc}(\R^3)$ for $\iota\in\{0,1\}$, it follows from Bernstein's inequality that
		\begin{equation*}
			\begin{split}
				\|P_{-1}|\nabla|^{-\iota}e^{\pm it|\nabla|}Q_j  f\|_{L^\infty(\R^3)}
				\lesssim&
				\sum_{k\le0}\|\dot{P}_{k}|\nabla|^{-\iota}e^{\pm it|\nabla|}Q_j  f\|_{L^\infty(\R^3)}\\
				\ls&
				\sum_{k\le0}2^{(\f32-\iota) k}\| Q_j f\|_{L^2(\R^3)}\\
				\ls&
				\| Q_j f\|_{L^2(\R^3)}.
			\end{split}
		\end{equation*}
		Therefore, for integer $k\ge -1$,
		\begin{equation}\label{YHCCC-40}
			\begin{split}
				\|P_{k}|\nabla|^{-\iota}e^{\pm it|\nabla|}Q_j  f\|_{L^\infty(\R^3)}
				\ls
				2^{(\f32-\iota) k}\| Q_j f\|_{L^2(\R^3)}.
			\end{split}
		\end{equation}
		Interpolating \eqref{YHCCC-40} with \eqref{ineq:local Linfty-L2 estimate7} yields that for $\theta\in[0,1]$ and small $\delta>0$,
		\begin{equation}\label{6.32}
			\|P_{k}|\nabla|^{-\iota}e^{\pm it|\nabla|}Q_j  f\|_{L^\infty(\R^3)}
			\ls
			2^{(\f32-\iota) k+\theta(1+\delta)(j+k)}(1+t)^{-\theta}\| Q_j f\|_{L^2(\R^3)}.
		\end{equation}
		By Lemma \ref{lem:Qj prop}, summing over $j$ in \eqref{6.32} yields \eqref{ineq:cor 3.8}.
	\end{proof}
	\begin{lemma}\label{lem:Technical lemma 2}
			For $t>0$, $\delta\in(0,\f14]$, integer $\iota\in\{0,1\}$, $j,\ k,\ l\in\Z$, one has
			\begin{equation}\label{2.72}
				2^{\f12 j}\|\q_j \dot{P}_k |\nabla|^{-\iota}e^{\pm it|\nabla|}\q_l f\|_{L^\infty(\R^3)}
				\ls
				2^{(2-\iota+\delta) k+(1+\delta)l}(1+2^kt)^{-\f12}\|\q_{l} f\|_{L^2(\R^3)}.
			\end{equation}
			Meanwhile, \eqref{2.72} remains valid when $\pk$ is replaced by $P_k$ or $\q_l$ is replaced by $Q_l$ for $k,\ l\ge -1$.
	\end{lemma}
	\begin{proof}
		For $j,\ k,\ l\in\Z$, one has
		\begin{equation*}
			\begin{split}
				\q_j\dot{P}_k |\nabla|^{-\iota}e^{\pm it|\nabla|}\q_l f(x)
				=&(2\pi)^{-3}\int_{\R^3}\dot{\psi}_j(x)K^\iota_k(t,x-y)\dot{\psi}_{[[l]]}(y)\q_lf(y)dy\\
				=&(2\pi)^{-3} I_{j,k,l}^\iota(t,x),
			\end{split}
		\end{equation*}
		where $K^\iota_k(t,x) = K^{\iota,0}_k(t,x)$ for $k \in \Z$, $K^{\iota,0}_k$ is defined in \eqref{def:KklM}.
We only treat the case of $k = 0$ since the general case of $k\not=0$ follows by a standard rescaling argument.
\vskip 0.1 true cm

	\noindent\textbf{Case 1.} $j\ge \log_2t+100$ and $l\le j-3$

\vskip 0.1 true cm

		In this case, for $x\in\supp \dot{\psi}_j$ and $y\in\supp \dot{\psi}_{[[l]]}$, one has
		\begin{equation*}
			|x-y|\ge|x|-|y|\ge {\f{5}{8}}\cdot2^j-\f{16}{5}\cdot2^l\ge2^{j-4}>2t,
		\end{equation*}
		where $\supp\dot{\psi}_0\subseteq [\f{5}{8},\f{8}{5}]$ and $\supp\dot{\psi}_{[[0]]}\subseteq[\f{5}{16},\f{16}{5}]$ are used.
Owing to Lemma \ref{lem:estimate of KklM} (4) with $M=0$ and interpolation method,
one can derive that for $k=0$,
		\begin{equation*}
			|K^\iota_0(t,x-y)|\ls |x-y|^{-\f32+\delta}\ls 2^{(-\f32+\delta)j}.
		\end{equation*}
		Therefore, by H\"{o}lder's inequality and the assumption of $l\le j-3$, we have
		\begin{equation}\label{2.76}
			\begin{split}
				2^{j}\|I_{j,0,l}^\iota(t,x)\|_{L^\infty(\R^3)}
				\ls&
				2^{j}\|\dot{\psi}_j(x)K^\iota_0(t,x-y)\dot{\psi}_{[[l]]}(y)\|_{L^\infty_xL^2_y}\|\q_lf\|_{L^2(\R^3)}\\
				\ls&
				2^{(-\f12+\delta)j+\f32 l}\|\q_lf\|_{L^2(\R^3)}\\
				\ls&
				2^{(1+\delta) l}\|\q_lf\|_{L^2(\R^3)}.
			\end{split}
		\end{equation}
		On the other hand, by Corollary \ref{cor:2.10}, one has
		\begin{equation}\label{2.77}
			\|I_{j,0,l}^\iota(t,x)\|_{L^\infty(\R^3)}
			\ls
			\|\dot{P}_{0}|\nabla|^{-\iota} e^{\pm it|\nabla|}\q_lf\|_{L^\infty(\R^3)}
			\ls
			2^{ (1+\delta) l}(1+t)^{-1}\|\q_lf\|_{L^2(\R^3)}.
		\end{equation}
		Interpolating \eqref{2.77} with \eqref{2.76} yields
		\begin{equation}\label{2.78}
			2^{\f12 j}\|I_{j,0,l}^\iota(t,x)\|_{L^\infty(\R^3)}\ls  2^{ (1+\delta)l}(1+t)^{-\f12}\|\q_lf\|_{L^2(\R^3)}.
		\end{equation}
\vskip 0.1 true cm
		\noindent\textbf{Case 2.} $j\le \log_2t+100$ and $l\le j-3$
\vskip 0.1 true cm
		In this case, by Corollary \ref{cor:2.10}, one can obtain
		\begin{equation}\label{2.79}
			\begin{split}
				2^{\f12 j}\|I_{j,0,l}^\iota(t,x)\|_{L^\infty(\R^3)}
				\ls&	2^{\f12 j}\|\dot{P}_0 |\nabla|^{-\iota} e^{\pm it|\nabla|} \q_lf\|_{L^\infty(\R^3)}\\
				\ls& (1+ t)^{\f12}2^{(1+\delta)l}(1+ t)^{-1}\|\q_lf\|_{L^2(\R^3)}\\
				\ls &2^{(1+\delta)l}(1+ t)^{-\f12}\|\q_lf\|_{L^2(\R^3)}.
			\end{split}
		\end{equation}
\vskip 0.1 true cm

	\noindent	\textbf{Case 3.} $l\ge j-3$

\vskip 0.1 true cm
		It follows from Corollary \ref{cor:2.10} with $2\delta$ instead of $\delta$ that
		\begin{equation}\label{2.80}
			2^{j}\|I_{j,0,l}^\iota(t,x)\|_{L^\infty(\R^3)}
			\ls	2^{j}\|\dot{P}_0 |\nabla|^{-\iota} e^{\pm it|\nabla|} \q_lf\|_{L^\infty(\R^3)}
			\ls 2^{(2+2\delta)l }(1+t)^{-1}\|\q_lf\|_{L^2(\R^3)}.
		\end{equation}
		On the other hand, it follows from Bernstein's inequality that
		\begin{equation}\label{2.81}
			\|I_{j,0,l}^\iota(t,x)\|_{L^\infty(\R^3)}\ls	\|\dot{P}_0 |\nabla|^{-\iota} e^{\pm it|\nabla|} \q_lf\|_{L^\infty(\R^3)} \ls \|\q_lf\|_{L^2(\R^3)}.
		\end{equation}
		Interpolating \eqref{2.80} with \eqref{2.81} yields
		\begin{equation}\label{2.82}
			2^{\f12 j}\|I_{j,0,l}^\iota(t,x)\|_{L^\infty(\R^3)}
			\ls
			2^{(1+\delta)l}(1+ t)^{-\f12}\|\q_lf\|_{L^2(\R^3)}.
		\end{equation}
		Collecting \eqref{2.78}, \eqref{2.79} and \eqref{2.82}, one has that for $j,\ l\in \Z$,
		\begin{equation*}
			2^{\f12 j}\|\q_j \dot{P}_0 |\nabla|^{-\iota}e^{\pm it|\nabla|}\q_l f\|_{L^\infty(\R^3)}
			\ls
			2^{(1+\delta)l}(1+t)^{-\f12}\|\q_l f\|_{L^2(\R^3)}.
		\end{equation*}
		For $k\not=0$, set $f_k(x) = f(2^{-k}x)$. Then
		\begin{equation*}
			\begin{split}
				&2^{\f12 j}\|\q_j \dot{P}_k |\nabla|^{-\iota}e^{\pm it|\nabla|}\q_l f\|_{L^\infty(\R^3)}\\
				&=(2\pi)^{-3}\cdot2^{\f12 j}\|\int_{\R^3}\dot{\psi}_j(x)K^\iota_k(t,x-y)\dot{\psi}_{l}(y)f(y)dy\|_{L^\infty(\R^3)}\\
				&=(2\pi)^{-3}\cdot2^{-\iota k+\f12 j}\|\int_{\R^3}\dot{\psi}_j(x)K^\iota_0(2^kt,2^k(x-y))\dot{\psi}_{l}(y)f(y)d(2^ky)\|_{L^\infty(\R^3)}\\
				&=2^{-(\iota+\f12) k}\cdot2^{\f12(j+k)}\|\q_{j+k}\dot{P}_0|\nabla|^{-\iota}e^{\pm i2^kt|\nabla|}\q_{l+k}f_k\|_{L^\infty(\R^3)} \\
				&\ls 2^{-(\iota+\f12) k}\cdot2^{(1+\delta)(l+k)}(1+2^kt)^{-\f12}\|\q_{l+k} f_k\|_{L^2(\R^3)}\\
				&\ls 2^{(2-\iota+\delta) k+(1+\delta)l}(1+2^kt)^{-\f12}\|\q_{l} f\|_{L^2(\R^3)}.
			\end{split}
		\end{equation*}
		Note that $e^{\pm it|\xi|}|\xi|^{-\iota}\in L^1_{loc}(\R^3)$ for $\iota\in\{0,1\}$. Summing over $k, l \le -1$ yields the corresponding inhomogeneous estimates. Then the proof of
Lemma \ref{lem:Technical lemma 2} is completed.
	\end{proof}

	\subsection{Weighted Strichartz estimates}\label{sec:Weighted Strichartz estimate}

	To prove Theorem \ref{thm:Weighted Strichartz estimate}, we now establish a series of localized and weighted Strichartz estimates
for the operators $P_{k}e^{\pm it|\nabla|}Q_jf$ and $P_{k}|\nabla|^{-1}e^{\pm it|\nabla|}Q_jf$, which include the endpoint
and non-endpoint cases, respectively.

\subsubsection{The endpoint case}

	\begin{lemma}\label{lem:Localized Strichartz estimate}
		For $\beta_1\in(0,1)$, $\beta_2\in(\beta_1,\min\{\f32\beta_1,1\})$, $t\ge t_0\ge 0$,
integers $\iota\in\{0,1\}$ and $j,\ k\ge -1$, one has
		\begin{equation}\label{local:Stric}
			\|(1+2^ks)^{\beta_1/2}P_k|\nabla|^{-\iota} e^{\pm is|\nabla|}Q_jf\|_{L^2([t_0,t];L^\infty(\R^3))}
			\ls2^{(1-\iota+\beta_2) k+\beta_2j}\|Q_jf\|_{L^2(\R^3)}.
		\end{equation}
	\end{lemma}
	\begin{proof}
		\eqref{ineq:local Linfty-L2 estimate7} ensures that for $p\in(2,\infty)$, $\delta=\f{\beta_2/\beta_1-1}{2}\in(0,\f14]$,
integers $j,\ k\ge -1$,
		\begin{equation}\label{2.65}
			\begin{aligned}
				&\|(1+2^ks)^{\f12}P_k|\nabla|^{-\iota}e^{\pm is|\nabla|}Q_jf\|_{L^p([t_0,t];L^\infty)}\\
				&\ls
				2^{(\f52+\delta-\iota) k+(1+\delta)j}\|(1+2^ks)^{-\f12}\|_{L^p[t_0,t]}\| Q_j f\|_{L^2(\R^3)}\\
				&\ls
				2^{(\f52+\delta-\iota-\f1p) k+(1+\delta)j}\|Q_jf\|_{L^2(\R^3)}.
			\end{aligned}
		\end{equation}
		Owing to Bernstein's inequality and \eqref{ineq:Strichartz'}, we have that for $k\in\Z$,
		\begin{equation*}
			\|\pk |\nabla|^{-\iota} e^{\pm is|\nabla|}Q_jf\|_{L^p([t_0,t];L^\infty)}
			\ls{2^{(\f32-\iota-\f1p) k}}\| Q_jf\|_{L^2(\R^3)}.
		\end{equation*}
		Summing over $k\le-1$ yields that for $k\ge -1$,
		\begin{equation}\label{2.66}
			\|P_k|\nabla|^{-\iota} e^{\pm is|\nabla|}Q_jf\|_{L^p([t_0,t];L^\infty)}
			\ls{2^{(\f32-\iota-\f1p) k}}\|Q_jf\|_{L^2(\R^3)}.
		\end{equation}
Interpolating \eqref{2.66} with \eqref{2.65} yields
\begin{equation}\label{YHCCC-30}
\begin{split}
&\|(1+2^ks)^{\beta_3/2}P_k|\nabla|^{-\iota}e^{\pm is|\nabla|}Q_jf\|_{L^p([t_0,t];L^\infty)}\\
&\ls{2^{[(\f52+\delta-\iota-\f1p) k+(1+\delta)j]\beta_3 +(\f32-\iota-\f1p)k(1-\beta_3)}}\|Q_jf\|_{L^2(\R^3)},
\end{split}
\end{equation}
where $\beta_3 = \f{\beta_2}{1+\delta} = \f{2\beta_1\beta_2}{\beta_1+\beta_2}\in(\beta_1,\beta_2)\subset(0,1)$.
Choosing $p=\frac{2}{1-(\beta_3-\beta_1)/2}\in(2,\frac{2}{1-(\beta_3-\beta_1)})$ in \eqref{YHCCC-30},
by H\"{o}lder's inequality, then we have
		\begin{equation*}
			\begin{split}
				&\|(1+2^ks)^{\beta_1/2}P_k|\nabla|^{-\iota}e^{\pm is|\nabla|}Q_jf\|_{L^2([t_0,t];L^\infty)}\\
				&\ls\|(1+2^ks)^{(\beta_1-\beta_3)/2}\|_{L^\frac{2p}{p-2}([t_0,t])}
				\|(1+2^ks)^{\beta_3/2}P_k|\nabla|^{-\iota}e^{\pm is|\nabla|}Q_jf\|_{L^p([t_0,t];L^\infty)}\\
				&\ls2^{(1-\iota+(1+\delta)\beta_3) k+(1+\delta)\beta_3j}\|Q_jf\|_{L^2(\R^3)}\\
				&\ls2^{(1-\iota+\beta_2) k+\beta_2j}\|Q_jf\|_{L^2(\R^3)},
			\end{split}
		\end{equation*}
		which leads to \eqref{local:Stric}.
	\end{proof}
	\begin{lemma}\label{lem:Localized Strichartz estimate 2}
		For $\beta_1\in(0,1)$, $\beta_2\in(\beta_1,\min\{\f32\beta_1,1\})$, $t\ge t_0\ge 0$, $j\in\Z$, integers $\iota\in\{0,1\}$ and $k,\ l\ge -1$, one has
		\begin{equation}\label{local:Stric 3}
			2^{\f12 \beta_1 (j+k)}\|\q_jP_{k}|\nabla|^{-\iota}e^{\pm is|\nabla|}Q_lf\|_{L^2([t_0,t];L^\infty(\R^3))}
			\ls2^{(1-\iota+\beta_2) k+\beta_2l}\|Q_lf\|_{L^2(\R^3)}.
		\end{equation}
	\end{lemma}
	\begin{proof}
		For $\delta=\f{\beta_2/\beta_1-1}{2}\in(0,\f14]$,  $j\in\Z$, integers $k,\ l\ge -1$, interpolating \eqref{2.72}
with \eqref{ineq:local Linfty-L2 estimate7} yields
		\begin{equation}\label{2.84}
         \begin{split}
		&2^{\f{\theta}{2}(j+k)}(1+2^ks)^{\f{1-\theta}{2}}\|\q_jP_{k}|\nabla|^{-\iota}e^{\pm is|\nabla|}Q_lf\|_{L^\infty(\R^3)}\\
		&\ls 2^{(\f52+\delta-\iota)k+(1+\delta)l}(1+2^ks)^{-\f12}\|Q_l f\|_{L^2(\R^3)},
        \end{split}
		\end{equation}
		where $\theta\in[0,1]$. This implies that for any $p\in(2,\infty)$,
		\begin{equation*}
			\|2^{\f{\theta}{2}(j+k)}(1+2^ks)^{\f{1-\theta}{2}}\q_jP_{k}|\nabla|^{-\iota}e^{\pm is|\nabla|}Q_lf\|_{L^p([t_0,t];L^\infty)}
			\ls	
			2^{(\f52+\delta-\iota-\f1p) k+(1+\delta)l}\|Q_lf\|_{L^2(\R^3)}.
		\end{equation*}
This, together with \eqref{2.66}, yields
		\begin{equation*}
			\begin{split}
				&\|2^{\f{\theta\beta_3}{2}(j+k)}(1+2^ks)^{\f{(1-\theta)\beta_3}{2}}\q_jP_{k}|\nabla|^{-\iota}e^{\pm is|\nabla|}Q_lf\|_{L^p([t_0,t];L^\infty)}\\
				&\ls 2^{[(\f52+\delta-\iota-\f1p) k+(1+\delta)l]\beta_3 +(\f32-\iota-\f1p)k(1-\beta_3)}\|Q_lf\|_{L^2(\R^3)},
			\end{split}
		\end{equation*}
		where $\beta_3 = \f{\beta_2}{1+\delta} = \f{2\beta_1\beta_2}{\beta_1+\beta_2}\in(\beta_1,\beta_2)\subset(0,1)$. Choosing $\theta = \f{\beta_1}{\beta_3}$ and $p=\frac{2}{1-(\beta_3-\beta_1)/2}\in(2,\frac{2}{1-(\beta_3-\beta_1)})$,
		then it follows from H\"{o}lder's inequality that
		\begin{equation*}
			\begin{split}
				&2^{\f{\beta_1}{2}(j+k)}\|\q_jP_{k}|\nabla|^{-\iota}e^{\pm is|\nabla|}Q_lf\|_{L^2([t_0,t];L^\infty)}\\
				&\ls
				\|(1+2^ks)^{(\beta_1-\beta_3)/2}\|_{L^\frac{2p}{p-2}[t_0,t]}
				\|2^{\f{\beta_1}{2}(j+k)}(1+2^ks)^{(\beta_3-\beta_1)/2}\q_jP_{k}|\nabla|^{-\iota}e^{\pm is|\nabla|}Q_lf\|_{L^p([t_0,t];L^\infty)}\\
				&\ls2^{(1-\iota+\beta_2) k+\beta_2l}\|Q_lf\|_{L^2(\R^3)}.
			\end{split}
		\end{equation*}
		This leads to \eqref{local:Stric 3}.
	\end{proof}
	\begin{cor}\label{cor:Localized Strichartz estimate 2}
		For $\beta_1\in(0,1)$, $\beta_2\in(\beta_1,\min\{\f32\beta_1,1\})$, $t\ge t_0\ge 0$,
integers $\iota\in\{0,1\}$ and $k\ge -1$, one has
			\begin{equation}\label{cor:local:Stric 3}
				\sum_{j\ge-1}\|[1+2^k(s+|x|)]^{\f{\beta_1}{2}}P_{k}|\nabla|^{-\iota}e^{\pm is|\nabla|}Q_jf\|_{L^2([t_0,t];L^\infty(\R^3))}
				\ls	2^{(1-\iota+\beta_2) k}\| \w{x}^{\beta_2} f\|_{L^2(\R^3)}.
		\end{equation}
	\end{cor}
	\begin{proof}
		It follows from Lemma \ref{lem:Localized Strichartz estimate 2} that
		\begin{equation*}
			\begin{split}
				&\sum_{l\ge-1}\||2^kx|^{\f{\beta_1}{2}}P_{k}|\nabla|^{-\iota}e^{\pm is|\nabla|}Q_lf\|_{L^2([t_0,t];L^\infty)}\\
				&\ls
				\sum_{j\in\Z,l\ge-1}2^{\f12 \beta_1 (j+k)}\|\q_jP_{k}|\nabla|^{-\iota}e^{\pm is|\nabla|}Q_lf\|_{L^2([t_0,t];L^\infty)}\\
				&\ls
				\sum_{l\ge-1}2^{(1-\iota+\f{\beta_1+\beta_2}{2}) k+\f{\beta_1+\beta_2}{2}l}\|Q_lf\|_{L^2(\R^3)}\\
				&\ls
				2^{(1-\iota+\beta_2 ) k}\| \w{x}^{\beta_2} f\|_{L^2(\R^3)}.
			\end{split}
		\end{equation*}
		Together with Lemma \ref{lem:Localized Strichartz estimate} and
$(1+2^ks)^{\f{\beta_1}{2}}+|2^kx|^{\f{\beta_1}{2}}\approx(1+2^k(s+|x|))^{\f{\beta_1}{2}}$, this yields \eqref{cor:local:Stric 3}.
	\end{proof}

	\subsubsection{The non-endpoint case for $e^{\pm it|\nabla|}$}

	\begin{lemma}\label{lem:Localized Strichartz estimate L4}
		For $\beta_1\in(0,1)$, $\beta_2\in(\beta_1,\min\{\f32\beta_1,1\})$, $p,\ r\in(2,\infty)$ with $\f1p+\f1r=\f12$,
$t\ge t_0\ge 0$ and integers $j,\ k\ge -1$, one has
		\begin{equation}\label{local:Stric L4}
			\|(1+2^ks)^{\f1p \beta_1}P_{k}e^{\pm is|\nabla|}Q_jf\|_{L^p([t_0,t];L^r(\R^3))}
			\ls 2^{\f2p(1+\beta_2)k +\f2p\beta_2 j}\|Q_jf\|_{L^2(\R^2)}.
		\end{equation}
		Moreover, \eqref{local:Stric L4} remains valid when $P_k$ is replaced by $\pk$ for $k\in\Z$.
	\end{lemma}
	\begin{proof}
		\eqref{ineq:local Linfty-L2 estimate7} ensures that for any $\alpha\in(0,1)$ and $\delta\in(0,\f14]$,
		\begin{equation}\label{YHCCC-131}
			\begin{split}
				&\|(1+2^ks)^{\f\alpha p}P_{k}e^{\pm is|\nabla|}Q_jf\|_{L^p([t_0,t];L^r)}\\
				&\ls
				\|(1+2^ks)^{\f\alpha p}\|P_{k}e^{\pm is|\nabla|}Q_jf\|_{L^2(\R^3)}^{\f2r}\|P_{k}e^{\pm is|\nabla|}Q_jf\|_{L^\infty(\R^3)}^{1-\f2r}\|_{L^p[t_0,t]}\\
				&\ls
				2^{\f2p[(\f52+\delta)k+(1+\delta)j]}\|(1+2^ks)^{\f{\alpha-2}{p} }\|_{L^p[t_0,t]}\|Q_jf\|_{L^2(\R^3)}\\
				&\ls
				2^{\f2p[(2+\delta)k+(1+\delta)j]}\|Q_jf\|_{L^2(\R^3)}.
			\end{split}
		\end{equation}
		Interpolating this inequality with \eqref{ineq:Strichartz'} yields that for $\theta\in[0,1]$,
	\begin{equation}\label{YHCCC-31}
        \begin{split}
			\|(1+2^ks)^{\theta\f\alpha p}P_{k}e^{\pm is|\nabla|}Q_jf\|_{L^p([t_0,t];L^r)}
			\ls 2^{\f2p(1+(1+\delta)\theta)k+\f2p\theta(1+\delta) j}\|Q_jf\|_{L^2(\R^2)}.
         \end{split}
		\end{equation}
Choosing $\delta=\f{\beta_2/\beta_1-1}{2}\in(0,\f14]$, $\theta = \f{\beta_2}{1+\delta} = \f{2\beta_1\beta_2}{\beta_1+\beta_2}\in(\beta_1,\beta_2)\subset(0,1)$ and $\alpha=\f{\beta_1}{\theta}=\f{\beta_1+\beta_2}{2\beta_2}\in(0,1)$
in \eqref{YHCCC-31}, then \eqref{local:Stric L4} is shown. Note that the estimates \eqref{YHCCC-131} and \eqref{YHCCC-31}
remain valid when $P_k$ is replaced by $\pk$ for all $k\in\mathbb{Z}$, hence the lemma holds.
	\end{proof}
	\begin{lemma}\label{lem:Localized Strichartz estimate 4}
For $\beta_1\in(0,1)$, $\beta_2\in(\beta_1,\min\{\f32\beta_1,1\})$, $p,\ r\in(2,\infty)$ with $\f1p+\f1r=\f12$,
$t\ge t_0\ge 0$,  $j\in\Z$ and integers $k,\ l\ge -1$, one has
		\begin{equation}\label{local:Stric 4}
			2^{\f{\beta_1}{p}(j+k)}\|\q_jP_{k}e^{\pm is|\nabla|}Q_lf\|_{L^p([t_0,t];L^r(\R^3))}
			\ls
			2^{\f2p(1+\beta_2)k +  \f2p\beta_2 l}\|Q_lf\|_{L^2(\R^2)}.
		\end{equation}
	\end{lemma}
	\begin{proof}
		\eqref{2.84} with $\theta={\alpha}$ and $\alpha\in(0,1)$ ensures that
		\begin{equation*}
			\begin{split}
				&\|2^{\f{\alpha}{p}(j+k)}\q_jP_{k}e^{\pm is|\nabla|}Q_lf\|_{L^p([t_0,t];L^r)}\\
				&\ls
				\big\|
				\|P_{k}e^{\pm is|\nabla|}Q_lf\|_{L^2(\R^3)}^{\f2r}
				\cdot \big(2^{\f{\alpha}{{2}}(j+k)}\|\q_jP_{k}e^{\pm is|\nabla|}Q_lf\|_{L^\infty(\R^3)}\big)^{\f2p}
				\big\|_{L^p[t_0,t]}\\
				&\ls
				2^{\f2p[(\f52+\delta)k+(1+\delta)j]}\|(1+2^ks)^{(-1+\f\alpha 2)\f2p}\|_{L^p[t_0,t]}\|Q_lf\|_{L^2(\R^3)}\\
				&\ls
				2^{\f2p[(2+\delta)k+(1+\delta)j]}\|Q_lf\|_{L^2(\R^3)}.
			\end{split}
		\end{equation*}
		Interpolating this inequality with \eqref{ineq:Strichartz'} yields
		\begin{equation}\label{YHCCC-32}
			2^{\f{\alpha\theta}{p}(j+k)}\|\q_jP_{k}e^{\pm is|\nabla|}Q_lf\|_{L^p([t_0,t];L^r)}
			\ls 2^{\f2p(1+(1+\delta)\theta)k +  \f2p\theta(1+\delta) l}\|Q_lf\|_{L^2(\R^2)}.
		\end{equation}
		The choices of $\delta=\f{\beta_2/\beta_1-1}{2}\in(0,\f14]$, $\theta = \f{\beta_2}{1+\delta} = \f{2\beta_1\beta_2}{\beta_1+\beta_2}\in(\beta_1,\beta_2)\subset(0,1)$ and $\alpha = \f{\beta_1}{\theta} = \f{\beta_1+\beta_2}{2\beta_2}\in(0,1)$
 in \eqref{YHCCC-32} lead to \eqref{local:Stric 4}.
	\end{proof}
	\begin{cor}\label{cor:Localized Strichartz estimate 4}
		For $\beta_1\in(0,1)$, $\beta_2\in(\beta_1,\min\{\f32\beta_1,1\})$, $p,\ r\in(2,\infty)$ with $\f1p+\f1r=\f12$,
$t\ge t_0\ge 0$ and integer $k\ge -1$, it holds that
			\begin{equation}\label{cor:local:Stric 4}
				\sum_{j\ge-1}\|[1+2^k(s+|x|)]^{\f{1}{p}\beta_1}P_{k}e^{\pm is|\nabla|}Q_jf\|_{L^p([t_0,t];L^r(\R^3))}
				\ls 2^{\f2p(1+\beta_2)k }\|\w{x}^{\f{2}{p}\beta_2}f\|_{L^2(\R^3)}.
		\end{equation}
	\end{cor}
	\begin{proof}
		By Lemma \ref{lem:Localized Strichartz estimate 4}, one has
		\begin{equation*}
			\begin{split}
				&\sum_{l\ge-1}\||2^kx|^{\f{\beta_1}{p}}P_{k}e^{\pm is|\nabla|}Q_lf\|_{L^p([t_0,t];L^r)}\\
				&\ls
				\sum_{j\in\Z,l\ge-1}2^{\f{\beta_1}{p} (j+k)}\|\q_jP_{k}e^{\pm is|\nabla|}Q_lf\|_{L^p([t_0,t];L^r)}\\
				&\ls
				\sum_{l\ge-1}2^{\f2p(1+\beta_2)k +  \f2p\f{\beta_1+\beta_2}{2} l}\|Q_lf\|_{L^2(\R^3)}\\
				&\ls
				2^{\f2p(1+\beta_2)k }\|\w{x}^{\f{2}{p}\beta_2}f\|_{L^2(\R^3)}.
			\end{split}
		\end{equation*}
		This, together with Lemma \ref{lem:Localized Strichartz estimate L4} and
$(1+2^ks)^{\f{\beta_1}{p}}+|2^kx|^{\f{\beta_1}{p}}\approx(1+2^k(s+|x|))^{\f{\beta_1}{p}}$,  yields \eqref{cor:local:Stric 4}.
	\end{proof}

	\subsubsection{The non-endpoint case for $|\nabla|^{-1}e^{\pm it|\nabla|}$}

	\begin{lemma}\label{lem:Localized Strichartz estimate 5}
		For $\beta_1\in(0,1)$, $\beta_2\in(\beta_1,\min\{\f32\beta_1,1\})$, $p\in(2,2+2\beta_2-\beta_1),\ r\in(2+\f{4}{2\beta_2-\beta_1},\infty)$ with $\f1p+\f1r=\f12$, $t\ge t_0\ge 0$ and integers $j,\ k\ge -1$, it holds that
		\begin{equation}\label{local:Stric 5}
			\|(1+2^ks)^{\f1p \beta_1}P_k |\nabla|^{-1}e^{\pm is|\nabla|}Q_jf\|_{L^p([t_0,t];L^r(\R^3))}
			\ls
			2^{(\f{2+2\beta_2}{p}-1)k +\f2p\beta_2 j}\|Q_jf\|_{L^2(\R^2)}.
		\end{equation}
	\end{lemma}
	\begin{proof}
		By \eqref{local:Stric L4} with $\pk$ instead of $P_k$ for $k\in\Z$ and Bernstein's inequality, we have
		\begin{equation}\label{2.164}
			\begin{split}
				&\|(1+2^ks)^{\f1p \beta_1}\pk |\nabla|^{-1}e^{\pm is|\nabla|}Q_jf\|_{L^p([t_0,t];L^r)}\\
				&\ls
				2^{-k}\|(1+2^ks)^{\f1p \beta_1} \pk e^{\pm is|\nabla|}Q_jf\|_{L^p([t_0,t];L^r)}\\
				&\ls
				2^{\f2p(1+\beta_2)k-k +\f2p\beta_2 j}\|Q_jf\|_{L^2(\R^2)}.
			\end{split}
		\end{equation}
		When $k\le-1$, owing to $2^k(1+s)\le 1+2^ks$, one can arrive at
		\begin{equation*}
			\|(1+s)^{\f1p \beta_1}\pk |\nabla|^{-1}e^{\pm is|\nabla|}Q_jf\|_{L^p([t_0,t];L^r)}
			\ls
			2^{\f1p(2+2\beta_2-\beta_1-p)k+\f2p\beta_2 j}\|Q_jf\|_{L^2(\R^2)}.
		\end{equation*}
		Note that $e^{\pm it|\xi|}|\xi|^{-1}\in L^1_{loc}(\R^3)$. Summing over $k\le -1$ leads to
		\begin{equation}\label{2.165}
			\|(1+s)^{\f1p \beta_1}P_{-1}|\nabla|^{-1}e^{\pm is|\nabla|}Q_jf\|_{L^p([t_0,t];L^r)}
			\ls
			2^{\f2p\beta_2 j}\|Q_jf\|_{L^2(\R^2)}.
		\end{equation}
		Combining \eqref{2.164} and \eqref{2.165} yields Lemma \ref{lem:Localized Strichartz estimate 5}.
	\end{proof}
	\begin{lemma}\label{lem:Localized Strichartz estimate 6}
		For $\beta_1\in(0,1)$, $\beta_2\in(\beta_1,\min\{\f32\beta_1,1\})$, $p\in(2,2+2\beta_2-\beta_1),\ r\in(2+\f{4}{2\beta_2-\beta_1},\infty)$
with $\f1p+\f1r=\f12$, $t\ge t_0\ge 0$, $j\in\Z$, and integers $k,\ l\ge -1$, one has
		\begin{equation}\label{local:Stric 6}
			2^{\f{\beta_1}{p}(j+k)}\|\q_jP_{k}|\nabla|^{-1}e^{\pm is|\nabla|}Q_lf\|_{L^p([t_0,t];L^r(\R^3))}
			\ls
			2^{(\f{2+2\beta_2}{p}-1)k +\f2p\beta_2 l}\|Q_lf\|_{L^2(\R^2)}.
		\end{equation}
	\end{lemma}
	\begin{proof}
		Analogously to \eqref{2.84}, interpolating \eqref{2.72} with \eqref{ineq:local Linfty-L2 estimate7} yields
that for $\delta=\f{\beta_2/\beta_1-1}{2}\in(0,\f14]$, $k\in\Z$, integers $j, l\ge -1$,
		\begin{equation*}
			2^{\f{\alpha}{2}(j+k)}(1+2^ks)^{\f{1-\alpha}{2}}\|\q_j\pk|\nabla|^{-1}e^{\pm is|\nabla|}Q_lf\|_{L^\infty(\R^3)}
			\ls
			2^{(\f32+\delta)k+(1+\delta)l}(1+2^ks)^{-\f12}\|Q_l f\|_{L^2(\R^3)},
		\end{equation*}
		where $\alpha\in(0,1)$. Following a similar argument as in the proof of Lemma \ref{lem:Localized Strichartz estimate 4}, one has
		\begin{equation}\label{2.166}
			\begin{split}
				\|2^{\f1p \beta_1(j+k)}\q_j\pk |\nabla|^{-1}e^{\pm is|\nabla|}Q_lf\|_{L^p([t_0,t];L^r)}
				\ls
				2^{\f2p(1+\beta_2)k-k +\f2p\beta_2 l}\|Q_lf\|_{L^2(\R^2)}.
			\end{split}
		\end{equation}
		Note that $e^{\pm it|\xi|}|\xi|^{-1}\in L^1_{loc}(\R^3)$. Summing over $k\le -1$ for \eqref{2.166} leads to
		\begin{equation}\label{2.167}
			2^{\f1p \beta_1j}\|\q_jP_{-1}|\nabla|^{-1}e^{\pm is|\nabla|}Q_lf\|_{L^p([t_0,t];L^r)}
			\ls
			2^{\f2p\beta_2 l}\|Q_lf\|_{L^2(\R^2)}.
		\end{equation}
		Combining \eqref{2.166} and \eqref{2.167} yields Lemma \ref{lem:Localized Strichartz estimate 6}.
	\end{proof}
	\begin{cor}\label{cor:Localized Strichartz estimate 6}
For $\beta_1\in(0,1)$, $\beta_2\in(\beta_1,\min\{\f32\beta_1,1\})$, $p\in(2,2+2\beta_2-\beta_1),\ r\in(2+\f{4}{2\beta_2-\beta_1},\infty)$ with $\f1p+\f1r=\f12$, $t\ge t_0\ge 0$ and integer  $k\ge -1$, one has
			\begin{equation}\label{cor:local:Stric 6}
				\sum_{j\ge-1}\|[1+2^k(s+|x|)]^{\f{1}{p}\beta_1}P_{k}|\nabla|^{-1}e^{\pm is|\nabla|}Q_jf\|_{L^p([t_0,t];L^r(\R^3))}
				\ls
				2^{(\f{2+2\beta_2}{p}-1)k}\|\w{x}^{\f{2}{p}\beta_2}f\|_{L^2(\R^3)}.
		\end{equation}
	\end{cor}
	\begin{proof}
By Lemmas \ref{lem:Localized Strichartz estimate 5}--\ref{lem:Localized Strichartz estimate 6}, following a similar argument
as in the proof of Corollary \ref{cor:Localized Strichartz estimate 4} yields Corollary \ref{cor:Localized Strichartz estimate 6}.
	\end{proof}

\subsection{Proof of Theorem \ref{thm:Weighted Strichartz estimate}}\label{Proof-1}
In this subsection, we complete the proof of Theorem \ref{thm:Weighted Strichartz estimate}.
	\begin{proof}
		Note that for $k\ge-1,\ 1+s+|x|\ls 1+2^k(s+|x|)$ and for $\iota\in\{0,1\}$,
		\begin{equation*}
			\begin{split}
				&\|(1+s+|x|)^{\f{1}{p}\beta_1}P_{k}|\nabla|^{-\iota}e^{\pm is|\nabla|}f\|_{L^p([t_0,t];L^r)} \\
				\ls&
				\sum_{j\ge -1}\|[1+2^k(s+|x|)]^{\f{1}{p}\beta_1}P_{k}|\nabla|^{-\iota}e^{\pm is|\nabla|}Q_jf\|_{L^p([t_0,t];L^r)}.
			\end{split}
		\end{equation*}
		Then \eqref{thm:Weighted Strichartz estimate 1} is proved by Corollary \ref{cor:Localized Strichartz estimate 2} and Corollary \ref{cor:Localized Strichartz estimate 4}; \eqref{thm:Weighted Strichartz estimate 2} is proved by Corollary \ref{cor:Localized Strichartz estimate 2} and Corollary \ref{cor:Localized Strichartz estimate 6}.
	\end{proof}

	\subsection{Weighted $L^2-L^2$ estimates}\label{sec:Weighted L^2-L^2 estimates}

	To prove Theorem~\ref{thm:2}, we also need to establish the following  conclusion.
	\begin{lemma}\label{lem:Technical lemma}
		For $\alpha\in(0,\f32)$, $t > 0 $, integers $j,\ k\ge -1$,
		\begin{itemize}
			\item[(1)] when $j\le \log_2t+100$, one has
			\begin{equation}
				2^{j\alpha}\|Q_jP_k e^{\pm it|\nabla|}f\|_{L^2(\R^3)}
				\ls
				(1+t)^\alpha\sum_{l< j-3}\|Q_lP_k f\|_{L^2(\R^3)}
				+
				\sum_{l\ge j-3}2^{l\alpha}\|Q_lP_k f\|_{L^2(\R^3)},
			\end{equation}
			\item[(2)] when $j\ge \log_2t+100$, one has
			\begin{equation}
				2^{j\alpha}\|Q_jP_k e^{\pm it|\nabla|}f\|_{L^2(\R^3)}
				\ls
				\sum_{l\ge-1}2^{l\alpha}\|Q_lP_k f\|_{L^2(\R^3)}.
			\end{equation}
		\end{itemize}
	\end{lemma}
	\begin{proof}
		Although the proof procedure is somewhat analogous to that of Lemma \ref{lem:Technical lemma 2},
we still give details for the reader's convenience. Note that
		\begin{equation}
			\begin{split}
				Q_jP_k e^{\pm it|\nabla|}f(x)
				=&Q_jP_{[[k]]} e^{\pm it|\nabla|}\big(\sum_{l\ge-1}Q_{[[l]]}Q_l\big)P_kf(x)\\
				=&\sum_{l\ge-1}(2\pi)^{-3}\int_{\R^3}\psi_j(x)K_k(t,x-y)\psi_{[[l]]}(y)Q_lP_kf(y)dy\\
				=&\sum_{l\ge-1}(2\pi)^{-3} I_{j,k,l}(t,x),
			\end{split}
		\end{equation}
		where $K_k(t,x) = K^{0,0}_k(t,x)$ for $k \ge 0$, $K^{0,0}_k$ is defined in \eqref{def:KklM} and $K_{-1}$ is defined in \eqref{def:K-1}.

\vskip 0.2 true cm
	\noindent	\textbf{Case 1.} $j\ge \log_2t+100$ and $l\le j-3$
\vskip 0.1 true cm
 In this case, due to $-1\le l\le j-3$, for $x\in\supp \psi_j$ and $y\in\supp \psi_{[[l]]}$, then it holds
		\begin{equation*}
			|x-y|\ge|x|-|y|\ge {\f{5}{8}}\cdot2^j-\f{16}{5}\cdot2^l\ge2^{j-4}>2t,
		\end{equation*}
		where we have used $\supp{\psi}_j\subseteq [\f{5}{8}\cdot2^j,\f{8}{5}\cdot2^j]$ for $j\ge 2$
and $\supp{\psi}_{[[l]]}\subseteq[0,\f{16}{5}\cdot2^l]$. Owing to Lemma \ref{lem:estimate of KklM} with $l=M=0$ and Lemma \ref{lem:estimate of K-1}, we have that for any $k\ge-1$,
		\begin{equation}
			|K_k(t,x-y)|\ls |x-y|^{-3}\ls 2^{-3j}.
		\end{equation}
		Therefore, by H\"{o}lder's inequality and $\alpha\in(0,\f32)$, one arrives at
		\begin{equation}\label{2.61}
			\begin{split}
				2^{j\alpha}\|I_{j,k,l}(t,x)\|_{L^2(\R^3)}
				\ls&
				2^{j\alpha}\|\psi_j(x)K_k(t,x-y)\psi_{[[l]]}(y)\|_{L^2_xL^2_y}\|Q_lP_kf\|_{L^2(\R^3)}\\
				\ls&
				2^{(\alpha-\f32) j+\f32 l}\|Q_lP_kf\|_{L^2(\R^3)}\\
				\ls&
				2^{l\alpha}\|Q_lP_kf\|_{L^2(\R^3)}.
			\end{split}
		\end{equation}
\vskip 0.2 true cm
	\noindent	\textbf{Case 2.} $j\le \log_2t+100$ and $l\le j-3$

\vskip 0.1 true cm
		In this case,
		\begin{equation}\label{2.62}
			2^{j\alpha}\|I_{j,k,l}(t,x)\|_{L^2(\R^3)}=	2^{j\alpha}\|Q_jP_{[[k]]} e^{\pm it|\nabla|} Q_lP_kf\|_{L^2(\R^3)}
			\ls (1+t)^\alpha\|Q_lP_kf\|_{L^2(\R^3)}.
		\end{equation}
\vskip 0.1 true cm
	\noindent	\textbf{Case 3.} $l\ge j-3$

\vskip 0.1 true cm
		In this case, one has
		\begin{equation}\label{2.63}
			2^{j\alpha}\|I_{j,k,l}(t,x)\|_{L^2(\R^3)}=	2^{j\alpha}\|Q_jP_{[[k]]} e^{\pm it|\nabla|} Q_lP_kf\|_{L^2(\R^3)}
			\ls 2^{l\alpha }\|Q_lP_kf\|_{L^2(\R^3)}.
		\end{equation}

		Combining \eqref{2.61}--\eqref{2.63} yields Lemma \ref{lem:Technical lemma}.
	\end{proof}
	\begin{cor}\label{cor:Technical lemma}
		For $\beta_1\in(0,\f32)$, $\beta_2>0$, $t\ge 0 $, integers $j,\ k\ge -1$, it holds that
		\begin{equation}\label{YHCCC-33}
			\begin{split}
				\|\w{x}^{\beta_1}P_k e^{\pm it|\nabla|}f\|_{L^2(\R^3)}
				\ls&
				(1+t)^{\beta_1}\|\w{x}^{\beta_2} f\|_{L^2(\R^3)}
				+
				\|\w{x}^{\beta_1+\beta_2 } f\|_{L^2(\R^3)}\\
				\ls&
				(1+t)^{\beta_1+\beta_2 }\|f\|_{L^2(\R^3)}
				+
				\|\w{x}^{\beta_1+\beta_2 } f\|_{L^2(\R^3)}.
			\end{split}
		\end{equation}
	\end{cor}
	\begin{proof}
		\eqref{YHCCC-33} follows from Lemma \ref{lem:Technical lemma}, Lemma \ref{lem:Qj prop} and decomposition of
$\R^3=\{|x|<t\}\cup\{|x|\ge t\}$, we omit the details.
	\end{proof}

\section{A priori estimates}\label{Section 4}
In this section, we establish the uniform energy estimates and the Strichartz estimates for problem \eqref{eq:wave}.

\subsection{Reformulation for nonlinear equation in \eqref{eq:wave}}

Since only small data solutions is considered in problem \eqref{eq:wave},
one can assume $Q^{00\gamma\delta}_{1ijkl}=Q^{00\gamma}_{2ijkl}=Q^{00}_{3ijkl}=0$ without loss of generality.
In this case, the system \eqref{eq:wave} takes the following form
\begin{equation}\label{eq:wave2}
	\begin{aligned}
		\square_{c_i} u^i =G^i(u,\p u,\p_x\p u) = \sum_{j=1}^m\sum_{\alpha,\beta=0}^3 \cQ^{\alpha\beta}_{ij} \partial_{\alpha\beta}^2 u^j+\mathcal{S}^i(u,\p u),
	\end{aligned}
\end{equation}
where
\begin{equation}\label{3.102}
\cQ^{00}_{ij}=0,\quad	\cQ^{\alpha\beta}_{ij}
	=\sum_{k,l=1}^m \big[\sum_{\gamma,\delta=0}^{3} Q^{\alpha\beta\gamma\delta}_{1ijkl}\partial_\gamma u^k \partial_\delta u^l
	+\sum_{\gamma=0}^{3} Q^{\alpha\beta\gamma}_{2ijkl}\partial_\gamma u^k  u^l
	+Q^{\alpha\beta}_{3ijkl} u^k  u^l\big]
\end{equation}
and
\begin{equation}
\mathcal{S}^i(u,\p u)=\sum_{j,k,l=1}^m\big[\sum_{\alpha,\beta,\gamma=0}^3S^{\alpha\beta\gamma}_{1ijkl}\partial_{\alpha} u^j\partial_\beta u^k  \partial_\gamma u^l+\sum_{\alpha,\beta=0}^3 S^{\alpha\beta}_{2ijkl}\partial_{\alpha}u^j\partial_\beta u^k  u^l
+\sum_{\alpha=0}^3 S^{\alpha}_{3ijkl}\partial_{\alpha} u^j  u^k  u^l\big].		
\end{equation}
Meanwhile, owing to the symmetric condition \eqref{eq:symmetric condition}, it holds that
\begin{equation}\label{eq:symmetric condition2}
	\cQ^{\alpha\beta}_{ij} =\cQ^{\beta\alpha}_{ij} =\cQ^{\alpha\beta}_{ji}.
\end{equation}

It follows from the spatial Fourier transformation and the equation (\ref{eq:wave2}) that
\begin{equation}\label{eq:u}
	\begin{aligned}
		u^i(t,x) &= \cos(c_it|\nabla|)u^i_0(x) + \frac{\sin(c_it|\nabla|)}{c_i|\nabla|}u^i_1(x) + \int_0^t \frac{\sin\big(c_i(t-s)|\nabla|\big)}{c_i|\nabla|} G^i(u,\p u,\p_x\p u)(s)ds. \\
	\end{aligned}
\end{equation}
Meanwhile,
	\begin{equation}\label{eq:pt Pku}
		\begin{aligned}
			&P_k \partial_tu^i(t,x) \\
			&= -c_i\sin(c_i t|\nabla|)|\nabla| P_ku^i_0(x) + \cos(c_i t|\nabla|) P_k u_1^i(x)
			+ \int_0^t \cos\big(c_i(t-s)|\nabla|\big) P_kG^i(u,\p u,\p_x\p u)(s)ds \\
			&= \frac{c_i}{2\sqrt{-1}}\big(e^{-\sqrt{-1}c_i t|\nabla|} - e^{\sqrt{-1}c_i t|\nabla|}\big) R\cdot\nabla P_k u^i_0(x) + \frac12\big(e^{\sqrt{-1}c_i t|\nabla|} + e^{-\sqrt{-1}c_i t|\nabla|}\big)  P_k u^i_1(x) \\
			&\quad+ \int_0^t \frac12\big(e^{\sqrt{-1}c_i(t-s)|\nabla|} + e^{-\sqrt{-1}c_i(t-s)|\nabla|}\big)P_k G^i(u,\p u,\p_x\p u)(s)ds
		\end{aligned}
	\end{equation}
and
\begin{equation}\label{eq:nabla Pku}
	\begin{aligned}
		&P_k\nabla u^i(t,x) \\
		&= \frac12\big(e^{\sqrt{-1}c_it|\nabla|} + e^{-\sqrt{-1}c_it|\nabla|}\big)  P_k\nabla u^i_0(x) + \frac{1}{2\sqrt{-1}c_i}(e^{\sqrt{-1}c_it|\nabla|} - e^{-\sqrt{-1}c_it|\nabla|})R P_k  u^i_1(x) \\
		 &\quad+ \frac{1}{2\sqrt{-1}c_i}\int_0^t (e^{\sqrt{-1}c_i(t-s)|\nabla|} - e^{-\sqrt{-1}c_i(t-s)|\nabla|}) R P_k G^i(u,\p u,\p_x\p u)(s)ds,
	\end{aligned}
\end{equation}
where $R = \f{\nabla}{|\nabla|}$ denotes the Riesz transformation.

Set $c=(c_1,\cdots,c_m)$.
Then one can abbreviate equations \eqref{eq:u}--\eqref{eq:nabla Pku} as:
\begin{equation}\label{eq:Pku}
P_k u(t,x) = |\nabla|^{-1}e^{\pm \sqrt{-1} ct|\nabla|}(\mathrm{Id},R)\cdot P_k\partial u(0) + \int_0^t  |\nabla|^{-1}e^{\pm \sqrt{-1}c(t-s)|\nabla|}  P_k G(u,\partial u,\p_x\p u)(s)ds,
\end{equation}
\begin{equation}\label{eq:partial Pku}
	P_k \partial u(t,x) = e^{\pm \sqrt{-1} ct|\nabla|}(\mathrm{Id},R)\cdot P_k\partial u(0) + \int_0^t  e^{\pm \sqrt{-1}c(t-s)|\nabla|} (\mathrm{Id},R) P_k G(u,\partial u,\p_x\p u)(s)ds.
\end{equation}

\subsection{Energy estimates}

\begin{lemma}[Energy estimate]\label{lem:Energy estimate}
	Let $\delta>0$ and integer $N \geq 2$. If $u$ is the solution of \eqref{eq:wave2}, where $\|\partial u(t)\|_{H^{N}(\mathbb{R}^3)} \leq \varepsilon_1$ for $t\ge 0$ and $\varepsilon_1>0$ is small, then one has
	\begin{equation}\label{3.4}
		\begin{aligned}
			\|\partial u(t)\|_{H^{N}(\mathbb{R}^3)} &\lesssim \|\partial u(0)\|_{H^{N}(\mathbb{R}^3)}
			+ \big(\sum_{k \geq -1} 2^{(1+\delta)k} \|P_k \partial^{\le 1} u\|_{L^2([0,t];L^\infty(\mathbb{R}^3))}\big)^2 \|\partial u\|_{L^\infty([0,t];H^{N}(\mathbb{R}^3))}.
		\end{aligned}
	\end{equation}
\end{lemma}

\begin{proof}
	Acting $\partial_x^a$ with $a\in\N_0^3$ on two sides of \eqref{eq:wave2} yields
	\begin{equation}\label{3.5}
		\begin{aligned}
			\square_{c_i}\partial_x^a u^i &= I_1^{a,i} + I_2^{a,i} +  I_3^{a,i},\\	
			I_1^{a,i} &= \sum_{j=1}^m\sum_{\alpha,\beta=0}^3 \cQ^{\alpha\beta}_{ij} \p_x^a\partial_{\alpha\beta}^2 u^j, \\
			I_2^{a,i} &= \sum_{j=1}^m\sum_{\alpha,\beta=0}^3\sum_{\substack{b+c=a, \\ |b| \leq |a|-1}}  \p_x^c\cQ^{\alpha\beta}_{ij} \p_x^b\partial_{\alpha\beta}^2 u^j, \\
			I_3^{a,i} &= \p_x^a\mathcal{S}^i(u,\p u),\\
		\end{aligned}
	\end{equation}
	where $I_2^{a,i}$ vanishes for $|a|=0$.

	Note that $\int_{\R^3}(\p_t f)^2 + c^2|\nabla f|^2 dx \approx \|\partial f\|_{L^2(\mathbb{R}^3)}^2 $.
Multiplying \eqref{3.5} by $\partial_t \partial_x^a u^i$ and integrating the resulting equality over
	$[0,t] \times \mathbb{R}^3$, it follows from the summation over the index $i$ that
	\begin{equation}\label{3.6}
		\begin{aligned}
			\|\partial_x^a\partial  u(t)\|_{L^2(\mathbb{R}^3)}^2 &\lesssim \| \partial_x^a\partial u(0)\|_{L^2(\mathbb{R}^3)}^2
+ \sum_{i=1}^{m}\big| \int_0^t \int_{\mathbb{R}^3} \partial_t \partial_x^a u^i I_1^{a,i} dx d\tau \big| \\
			&\quad +\sum_{i=1}^{m} \int_0^t \|\partial_t \partial_x^a u^i(\tau)\|_{L^2(\mathbb{R}^3)} \big(\|I_2^{a,i}\|_{L^2(\mathbb{R}^3)}+\|I_3^{a,i}\|_{L^2(\mathbb{R}^3)}\big) d\tau.
		\end{aligned}
	\end{equation}
	Firstly, we deal with the last term in the first line of \eqref{3.6}. Note that
	\begin{equation}\label{3.7}
		\begin{aligned}
			\cQ^{\alpha\beta}_{ij} \partial_t \partial_x^a u^i  \p_x^a\partial_{\alpha\beta}^2 u^j
			=&\partial_\alpha (\cQ^{\alpha\beta}_{ij} \partial_t \partial_x^a u^i  \p_x^a\partial_{\alpha} u^j)
			-\p_\alpha \cQ^{\alpha\beta}_{ij}\partial_t \partial_x^a u^i \partial_\beta \partial_x^a u^j\\
			&-\frac{1}{2} \partial_t (\cQ^{\alpha\beta}_{ij} \partial_t \partial_x^a u^i  \p_x^a\partial_{\alpha} u^j)
			+\frac{1}{2} \p_t \cQ^{\alpha\beta}_{ij}  \partial_\alpha \partial_x^a u^i \partial_\beta \partial_x^a u^j ,
		\end{aligned}
	\end{equation}
	where the Einstein summation convention for $i,j = 1,\cdots,m,\ \alpha,\beta=0,1,2,3$ and the symmetric condition \eqref{eq:symmetric condition2} are used in \eqref{3.7}. On the other hand, noting that if $u\in H^2(\R^3)\subset S'_h(\R^3)$, then it follows from Bernstein's inequality that
	\begin{equation*}
		\|u\|_{L^\infty(\R^3)}
		\ls
		\sum_{k\in\Z}2^{\f32 k}\|\pk u\|_{L^2(\R^3)}
		\ls
		\sum_{k\in\Z}2^{\f12 k}\|\pk\nabla u\|_{L^2(\R^3)}
		\ls
		\|\nabla u\|_{H^{1}(\R^3)}.
	\end{equation*}
	Hence, it holds
	\begin{equation}\label{3.8}
		\begin{aligned}
			\|\cQ^{\alpha\beta}_{ij}(t)\|_{L^\infty}
			\ls
			\|\partial u(t)\|_{L^\infty(\mathbb{R}^3)}^2 +\|u(t)\|_{L^\infty}^2
			\ls
			\|\partial u(t)\|_{H^2}^2
			\ls \ve_1^2.\\
		\end{aligned}
	\end{equation}
	It follows from \eqref{initial:data} (or \eqref{initial:data2}), \eqref{3.8} and integration over $[0,t] \times \mathbb{R}^3$
for \eqref{3.7} that
	\begin{equation*}
		\begin{aligned}
		\sum_{i=1}^{m}\big| \int_0^t \int_{\mathbb{R}^3} \partial_t \partial_x^a u^i I_1^{a,i} dx dt\big|
        &\lesssim \ve^2\|\partial_x^a\partial  u(0)\|_{L^2}^2 + \ve^2_1\|\partial_x^a\partial  u(t)\|_{L^2}^2 \\
		&\quad + \int_0^t  \|\partial_x^{\le 1}\partial^{\le 1}  u(\tau)\|_{L^\infty}^2 \|\partial_x^a\partial  u(\tau)\|_{L^2}^2 d\tau,
		\end{aligned}	
	\end{equation*}
	where the boundary term in \eqref{3.7} vanishes due to $\cQ^{\alpha\beta}_{ij}\partial_t \partial_x^a u^i \partial_\beta \partial_x^a u^j  \in L^1(\mathbb{R}^3)$.\\
	Next, we treat $\|I_2^{a,i}\|_{L^2(\mathbb{R}^3)}$ with $|a| \geq 1$ in \eqref{3.5}.
It is easy to find that
	\begin{equation*}
		\begin{aligned}
			\|I_2^{a,i}\|_{L^2}
			\lesssim& \big\| \|P_k I_2^{a,i}\|_{L^2}\big\|_{\ell_k^2} \\
			\lesssim& \sum_{j=1}^m\sum_{\alpha,\beta=0}^3\sum_{\substack{b+c=a, \\ |b| \leq |a|-1}} \big\| \sum_{(k_1,k_2) \in \mathcal{X}_k} \| P_k(P_{k_1} \partial_{\alpha\beta}^2 \partial_x^b u^j P_{k_2} \p_x^c\cQ^{\alpha\beta}_{ij} )\|_{L^2} \big\|_{\ell_k^2}.
		\end{aligned}
	\end{equation*}
Note that $\mathcal{Q}^{00}_{ij}=0$ and $|c|\ge 1$. When $\max\{k_1, k_2\} = k_1$, one then has that by Bernstein's inequality
	\begin{equation*}
		\begin{aligned}
			&\|P_k(P_{k_1} \partial_{\alpha\beta}^2 \partial_x^b u^j P_{k_2} \p_x^c\cQ^{\alpha\beta}_{ij} )\|_{L^2}\\
			&\lesssim \|P_{k_1} \partial\partial_x \partial_x^b u\|_{L^2} \|P_{k_2} \p_x^c\cQ^{\alpha\beta}_{ij} \|_{L^\infty} \\
			&\lesssim 2^{k_1(|b|+1)+k_2(|c|-1)} \|P_{k_1} \partial u\|_{L^2} \|P_{k_2}\p_x \cQ^{\alpha\beta}_{ij}\|_{L^\infty}  \\
			&\lesssim 2^{k_1|a|} \|P_{k_1} \partial u\|_{L^2} \|P_{k_2} \p_x \cQ^{\alpha\beta}_{ij}\|_{L^\infty}.
		\end{aligned}	
	\end{equation*}
	Therefore,
	\begin{equation*}
		\begin{aligned}
			\big\| \sum_{\substack{(k_1,k_2) \in \mathcal{X}_k, \\ \max\{k_1,k_2\}=k_1}} \| P_k(P_{k_1} \partial_{\alpha\beta}^2 \partial_x^b u^j P_{k_2} \p_x^c\cQ^{\alpha\beta}_{ij} )\|_{L^2}\big\|_{\ell_k^2}
			\lesssim \|\partial u\|_{H^{|a|}}  \| \cQ^{\alpha\beta}_{ij} \|_{B^{1+\delta}_{\infty,1}}.
		\end{aligned}
	\end{equation*}
	When $\max\{k_1, k_2\} = k_2$, we can analogously obtain
	\begin{equation}\label{3.11}
	\begin{aligned}
	\big\| \sum_{\substack{(k_1,k_2) \in \mathcal{X}_k, \\ \max\{k_1,k_2\}=k_2}} \|P_k(P_{k_1} \partial_{\alpha\beta}^2 \partial_x^b u^j P_{k_2} \p_x^c\cQ^{\alpha\beta}_{ij} )\|_{L^2}\big\|_{\ell_k^2}
	\lesssim \|\p_x\cQ^{\alpha\beta}_{ij} \|_{H^{|a|-1}} \| \partial u\|_{B^{1+\delta}_{\infty,1}} .
	\end{aligned}	
	\end{equation}
Note that both $\dot{B}^s_{p,r}(\R^3)\cap L^\infty(\R^3)$ and $B^s_{p,r}(\R^3)\cap L^\infty(\R^3)$ are algebras for $s > 0$, $p,r\in[1,\infty]$.
Then for $1\le |a|\le N$, it holds
\begin{equation*}
\begin{split}
\| \cQ^{\alpha\beta}_{ij} \|_{B^{1+\delta}_{\infty,1}}&\ls \| \p^{\le 1} u \|_{B^{1+\delta}_{\infty,1}}\| \p^{\le 1} u \|_{L^{\infty}}\ls \| \p^{\le 1} u \|_{B^{1+\delta}_{\infty,1}}^2,\\
\|\p_x\cQ^{\alpha\beta}_{ij} \|_{H^{|a|-1}}&\ls\|\p_x\cQ^{\alpha\beta}_{ij} \|_{L^{2}} +\|\cQ^{\alpha\beta}_{ij} \|_{\dot{H}^{|a|}}\\
&\ls\|\partial_x^{\le 1}\partial  u\|_{L^2}\|\partial^{\le 1}  u\|_{L^\infty}
+\|\partial^{\le 1}  u\|_{\dot{H}^{|a|}} \|\partial^{\le 1}  u\|_{L^\infty}\\
&\ls\|\partial u\|_{{H}^{|a|}} \|\partial^{\le 1}  u\|_{L^\infty}.
\end{split}
\end{equation*}
Hence,
\begin{equation*}
\sum_{i=1}^{m}\|I_2^{a,i}\|_{L^2}\ls \|\partial u\|_{H^{|a|}} \|\partial^{\le1} u\|_{B^{1+\delta}_{\infty,1}}^2.
\end{equation*}
Finally, we handle $\|I_3^{a,i}\|_{L^2(\mathbb{R}^3)}$ in \eqref{3.5}. For $|a|= 0$, one has
\begin{equation*}
\sum_{i=1}^{m}\|I_3^{0,i}\|_{L^2}\ls \|\p u\|_{L^2}\|\p^{\le 1} u\|_{L^\infty}^2.
\end{equation*}
For $1\le|a|\le N$, it holds
\begin{equation}\label{3.15}
\begin{split}
\sum_{i=1}^{m}\|I_3^{a,i}\|_{L^2}
&\ls \|\partial u\|_{\dot{H}^{|a|}} \|\partial^{\le1} u\|_{L^\infty}^2 + \|\partial^{\le 1} u\|_{\dot{H}^{|a|}} \|\partial^{\le1} u\|_{L^\infty}\|\partial u\|_{L^\infty}\\
&\ls 	\|\partial u\|_{{H}^{|a|}} \|\partial^{\le 1}  u\|_{L^\infty}^2.
\end{split}	
\end{equation}
Collecting \eqref{3.6}--\eqref{3.15} together with the smallness of  $\ve$ and $\ve_1$ yields
\begin{equation}\label{4.29}
\begin{aligned}
\|\partial  u(t)\|_{H^{|a|}(\mathbb{R}^3)}^2 &\lesssim \|\partial  u(0)\|_{H^{|a|}(\mathbb{R}^3)}^2
+\int_0^t   \|\partial^{\le1} u(\tau)\|_{B^{1+\delta}_{\infty,1}(\R^3)}^2   \|\partial  u(\tau )\|_{H^{|a|}(\mathbb{R}^3)}^2 d\tau.
\end{aligned}
\end{equation}
Then \eqref{3.4} is derived  from \eqref{4.29}, Gronwall's inequality and H\"older's inequality.
\end{proof}

\subsection{Strichartz estimate for Theorem \ref{thm:1}}

\begin{lemma}\label{lem:Estimate of Besov norm}
	For $t\ge0$ and $\delta>0$, it holds that
	\begin{equation}\label{ineq:Estimate of Besov norm}
	\sum_{k\ge -1}2^{(1+\delta) k}\|P_k\p^{\le 1} u\|_{L^{2}_tL^\infty_x}
	\ls
	\ln(e+t)
	\big[\|\p u(0)\|_{H_x^{2+3\delta}}+
	\big(\sum_{k \geq -1}2^k\|P_k\p^{\le 1} u\|_{L^2_tL^\infty_x}\big)^2\|\p u\|_{L^\infty_t H_x^{3+3\delta}}\big].
	\end{equation}
Moreover, if $G(u, \p u, \p^2 u)$ is independent of $u$ (see \eqref{eq:C independent u}), then we have
	\begin{equation}
		\sum_{k\ge -1}2^{(1+\delta) k}\|P_k\p u\|_{L^{2}_tL^\infty_x}
		\ls
		\ln^{\f12}(e+t)
		\big[\|\p u(0)\|_{H_x^{2+3\delta}}+
		\big(\sum_{k \geq -1}2^k\|P_k\p u\|_{L^2_tL^\infty_x}\big)^2\|\p u\|_{L^\infty_t H_x^{3+3\delta}}\big].
	\end{equation}
\end{lemma}		
\begin{proof}
	From \eqref{eq:Pku} and \eqref{eq:partial Pku}, one has
	\begin{equation*}
	\begin{split}
	\|P_k\p^{\le 1} u\|_{L^{2}_tL^\infty_x}
	\ls
	&\sum_{\iota=0,1}\big\|\int_{0}^{s} |\nabla|^{-\iota}e^{\pm \sqrt{-1}c(s-\tau)|\nabla|} (\mathrm{Id},R) P_k G(u,\p u,\p_x\p u)(\tau)d\tau\big\|_{L^{2}_s([0,t])L^\infty_x}\\
	&+\sum_{\iota=0,1}\||\nabla|^{-\iota}e^{\pm \sqrt{-1}cs|\nabla|}(\mathrm{Id},R)\cdot P_k\partial u(0)\|_{L^{2}_sL^\infty_x}.
	\end{split}
	\end{equation*}
	Then it follows from Minkowski's inequality that
	\begin{equation*}
	\begin{split}
	\|P_k\p^{\le 1} u\|_{L^{2}_tL^\infty_x}
	\ls
	&\sum_{\iota=0,1}\int_{0}^{t}\big\| |\nabla|^{-\iota}e^{\pm \sqrt{-1}c(s-\tau)|\nabla|} (\mathrm{Id},R) P_k G(u,\p u,\p_x\p u)(\tau)d\tau\big\|_{L^{2}_s([\tau,t])L^\infty_x}d\tau\\
	&+\sum_{\iota=0,1}\||\nabla|^{-\iota}e^{\pm \sqrt{-1}cs|\nabla|}(\mathrm{Id},R)\cdot P_k\partial u(0)\|_{L^{2}_sL^\infty_x}.
	\end{split}
	\end{equation*}
	Applying Lemmas \ref{lem:strichartz L2Linfty}--\ref{lem:strichartz L2Linfty2} and the $L^2$ isometry property of $R = \f{\nabla}{|\nabla|}$
 yields
	\begin{equation}\label{inbootstrap1}
	\begin{split}
	2^{(1+\delta) k}\|P_k\p^{\le 1 } u\|_{L^{2}_tL^\infty_x}
	\ls
	&2^{(2+\delta) k}
	\int_{0}^{t}\ln(e+2^k c(t-\tau))\left\|(\mathrm{Id},R) P_k G(u,\p u,\p_x\p u)(\tau)\right\|_{L^2_x}d\tau\\
	&+2^{(2+\delta) k}\ln(e+2^k ct)\|(\mathrm{Id},R)\cdot P_k\partial u(0)\|_{L^2_x}\\
	\ls
	&2^{(2+\delta) k}\ln(e+2^k ct)
	\int_{0}^{t}\left\| P_kG(u,\p u,\p_x\p u)(\tau)\right\|_{L^2_x}d\tau\\
	&+2^{(2+\delta) k}\ln(e+2^k ct)\|P_k\p u(0)\|_{L^2_x}.
	\end{split}
	\end{equation}
	Owing to $\ln(e+2^k ct)\le (k+2)\ln2+\ln(c+1)+\ln(e+t) \ls 2^{\delta k}\ln(e+t)$ for all $k\ge-1$ and $t\ge0$,
one can arrive at
	\begin{equation}\label{3.12}
		2^{(1+\delta) k}\|P_k\p^{\le1} u\|_{L^{2}_tL^\infty_x}
		\ls
		2^{(2+2\delta) k}\ln(e+t)\big[\|P_k\p u(0)\|_{L^2_x}+
		\int_{0}^{t}\big\| P_k G(u,\p u,\p_x\p u)(\tau)\big\|_{L^2_x}d\tau\big].
	\end{equation}
Summing over $k$ in \eqref{3.12} yields
	\begin{equation}\label{inbootstrap2}
	\sum_{k\ge -1}2^{(1+\delta) k}\|P_k\p^{\le1} u\|_{L^{2}_tL^\infty_x}
	\ls \ln(e+t)
	\big[\|\p u(0)\|_{H_x^{2+3\delta}}
	+\int_{0}^{t}\left\|G(u,\p u,\p_x\p u)(\tau)\right\|_{H_x^{2+3\delta}}d\tau\big].
	\end{equation}
In addition,
	\begin{equation}\label{4.36}
		\begin{aligned}
			&\left\|G(u,\p u,\p_x\p u)(\tau)\right\|_{H_x^{2+3\delta}}\\
			&\ls
			\|\p_x^{\le 1}\p u\cdot\p^{\le 1} u\cdot\p^{\le 1} u\|_{L_x^2}
			+
			\|\p_x^{\le 1}\p u\cdot\p^{\le 1} u\cdot\p^{\le 1} u\|_{\dot{H}_x^{2+3\delta}}\\
			&\ls
			\|\p_x^{\le 1}\p u\|_{L_x^2} \|\p^{\le 1} u\|_{L_x^\infty}^2
			+
			\|\p_x^{\le 1}\p u\|_{\dot{H}_x^{2+3\delta}} \|\p^{\le 1} u\|_{L_x^\infty}^2
			+
			\|\p^{\le 1} u\|_{\dot{H}_x^{2+3\delta}} \|\p_x^{\le 1}\p u\|_{L_x^\infty}\|\p^{\le 1} u\|_{L_x^\infty}\\
			&\ls
			\|\p u\|_{{H}_x^{3+3\delta}} \|\p_x^{\le 1}\p^{\le 1} u\|_{L_x^\infty}^2.		
		\end{aligned}
	\end{equation}
	By \eqref{inbootstrap2} and \eqref{4.36}, we have
	\begin{equation}\label{YHCCC-401}
		\sum_{k\ge -1}2^{(1+\delta) k}\|P_k\p^{\le1} u\|_{L^{2}_tL^\infty_x}
		\ls \ln(e+t)
		\big[\|\p u(0)\|_{H_x^{2+3\delta}}
		+\big(\sum_{k \geq -1}2^k\|P_k\p^{\le1} u\|_{L^2_tL^\infty_x}\big)^2\|\p u\|_{L^\infty_t H_x^{3+3\delta}}\big].
	\end{equation}
	On the other hand, under condition \eqref{eq:C independent u}, it follows from
Lemma \ref{lem:strichartz L2Linfty} that \eqref{inbootstrap1} can be improved to
	\begin{equation}
		\begin{split}
			2^{(1+\delta) k}\|P_k\p u\|_{L^{2}_tL^\infty_x}
			\ls
			&2^{(2+\delta) k}\ln^{\f12}(e+2^k ct)
			\int_{0}^{t}\left\| P_k  G(\p u,\p_x\p u)(\tau)\right\|_{L^2_x}d\tau\\
			&+2^{(2+\delta) k}\ln^{\f12}(e+2^k ct)\|P_k\p u(0)\|_{L^2_x}.
		\end{split}
	\end{equation}
	Then, proceeding exactly as in the proof of estimates \eqref{3.12}--\eqref{YHCCC-401}, we have
	\begin{equation*}
		\sum_{k\ge -1}2^{(1+\delta) k}\|P_k\p u\|_{L^{2}_tL^\infty_x}
		\ls \ln^{\f12}(e+t)
		\big[\|\p u(0)\|_{H_x^{2+3\delta}}
		+\big(\sum_{k \geq -1}2^k\|P_k\p u\|_{L^2_tL^\infty_x}\big)^2\|\p u\|_{L^\infty_t H_x^{3+3\delta}}\big].
	\end{equation*}
	Thus, the proof of Lemma \ref{lem:Estimate of Besov norm} is complete.
\end{proof}	

\subsection{Strichartz estimate I for Theorem \ref{thm:2}}

\begin{lemma}\label{lem:Estimate of Besov norm 2}
	For $\mu\in(0,1),\ \delta\in (0,10^{-4}\mu)$ and $t\ge0$, it holds that
	\begin{equation}\label{ineq:Estimate of Besov norm 2}
	\begin{split}
	&\sum_{k\ge -1}2^{(1+\delta) k}\|(1+s+|x|)^{\f{\mu}{2}-10\delta}P_k\p^{\le 1} u\|_{L^{2}_tL^\infty_x}\\
	\ls&
	\sum_{k\ge -1}2^{(2 +\mu -18.9\delta) k}\|\w{x}^{\mu}P_k\partial u(0)\|_{L^2_x}\\
	&+
	\sum_{k\ge -1}2^{(3 +\mu -18.8\delta) k}\|(1+s+|x|)^{2.2\delta} P_{k}\p u\|_{L_t^{1+\f{1}{4\delta}}L^{\f{2+8\delta}{1-4\delta}}_x}\\
	&\quad\times
	\big(\sum_{l\ge-1}2^{(2 +\mu-18.8\delta) l} \|(1+s+|x|)^{\f{\mu}{2}-11\delta}P_{l}\p^{\le 1} u\|_{L_t^{2+8\delta}L^{2+\f{1}{2\delta}}_x}\big)^2.
	\end{split}
	\end{equation}
\end{lemma}	
\begin{proof}
	Denote $A_\nu(s,x)=(1+s+|x|)^{\nu}$. Due to $A_\nu(s,x)\approx A_\nu(cs,x) $ for any $c\in[\min\{c_i\},\max\{c_i\}]$,
then it follows from \eqref{eq:partial Pku} that
	\begin{equation}\label{3.16}
	\begin{split}
	&\|A_{\f{\mu}{2}-10\delta}(s,x)P_k\p^{\le 1} u\|_{L^{2}_tL^\infty_x}\\
	&\ls
	\sum_{\iota=0,1}\big\|\int_{0}^{s} A_{\f{\mu}{2}-10\delta}(cs,x)|\nabla|^{-\iota}e^{\pm \sqrt{-1}c(s-\tau)|\nabla|} (\mathrm{Id},R) P_k G(u,\p u,\p_x\p u)(\tau)d\tau\big\|_{L^{2}_s([0,t];L^\infty_x)}\\
	&\quad+\sum_{\iota=0,1}\|A_{\f{\mu}{2}-10\delta}(cs,x)|\nabla|^{-\iota}e^{\pm \sqrt{-1}cs|\nabla|}(\mathrm{Id},R) P_k\partial u(0)\|_{L^{2}_sL^\infty_x}.
	\end{split}
	\end{equation}
	Together with Minkowski's inequality, this yields
	\begin{equation*}
	\begin{split}
	&\|A_{\f{\mu}{2}-10\delta}(s,x)P_k\p^{\le 1} u\|_{L^{2}_tL^\infty_x}\\
	&\ls
	\sum_{\iota=0,1}\int_{0}^{t} \big\|A_{\f{\mu}{2}-10\delta}(cs,x)P_{[[k]]}|\nabla|^{-\iota}e^{\pm \sqrt{-1}cs|\nabla|} e^{\mp \sqrt{-1}c\tau|\nabla|} (\mathrm{Id},R)
	 P_kG(u,\p u,\p_x\p u)(\tau)\big\|_{L^{2}_s([\tau,t];L^\infty_x)}d\tau\\
	&\quad+\sum_{\iota=0,1}\|A_{\f{\mu}{2}-10\delta}(cs,x)P_{[[k]]}|\nabla|^{-\iota}e^{\pm \sqrt{-1}cs|\nabla|} (\mathrm{Id},R)P_k\partial u(0)\|_{L^{2}_sL^\infty_x}.
	\end{split}
	\end{equation*}
	Applying Theorem \ref{thm:Weighted Strichartz estimate} with $\beta_1={\mu-20\delta}$ and $\beta_2=\mu-19.9\delta$, one has
	\begin{equation}\label{3.19}
	\begin{split}
	&\|A_{\f{\mu}{2}-10\delta}(s,x)P_k\p^{\le 1} u\|_{L^{2}_tL^\infty_x}\\
	&\ls
	2^{(1+\mu-19.9\delta) k}
	\int_{0}^{t}\big\|\w{x}^{\mu-19.9\delta}P_ke^{\mp \sqrt{-1}c\tau|\nabla|} (\mathrm{Id},R) G(u,\p u,\p_x\p u)(\tau)\big\|_{L^2_x}d\tau\\
	&\quad+2^{(1+\mu-19.9\delta) k}
	\|\w{x}^{\mu-19.9\delta}P_k(\mathrm{Id},R)\partial u(0)\|_{L^2_x}.
	\end{split}
	\end{equation}
	Owing to the $L^2$ isometry property of $R = \f{\nabla}{|\nabla|}$ and $\w{x}^{2\mu-39.8\delta}\in A_2(\R^3)$ for $\mu\in(0,1)$,
 then by Lemma \ref{lem:riesz L2}, we obtain
	\begin{equation}\label{3.20}
	\|\w{x}^{\mu-19.9\delta}(\mathrm{Id},R)P_k\partial u(0)\|_{L^2_x}
	\ls
	\|\w{x}^{\mu-19.9\delta}P_k\partial u(0)\|_{L^2_x}.
	\end{equation}
On the other hand, by Corollary \ref{cor:Technical lemma} with $\beta_1 =\mu-19.9\delta$ and $\beta_2 = 0.1\delta$,
it holds that
\begin{equation*}
\begin{split}
&2^{(1+\mu-19.9\delta) k}
\int_{0}^{t}\left\|\w{x}^{\mu-19.9\delta}P_ke^{\mp \sqrt{-1}c\tau|\nabla|} (\mathrm{Id},R) G(u,\p u,\p_x\p u)(\tau)\right\|_{L^2_x}d\tau\\
&\ls 2^{(1+\mu-19.9\delta) k}
\int_{0}^{t}(1+c\tau)^{\mu-19.8\delta}\left\|(\mathrm{Id},R)
P_k G(u,\p u,\p_x\p u)(\tau)\right\|_{L^2_x}d\tau\\
&\quad+2^{(1+\mu-19.9\delta) k}
\int_{0}^{t}\left\|\w{x}^{\mu-19.8\delta}(\mathrm{Id},R)
P_k G(u,\p u,\p_x\p u)(\tau)\right\|_{L^2_x}d\tau.
\end{split}
\end{equation*}
	By Lemma \ref{lem:riesz L2}, we have
	\begin{equation}\label{3.22}
	\begin{split}
	&2^{(1+\mu-19.9\delta) k}
	\int_{0}^{t}\big\|\w{x}^{\mu-19.9\delta}P_ke^{\mp \sqrt{-1}c\tau|\nabla|} (\mathrm{Id},R) G(u,\p u,\p_x\p u)(\tau)\big\|_{L^2_x}d\tau\\
	&\ls
	2^{(1+\mu-19.9\delta) k}
	\int_{0}^{t}\left\|A_{\mu-19.8\delta}(\tau,x)
	P_k G(u,\p u,\p_x\p u)(\tau)\right\|_{L^2_x}d\tau.\\
	\end{split}
	\end{equation}
	Collecting \eqref{3.19}, \eqref{3.20} and \eqref{3.22} yields
	\begin{equation}\label{3.23}
		\begin{split}
			&\|A_{\f{\mu}{2}-10\delta}(s,x)P_k\p^{\le 1} u\|_{L^{2}_tL^\infty_x}\\
			&\ls2^{(1+\mu-19.9\delta) k}
			\big[\|\w{x}^{\mu-19.8\delta}P_k\partial u(0)\|_{L^2_x}
			+\int_{0}^{t}\big\|A_{\mu-19.8\delta}(\tau)P_k  G(u,\p u,\p_x\p u)(\tau)\big\|_{L^{2}_x}d\tau
			\big].
		\end{split}
	\end{equation}
	Summing over $k$ in \eqref{3.23}, one has
	\begin{equation}\label{YHCCC-101}
	\begin{split}
		&\sum_{k\ge -1}2^{(1+\delta) k}\|A_{\f{\mu}{2}-10\delta}(s,x)P_k\p^{\le 1} u\|_{L^{2}_tL^\infty_x}\\
		&\ls
	\sum_{k\ge -1}2^{(2+\mu -18.9\delta) k}
	\big[\|\w{x}^{\mu}P_k\partial u(0)\|_{L^2_x}
	+\int_{0}^{t}\left\|A_{\mu-19.8\delta}(\tau)
	 P_k  G(u,\p u,\p_x\p u)(\tau)\right\|_{L^{2}_x}d\tau\big].\\
	\end{split}
	\end{equation}
	Due to $(1+\tau+|x|)^{\nu}\approx(1+\tau)^{\nu}+\w{x}^{\nu}$ and $\w{x}^{2\mu-39.6\delta}\in A_2(\R^3)$,
then applying H\"older's inequality and Lemma \ref{lem:weighted bernstein}, we arrive at
	\begin{equation}\label{3.25}
	\begin{split}
		&\|A_{\mu-19.8\delta}(\tau,x)P_k G(u,\p u,\p_x\p u)\|_{L^2_x}\\
		&\ls
		\sum_{(k_1,k_2,k_3)\in \cY_k}\|A_{\mu-19.8\delta}(\tau,x)P_{k_1}\p^{\le 1} u\cdot P_{k_2}\p^{\le 1} u\cdot P_{k_3}\p_x^{\le 1}\p u\|_{L^2_x}\\
		&\ls
		\sum_{(k_1,k_2,k_3)\in \cY_k}
		\|A_{\f{\mu}{2}-11\delta}P_{k_1}\p^{\le 1} u\|_{L^{2+\f{1}{2\delta}}_x }
		\cdot\|A_{\f{\mu}{2}-11\delta}P_{k_2}\p^{\le 1} u\|_{L^{2+\f{1}{2\delta}}_x}
		\cdot 2^{k_3}\|A_{2.2\delta} P_{k_3}\p u\|_{L^{\f{2+8\delta}{1-4\delta}}_x}.
	\end{split}	
	\end{equation}
 Recalling the definition of $\cY_k$, one has
\begin{equation}\label{3.27}
\begin{split}
&\sum_{k\ge -1}2^{(2 +\mu -18.9\delta) k}\|A_{\mu-19.8\delta}
P_k G(u,\p u,\p_x\p u)\|_{L^2_x}\\
&\ls
\sum_{k_1,k_2,k_3\ge -1} 2^{(2 +\mu -18.8\delta) \max\{k_i\}}
\|A_{\f{\mu}{2}-11\delta}P_{k_1}\p^{\le 1} u\|_{L^{2+\f{1}{2\delta}}_x}
\cdot\|A_{\f{\mu}{2}-11\delta}P_{k_2}\p^{\le 1} u\|_{L^{2+\f{1}{2\delta}}_x}  \\
&\qquad\qquad\quad\cdot 2^{k_3}\|A_{2.2\delta} P_{k_3}\p u\|_{L^{\f{2+8\delta}{1-4\delta}}_x}\\
&\ls
\sum_{k\ge -1}2^{(3 +\mu-18.8\delta) k}\|A_{2.2\delta} P_{k}\p u\|_{L^{\f{2+8\delta}{1-4\delta}}_x}
\big(\sum_{l\ge-1}2^{(2 +\mu -18.8\delta) l} \|A_{\f{\mu}{2}-11\delta}P_{l}\p^{\le 1} u\|_{L^{2+\f{1}{2\delta}}_x}\big)^2.
\end{split}
\end{equation}
Then it follows from  Minkowski's inequality and H\"older's inequality that
\begin{equation}\label{YHCCC-102}
\begin{split}
	&\sum_{k\ge -1}2^{(2 +\mu -18.9\delta) k}\|A_{\mu-19.8\delta}
	P_k G(u,\p u,\p_x\p u)\|_{L^1_tL^2_x}\\
	&\ls
	\sum_{k\ge -1}2^{(3 +\mu -18.8\delta) k}\|A_{2.2\delta} P_{k}\p u\|_{L_t^{1+\f{1}{4\delta}}L^{\f{2+8\delta}{1-4\delta}}_x}
	\big(\sum_{l\ge-1}2^{(2 +\mu -18.8\delta) l} \|A_{\f{\mu}{2}-11\delta}P_{l}\p^{\le 1} u\|_{L_t^{2+8\delta}L^{2+\f{1}{2\delta}}_x}\big)^2.
\end{split}
\end{equation}
Substituting \eqref{YHCCC-102} into \eqref{YHCCC-101} yields \eqref{ineq:Estimate of Besov norm 2}.
\end{proof}

\subsection{Treatment on the cubic nonlinearity}
To estimate
	\[
	\sum_{k\ge -1}2^{(3 +\alpha -18.8\delta) k}\|A_{2.2\delta} P_{k}\p u\|_{L_t^{1+\f{1}{4\delta}}L^{\f{2+8\delta}{1-4\delta}}_x}
	\quad \text{and} \quad
	\sum_{l\ge-1}2^{(2 +\alpha -18.8\delta) l} \|A_{\f{\alpha}{2}-11\delta}P_{l}\p^{\le 1} u\|_{L_t^{2+8\delta}L^{2+\f{1}{2\delta}}_x}
	\]
	in Lemmas \ref{lem:Estimate of Besov norm 3}--\ref{lem:Estimate of Besov norm 4} below, it is required to deal with
	$\int_0^t \|G(u,\p u, \p_x\p u)(s)\|_{L^2_x} ds$. To this end, we now establish the following conclusion.
\begin{lemma}\label{lem:weight and Exponent Allocation for the Cubic Nonlinear Term}
	For $s>0,\ \beta\in(0,\f34),\ \delta>0$ and $t\ge0$, it holds that
	\begin{equation}\label{YHCCC-35}
		\begin{aligned}
			&\sum_{k\ge -1} 2^{sk}\int_{0}^{t}\big\|(1+\tau+|x|)^{2\beta}P_k G(u,\partial u,\p_x\p u)(\tau)\big\|_{L^2_x}d\tau\\
			&\ls
			\big( \sum_{k\ge -1}\|(1+s+|x|)^{\beta}P_k\p^{\le 1} u\|_{L^{2}_tL^\infty_x}\big)^2
			\|\p u\|_{L^\infty_tH^{s+1+\delta}_x}.
		\end{aligned}
	\end{equation}
\end{lemma}

\begin{proof}
	Set $A_\nu(s,x)=(1+s+|x|)^{\nu}$.	Due to $(1+\tau+|x|)^{\nu}\approx(1+\tau)^{\nu}+\w{x}^{\nu}$ and
$\w{x}^{4\beta}\in A_2(\R^3)$, then
	\begin{equation}\label{3.39}
		\begin{split}
			\|A_{2\beta}(\tau,x)P_k G(u,\p u,\p_x\p u)\|_{L^2_x}
			\ls
			\sum_{(k_1,k_2,k_3)\in \cY_k}\|A_{2\beta}(\tau,x)P_{k_1}\p^{\le 1} u\cdot P_{k_2}\p^{\le 1} u\cdot P_{k_3}\p_x^{\le 1}\p u\|_{L^2_x}.
		\end{split}	
	\end{equation}
When $-1\le k \le 3$, it follows from H\"older's inequality that
	\begin{equation}\label{3.98}
		\begin{split}
			\|A_{2\beta}P_k G(u,\partial u,\p_x\p u)\|_{L^2_x}
			\ls
			&
			\sum_{(k_1,k_2,k_3)\in\cY_k} \|A_{\beta}P_{k_1}\p^{\le 1} u\|_{L^\infty_x} \|A_{\beta}P_{k_2}\p^{\le 1} u\|_{L^\infty_x} \|P_{k_3}\p_x^{\le 1}\p u\|_{L^2_x}\\
			\ls&
			\big(\sum_{k\ge-1} \|A_{\beta}P_{k}\p^{\le 1} u\|_{L^\infty_x}\big)^2
			\|\p u\|_{H^{1+\delta}_x}.
		\end{split}
	\end{equation}
When $k\ge 4$, recalling the definition of $\cY_k$, we have $\max_{i=1,2,3}\{k_i\}\ge k-4\ge 0$.
Therefore, applying H\"older's inequality and Lemma \ref{lem:weighted bernstein} yields
\begin{equation}\label{3.99}
\begin{split}
&2^{sk}\|A_{2\beta}P_k G(u,\partial u,\p_x\p u)\|_{L^2_x}\\
&\ls
2^{-\delta k}\sum_{\substack{(k_1,k_2,k_3)\in\cY_k, \\ \max\{k_i\}=k_1}} 2^{(s+\delta)k_1 }\|P_{k_1}\p^{\le 1} u\|_{L^2_x} \|A_{\beta}P_{k_2}\p^{\le 1} u\|_{L^\infty_x} 2^{k_3}\|A_\beta P_{k_3}\p u\|_{L^\infty_x}\\
&\quad+2^{-\delta k}\sum_{\substack{(k_1,k_2,k_3)\in\cY_k, \\ \max\{k_i\}=k_3}} \|A_{\beta}P_{k_1}\p^{\le 1} u\|_{L^\infty_x} \|A_{\beta}P_{k_2}\p^{\le 1} u\|_{L^\infty_x} 2^{(s+1+\delta)k_3}\|P_{k_3}\p u\|_{L^2_x}\\
&\ls
2^{-\delta k}\sum_{\substack{(k_1,k_2,k_3)\in\cY_k, \\ \max\{k_i\}=k_1}} \big(2^{(s+1+\delta)k_1 }\|P_{k_1}\p u\|_{L^2_x}+2^{(s+\delta)k_1 }\|P_{k_1}\p_x u\|\big) \|A_{\beta}P_{k_2}\p^{\le 1} u\|_{L^\infty_x} \|A_\beta P_{k_3}\p u\|_{L^\infty_x}\\
&\quad+2^{-\delta k}\sum_{\substack{(k_1,k_2,k_3)\in\cY_k, \\ \max\{k_i\}=k_3}} \|A_{\beta}P_{k_1}\p^{\le 1} u\|_{L^\infty_x} \|A_{\beta}P_{k_2}\p^{\le 1} u\|_{L^\infty_x} 2^{(s+1+\delta)k_3}\|P_{k_3}\p u\|_{L^2_x}\\
&\ls
 2^{-\delta k}\big(\sum_{k\ge-1} \|A_{\beta}P_{k}\p^{\le 1} u\|_{L^\infty_x}\big)^2
\|\p u\|_{H^{s+1+2\delta}_x}.
\end{split}
\end{equation}
Summing over $k$ in \eqref{3.98} and \eqref{3.99}, by Minkowski's inequality and H\"older's inequality, one has
\begin{equation}\label{YHCCC-34}
\begin{aligned}
&\sum_{k\ge -1} 2^{sk}\int_{0}^{t}\big\|(1+\tau+|x|)^{2\beta}P_k G(u,\partial u,\p_x\p u)(\tau)\big\|_{L^2_x}d\tau\\
&\ls
\big(\sum_{k\ge -1}\|(1+s+|x|)^{\beta}P_k\p^{\le 1} u\|_{L^{2}_tL^\infty_x}\big)^2
\|\p u\|_{L^\infty_tH^{s+1+2\delta}_x}.
\end{aligned}
\end{equation}
Replacing $\delta$ by $\f\delta2$ in \eqref{YHCCC-34} yields  \eqref{YHCCC-35}.
\end{proof}

\subsection{Strichartz estimate II for Theorem \ref{thm:2}}
\begin{lemma}\label{lem:Estimate of Besov norm 3}
	For $\mu\in(0,1),\ \delta\in(0,10^{-4}\mu)$ and $t\ge0$, it holds that
	\begin{equation}\label{ineq:Estimate of Besov norm 3}
	\begin{split}
	&\sum_{k\ge -1}2^{(3 +\mu -18.8\delta) k}\|(1+s+|x|)^{2.2\delta} P_{k}\p u\|_{L_t^{1+\f{1}{4\delta}}L^{\f{2+8\delta}{1-4\delta}}_x}\\
	&\ls
	\sum_{k\ge -1}2^{(3+\mu -5.8\delta) k}
	\|\w{x}^{6\delta}P_k\partial u(0)\|_{L^2_x}
	+\big(\sum_{k\ge -1}\|(1+s+|x|)^{3\delta}P_{k}\p^{\le 1} u\|_{L^2_tL^\infty_x}\big)^2
	\|\p u\|_{L^\infty_tH^{4+\mu}_x}.
	\end{split}
	\end{equation}
\end{lemma}

\begin{proof}
Although the proof is similar to that of Lemma \ref{lem:Estimate of Besov norm 2}, we still give the full details for the reader's convenience due to the different treatment of the indices for the cubic terms. Set $A_\nu(s,x)=(1+s+|x|)^{\nu}$. Then from \eqref{eq:partial Pku} and Minkowski's inequality, we have
	\begin{equation*}
	\begin{split}
	&\|A_{2.2\delta}(s,x)P_k\p u\|_{L_t^{1+\f{1}{4\delta}}L^{\f{2+8\delta}{1-4\delta}}_x}\\
	&\ls
	\int_{0}^{t} \big\| A_{2.2\delta}(cs,x)P_{[[k]]}e^{\pm \sqrt{-1}cs|\nabla|} e^{\mp \sqrt{-1}c\tau|\nabla|} (\mathrm{Id},R)
	 P_k G(u,\partial u,\p_x\p u)(\tau)\big\|_{L^{1+\f{1}{4\delta}}_s([\tau,t];L^{\f{2+8\delta}{1-4\delta}}_x)}d\tau\\
	&\quad+\| A_{2.2\delta}(cs,x)P_{[[k]]}e^{\pm \sqrt{-1}cs|\nabla|}(\mathrm{Id},R) P_k\partial u(0)\|_{L_s^{1+\f{1}{4\delta}}L^{\f{2+8\delta}{1-4\delta}}_x}.
	\end{split}
	\end{equation*}
	Applying Theorem \ref{thm:Weighted Strichartz estimate} with $\beta_1 = \f{2.2}{4}+2.2\delta\in(0,1)$ and $\beta_2 = \f58+2.5\delta\in(\beta_1,\min\{\f32\beta_1,1\})$, one can obtain
	\begin{equation}\label{3.34}
	\begin{split}
	&\|A_{2.2\delta}(s,x)P_k\p u\|_{L_t^{1+\f{1}{4\delta}}L^{\f{2+8\delta}{1-4\delta}}_x}\\
	&\ls
	2^{ 13\delta k}
	\int_{0}^{t}\big\|\w{x}^{5\delta}P_ke^{\mp \sqrt{-1}c\tau|\nabla|} (\mathrm{Id},R) G(u,\partial u,\p_x\p u)(\tau)\big\|_{L^2_x}d\tau\\
	&\quad+2^{ 13\delta k}
	\|\w{x}^{5\delta}P_k(\mathrm{Id},R)\partial u(0)\|_{L^2_x}.
	\end{split}
	\end{equation}
	Note that by Lemma \ref{lem:riesz L2},
	\begin{equation}\label{3.35}
	\|\w{x}^{5\delta}(\mathrm{Id},R)P_k\partial u(0)\|_{L^2_x}
	\ls
	\|\w{x}^{5\delta}P_k\partial u(0)\|_{L^2_x}.
	\end{equation}
	On the other hand, by Corollary \ref{cor:Technical lemma} with $\beta_1=5\delta$ and $\beta_2=\delta$, one has
	\begin{equation*}
	\begin{split}
	&2^{ 13\delta k}
	\int_{0}^{t}\big\|\w{x}^{5\delta}P_ke^{\mp \sqrt{-1}c\tau|\nabla|} (\mathrm{Id},R) G(u,\partial u,\p_x\p u)(\tau)\big\|_{L^2_x}d\tau\\
	&\ls
	2^{ 13\delta k}
	\int_{0}^{t}(1+c\tau)^{6\delta}\big\|(\mathrm{Id},R) P_kG(u,\partial u,\p_x\p u)(\tau)\big\|_{L^2_x}d\tau\\
	&\quad+
	2^{ 13\delta k}
	\int_{0}^{t}\big\|\w{x}^{6\delta}(\mathrm{Id},R)P_k G(u,\partial u,\p_x\p u)(\tau)\big\|_{L^2_x}d\tau.
	\end{split}
	\end{equation*}
	This, together with Lemma \ref{lem:riesz L2}, yields
	\begin{equation}\label{3.37}
	\begin{split}
	&2^{ 13\delta k}
	\int_{0}^{t}\big\|\w{x}^{5\delta}P_ke^{\mp \sqrt{-1}c\tau|\nabla|} (\mathrm{Id},R) G(u,\partial u,\p_x\p u)(\tau)\big\|_{L^2_x}d\tau\\
	&\ls
	2^{ 13\delta k}
	\int_{0}^{t}\big\|A_{6\delta}(\tau,x)P_k G(u,\partial u,\p_x\p u)(\tau)\big\|_{L^2_x}d\tau.
	\end{split}
	\end{equation}
	Collecting \eqref{3.34}, \eqref{3.35} and \eqref{3.37}, we arrive at
	\begin{equation}\label{3.38}
		\begin{split}
		&\|A_{2.2\delta}P_k\p u\|_{L_t^{1+\f{1}{4\delta}}L^{\f{2+8\delta}{1-4\delta}}_x}
		\ls
		2^{ 13\delta k}
		\big[	\|\w{x}^{6\delta}P_k\partial u(0)\|_{L^2_x}	+
		\int_{0}^{t}\big\|A_{6\delta}(\tau,x)P_k G(u,\partial u,\p_x\p u)(\tau)\big\|_{L^2_x}d\tau\big].
		\end{split}
		\end{equation}
	Summing over $k$ in \eqref{3.38} yields
	\begin{equation}\label{YHCCC-36}
	\begin{split}
		&\sum_{k\ge -1}2^{(3+\mu-18.8\delta) k}\|A_{2.2\delta}P_k\p u\|_{L_t^{1+\f{1}{4\delta}}L^{\f{2+8\delta}{1-4\delta}}_x}\\
		&\ls
		\sum_{k\ge -1}2^{(3+\mu -5.8\delta) k}\big[
		\|\w{x}^{6\delta}P_k\partial u(0)\|_{L^2_x}
		+
		\int_{0}^{t}\big\|A_{6\delta}(\tau,x)P_k G(u,\partial u,\p_x\p u)(\tau)\big\|_{L^2_x}d\tau\big].
	\end{split}	
	\end{equation}
	In addition, it follows from Lemma \ref{lem:weight and Exponent Allocation for the Cubic Nonlinear Term} that
	\begin{equation}\label{YHCCC-37}
	\begin{split}
	\sum_{k\ge -1}2^{(3+\mu -5.8\delta) k}
	\|A_{6\delta}P_k G(u,\partial u,\p_x\p u)\|_{L^1_tL^2_x}
	\ls
	\big(\sum_{k\ge -1}\|A_{3\delta}P_{k}\p^{\le 1} u\|_{L^2_tL^\infty_x}\big)^2
	\|\p u\|_{L^\infty_tH^{4+\mu}_x}.
	\end{split}
	\end{equation}
	Then \eqref{ineq:Estimate of Besov norm 3} follows by substituting \eqref{YHCCC-37} into \eqref{YHCCC-36}.
\end{proof}

\subsection{Strichartz estimate III for Theorem \ref{thm:2}}
\begin{lemma}\label{lem:Estimate of Besov norm 4}
	For $\mu\in(0,1),\ \delta\in(0,10^{-4}\mu)$ and $t\ge0$, it holds that
	\begin{equation}\label{ineq:Estimate of Besov norm 4}
		\begin{split}
			&\sum_{k\ge -1}2^{(2 +\mu -18.8\delta) k}\|(1+s+|x|)^{\f{\mu}{2}-11\delta} P_{k}\p^{\le 1} u\|_{L_t^{2+8\delta}L^{2+\f{1}{2\delta}}_x}\\
			&\ls
			\sum_{k\ge -1}2^{(3+2\mu -2\delta) k}
			\|\w{x}^{\mu}P_k\partial u(0)\|_{L^2_x}
			+
			\big(\sum_{k\ge -1}\|(1+s+|x|)^{\f{\mu}{2}-10\delta}P_{k}\p^{\le 1} u\|_{L^2_tL^\infty_x}\big)^2
			\|\p u\|_{L^\infty_tH^{4+2\mu}_x}.
		\end{split}
	\end{equation}
\end{lemma}	

\begin{proof}
	Set $A_\nu(s,x)=(1+s+|x|)^{\nu}$. It follows from \eqref{eq:Pku}, \eqref{eq:partial Pku} and Minkowski's inequality that
	\begin{equation*}
		\begin{split}
			&\|A_{\f{\mu}{2}-11\delta}(s,x)P_k\p^{\le 1} u\|_{L_t^{2+8\delta}L^{2+\f{1}{2\delta}}_x}\\
			&\ls
			\sum_{\iota = 0,1}\int_{0}^{t} \big\| A_{\f{\mu}{2}-11\delta}(cs,x)P_{[[k]]}|\nabla|^{-\iota}e^{\pm \sqrt{-1}c(s-\tau)|\nabla|}  (\mathrm{Id},R)
			P_k G(u,\partial u,\p_x\p u)(\tau)\big\|_{L^{2+8\delta}_s([\tau,t];L^{2+\f{1}{2\delta}}_x)}d\tau\\
			&\quad+\sum_{\iota = 0,1}\| A_{\f{\mu}{2}-11\delta}(cs,x)P_{[[k]]}|\nabla|^{-\iota}e^{\pm \sqrt{-1}cs|\nabla|}(\mathrm{Id},R) P_k\partial u(0)\|_{L_s^{2+8\delta}L^{2+\f{1}{2\delta}}_x}.
		\end{split}
	\end{equation*}
	Applying Theorem \ref{thm:Weighted Strichartz estimate} with $\beta_1 = (\f{\mu}{2}-11\delta)(2+8\delta)=\mu-22\delta+4\mu\delta-88\delta^2\in(0,1)$ and $\beta_2 = (\f{\mu}{2}-10.5\delta)(2+8\delta)=\mu-21\delta+4\mu\delta-84\delta^2\in(\beta_1,\min\{\f32\beta_1,1\})$,
	and noting $2<2+8\delta<2+2\beta_2-\beta_1$, we can obtain
	\begin{equation}\label{3.134}
		\begin{split}
			&\|A_{\f{\mu}{2}-11\delta}(s,x)P_k\p^{\le 1} u\|_{L_t^{2+8\delta}L^{2+\f{1}{2\delta}}_x}\\
			&\ls
			2^{ (1+\mu-21\delta) k}
			\int_{0}^{t}\big\|\w{x}^{\mu-21\delta}P_ke^{\mp \sqrt{-1}c\tau|\nabla|} (\mathrm{Id},R) G(u,\partial u,\p_x\p u)(\tau)\big\|_{L^2_x}d\tau\\
			&\quad+2^{ (1+\mu-21\delta) k}
			\|\w{x}^{\mu-21\delta}P_k(\mathrm{Id},R)\partial u(0)\|_{L^2_x}.
		\end{split}
	\end{equation}
	In addition, by Lemma \ref{lem:riesz L2}, one has
	\begin{equation}\label{3.135}
		\|\w{x}^{\mu-21\delta}(\mathrm{Id},R)P_k\partial u(0)\|_{L^2_x}
		\ls
		\|\w{x}^{\mu-21\delta}P_k\partial u(0)\|_{L^2_x}.
	\end{equation}
	On the other hand, by Corollary \ref{cor:Technical lemma} with $\beta_1 =\mu-21\delta$ and $\beta_2 = \delta$,
	\begin{equation*}
		\begin{split}
			&2^{ (1+\mu-21\delta) k}
			\int_{0}^{t}\big\|\w{x}^{\mu-21\delta}P_ke^{\mp \sqrt{-1}c\tau|\nabla|} (\mathrm{Id},R) G(u,\partial u,\p_x\p u)(\tau)\big\|_{L^2_x}d\tau\\
			&\ls
			2^{ (1+\mu-21\delta) k}
			\int_{0}^{t}(1+c\tau)^{\mu-20\delta}\big\|(\mathrm{Id},R) P_kG(u,\partial u,\p_x\p u)(\tau)\big\|_{L^2_x}d\tau\\
			&\quad+
			2^{ (1+\mu-21\delta) k}
			\int_{0}^{t}\big\|\w{x}^{\mu-20\delta}(\mathrm{Id},R)P_k G(u,\partial u,\p_x\p u)(\tau)\big\|_{L^2_x}d\tau.
		\end{split}
	\end{equation*}
	Together with Lemma \ref{lem:riesz L2}, this yields
	\begin{equation}\label{3.137}
		\begin{split}
			&2^{ (1+\mu-21\delta) k}
			\int_{0}^{t}\big\|\w{x}^{\mu-21\delta}P_ke^{\mp \sqrt{-1}c\tau|\nabla|} (\mathrm{Id},R) G(u,\partial u,\p_x\p u)(\tau)\big\|_{L^2_x}d\tau\\
			&\ls
			2^{ (1+\mu-21\delta) k}
			\int_{0}^{t}\left\|A_{\mu-20\delta}(\tau,x)P_k G(u,\partial u,\p_x\p u)(\tau)\right\|_{L^2_x}d\tau.
		\end{split}
	\end{equation}
	Combining \eqref{3.134}, \eqref{3.135} and \eqref{3.137}, we have
	\begin{equation}\label{3.138}
		\begin{split}
			&\|A_{\f{\mu}{2}-11\delta}P_k\p^{\le 1} u\|_{L_t^{2+8\delta}L^{2+\f{1}{2\delta}}_x}\\
			&\ls
			2^{ (1+\mu-21\delta) k}
			\big[\|\w{x}^{\mu-20\delta}P_k\partial u(0)\|_{L^2_x}+
			\int_{0}^{t}\big\|A_{\mu-20\delta}(\tau,x)P_k G(u,\partial u,\p_x\p u)(\tau)\big\|_{L^2_x}d\tau\big].
		\end{split}
	\end{equation}
Summing over $k$ in \eqref{3.138} yields
	\begin{equation}\label{YHCCC-38}
			\begin{split}
				&\sum_{k\ge -1}2^{(2+\mu-18.8\delta) k}\|A_{\f{\mu}{2}-11\delta}P_k\p^{\le 1} u\|_{L_t^{2+8\delta}L^{2+\f{1}{2\delta}}_x}\\
				&\ls
				\sum_{k\ge -1}2^{(3+2\mu -2\delta) k}\big[
				\|\w{x}^{\mu}P_k\partial u(0)\|_{L^2_x}
				+
				\int_{0}^{t}\big\|A_{\mu-20\delta}(\tau,x)P_k G(u,\partial u,\p_x\p u)(\tau)\big\|_{L^2_x}d\tau\big].
			\end{split}	
	\end{equation}
	In addition, it follows from Lemma \ref{lem:weight and Exponent Allocation for the Cubic Nonlinear Term} that
	\begin{equation}\label{YHCCC-39}
		\begin{split}
			\sum_{k\ge -1}2^{(3+2\mu -2\delta) k}
			\|A_{2\mu}P_k G(u,\partial u,\p_x\p u)\|_{L^1_tL^2_x}
			\ls
			\big(\sum_{k\ge -1}\|A_{\f{\mu}{2}-10\delta}P_{k}\p^{\le 1} u\|_{L^2_tL^\infty_x}\big)^2
			\|\p u\|_{L^\infty_tH^{4+2\mu}_x}.
		\end{split}
	\end{equation}
	Substituting \eqref{YHCCC-39} into \eqref{YHCCC-38} yields \eqref{ineq:Estimate of Besov norm 4}.
\end{proof}

\section{Proof of Theorem \ref{thm:1}}\label{Section 5}

Suppose that for $\delta\in(0,\f{1}{3})$, $t>0$ and integer $N\ge 4$,
\begin{equation}\label{bootstrap1}
\|\p u\|_{L^\infty([0,t];H^N)}\le \varepsilon_1\le1,	
\end{equation}
\begin{equation}\label{bootstrap2}
	\sum_{k\ge-1}2^{(1+\delta) k}\ln^{-1}(e+t)\|P_k\p^{\le 1} u\|_{L^{2}([0,t];L^\infty) }\le \varepsilon_1\le1.
\end{equation}
We shall improve the constant $\varepsilon_1$ to $\f{\varepsilon_1}{2}$ in \eqref{bootstrap1} and \eqref{bootstrap2}.
Firstly, it follows from Lemma \ref{lem:Energy estimate} that
\begin{equation}\label{improve1}
\begin{split}
\|\p u\|_{L^\infty([0,t];H^N)}
&\ls\|\p u(0)\|_{H^N(\R^3)}+\big(\sum_{k\ge-1}2^{(1+\delta) k}\|P_k\p^{\le1 } u\|_{L^{2}([0,t];L^\infty)}\big)^2\|\p u\|_{L^\infty([0,t];H^N)}\\
&\ls \varepsilon +\varepsilon_1^3\ln^2(e+t).
\end{split}	
\end{equation}
In addition, by Lemma \ref{lem:Estimate of Besov norm}, one has
\begin{equation}\label{improve2}
\sum_{k\ge-1}2^{(1+\delta) k}\ln^{-1}(e+t)\|P_k\p^{\le1 } u\|_{L^{2}([0,t];L^\infty)}\ls \varepsilon +\varepsilon_1^3\ln^2(e+t).
\end{equation}
From \eqref{improve1} and \eqref{improve2}, there is $C_1 \ge 1$ such that
\begin{equation}\label{improve3}
\|\p u\|_{L^\infty([0,t];H^N)}\le C_1(\varepsilon +\varepsilon_1^3\ln^2(e+t)),
\end{equation}
\begin{equation}\label{improve4}
	\sum_{k\ge-1}2^{(1+\delta) k}\ln^{-1}(e+t) \|P_k\p^{\le1 } u\|_{L^{2}([0,t];L^\infty) }\le C_1( \varepsilon +\varepsilon_1^3\ln^2(e+t)).
\end{equation}
Choosing $\varepsilon_1 = 4C_1 \varepsilon,\ T_\ve = e^{\kappa_0\ve^{-1}}-e$ and $\kappa_0 = \f{1}{4C_1}$ in (\ref{improve3}) and (\ref{improve4}),
one obtains
\begin{equation*}
\|\p u\|_{L^\infty([0,t];H^N)}\le \f{\varepsilon_1}{2}	
\end{equation*}
and
\begin{equation*}
\sum_{k\ge-1}2^{(1+\delta) k}\ln^{-1}(e+t)\|P_k\p^{\le 1} u\|_{L^{2}([0,t];L^\infty)}\le \f{\varepsilon_1}{2}.
\end{equation*}
This, together with the local existence of classical solution to \eqref{eq:wave} and the continuity argument, ensures that
	\eqref{eq:wave} admits a unique solution $u\in C([0,T_\ve];H^{N+1}(\R^3))\cap C^1([0,T_\ve];H^{N}(\R^3))$.

In addition, if $G(u, \p u, \p^2 u)$ is independent of $u$ (see \eqref{eq:C independent u}),
then we suppose that for integer $N\ge 4$, $\delta\in(0,\f{1}{3})$ and $t\ge0$,
	\begin{equation}\label{bootstrap11}
	\|\p u\|_{L^\infty([0,t];H^N)}\le \varepsilon_1\le1,	
	\end{equation}
	\begin{equation}\label{bootstrap12}
		\sum_{k\ge-1}2^{(1+\delta) k}\ln^{-\f12}(e+t)\|P_k\p u\|_{L^{2}([0,t];L^\infty)}\le \varepsilon_1\le1.
	\end{equation}
By the same argument as above, the lifespan $T_\varepsilon$ can be improved to $T_\varepsilon=e^{\kappa_0 \varepsilon^{-2}}-e$.

\section{Proof of Theorem \ref{thm:2}}\label{Section 6}

\subsection{Uniform energy estimates}\label{Sec.6.1}

Suppose that for $\mu\in(0,1)$, $\delta=10^{-6}\mu$, $t>0$ and integer $N\ge \lceil4+2\mu\rceil$,
\begin{equation}\label{bootstrap3}
\|\p u\|_{L^\infty([0,t];H^N)}\le \varepsilon_2\le1,	
\end{equation}
\begin{equation}\label{bootstrap4}
\sum_{k\ge-1}2^{(1+\delta) k}\|(1+s+|x|)^{\f{\mu}{2}-10\delta}P_k\p^{\le 1} u\|_{L^{2}([0,t];L^\infty)}\le \varepsilon_2\le1,
\end{equation}
\begin{equation}\label{bootstrap5}
\sum_{k\ge-1}2^{(3 +\mu -18.8\delta) k} \|(1+s+|x|)^{2.2\delta}P_{k}\p u\|_{L^{1+\f{1}{4\delta}}([0,t];L^{\f{2+8\delta}{1-4\delta}})}\le \varepsilon_2\le1,
\end{equation}
\begin{equation}\label{bootstrap6}
	\sum_{k\ge-1}2^{(2 +\mu -18.8\delta) k} \|(1+s+|x|)^{\f{\mu}{2}-11\delta}P_{k}\p^{\le 1} u\|_{L^{2+8\delta}([0,t];L^{2+\f{1}{2\delta}})}\le \varepsilon_2\le1.
\end{equation}
We shall improve the constant $\varepsilon_2$ in \eqref{bootstrap3}--\eqref{bootstrap6} to $\f{\varepsilon_2}{2}$.
Firstly, by Lemma \ref{lem:Energy estimate}, one has
\begin{equation}\label{improve21}
\begin{split}
\|\p u\|_{L^\infty_tH^N_x}
&\ls \|\p u(0)\|_{H^N_x}+\big(\sum_{k\ge-1}
2^{(1+\delta) k}\|(1+s+|x|)^{\f{\mu}{2}-10\delta}P_k\p^{\le 1} u\|_{L^{2}_t L^\infty_x}\big)^2\|\p u\|_{L^\infty_tH^N_x}\\
&\ls \varepsilon +\varepsilon_2^3.
\end{split}	
\end{equation}
In addition, it follows from Lemma \ref{lem:Estimate of Besov norm 2} that
\begin{equation}\label{improve22}
	\begin{split}
		&\sum_{k\ge -1}2^{(1+\delta) k}\|(1+s+|x|)^{\f{\mu}{2}-10\delta}P_k\p^{\le 1} u\|_{L^{2}_tL^\infty_x}\\
		\ls&
		\sum_{k\ge -1}2^{(2 +\mu -18.9\delta) k}\|\w{x}^{\mu}P_k\partial u(0)\|_{L^2_x}\\
		&+
		\sum_{k\ge -1}2^{(3 +\mu -18.8\delta) k}\|(1+s+|x|)^{2.2\delta} P_{k}\p u\|_{L_t^{1+\f{1}{4\delta}}L^{\f{2+8\delta}{1-4\delta}}_x}\\
		&\quad\times
		\big(\sum_{l\ge-1}2^{(2 +\mu-18.8\delta) l} \|(1+s+|x|)^{\f{\mu}{2}-11\delta}P_{l}\p^{\le 1} u\|_{L_t^{2+8\delta}L^{2+\f{1}{2\delta}}_x}\big)^2.
	\end{split}
\end{equation}
Due to Lemma \ref{lem:weighted bernstein} and $\w{x}^{2\mu}\in A_2(\R^3)$, one has
\begin{equation*}
	\sum_{k\ge -1}2^{(2 +\mu -18.9\delta) k}\|\w{x}^{\mu}P_k\partial u(0)\|_{L^2_x}
	\ls
	\sum_{|a|\le 3}\|\w{x}^{\mu}\p^a_x\partial u(0)\|_{L^2_x}\ls \ve.
\end{equation*}
Therefore,
\begin{equation}\label{5.7}
\sum_{k\ge-1}2^{(1+\delta) k}\|(1+s+|x|)^{\f{\mu}{2}-10\delta}P_k\p^{\le 1} u\|_{L^{2}([0,t];L^\infty)}\ls \ve+\ve^3_2.
\end{equation}
Thirdly, we estimate $\sum_{k\ge -1}2^{(3 +\mu -18.8\delta) k}\|(1+s+|x|)^{2.2\delta} P_{k}\p u\|_{L_t^{1+\f{1}{4\delta}}L^{\f{2+8\delta}{1-4\delta}}_x}$.
By Lemma \ref{lem:Estimate of Besov norm 3}, we have
\begin{equation}\label{improve23}
\begin{split}
	&\sum_{k\ge -1}2^{(3 +\mu -18.8\delta) k}\|(1+s+|x|)^{2.2\delta} P_{k}\p u\|_{L_t^{1+\f{1}{4\delta}}L^{\f{2+8\delta}{1-4\delta}}_x}\\
	&\ls
	\sum_{k\ge -1}2^{(3+\mu -5.8\delta) k}
	\|\w{x}^{6\delta}P_k\partial u(0)\|_{L^2_x}
	+
	\big(\sum_{k\ge -1}\|(1+s+|x|)^{3\delta}P_{k}\p^{\le 1} u\|_{L^2_tL^\infty_x}\big)^2
	\|\p u\|_{L^\infty_tH^{4+\mu}_x}.
\end{split}
\end{equation}
Note that by Lemma \ref{lem:weighted bernstein} and $\w{x}^{12\delta}\in A_2(\R^3)$,
\begin{equation*}
\sum_{k\ge -1}2^{(3+\mu -5.8\delta) k}
\|\w{x}^{6\delta}P_k\partial u(0)\|_{L^2_x}
\ls
\sum_{|a|\le 4}\|\w{x}^{\mu}\p^a_x\partial u(0)\|_{L^2_x}\ls \ve.
\end{equation*}
Hence,
\begin{equation}\label{5.10}
\sum_{k\ge -1}2^{(3 +\mu -18.8\delta) k}\|(1+s+|x|)^{2.2\delta} P_{k}\p u\|_{L_t^{1+\f{1}{4\delta}}L^{\f{2+8\delta}{1-4\delta}}_x}\ls \ve+\ve^3_2.
\end{equation}
Finally, we estimate $\sum_{k\ge-1}2^{(2 +\mu-18.8\delta) k} \|(1+s+|x|)^{\f{\mu}{2}-11\delta}P_{k}\p^{\le 1} u\|_{L_t^{2+8\delta}L^{2+\f{1}{2\delta}}_x}$.
It follows from Lemma \ref{lem:Estimate of Besov norm 4} that
\begin{equation*}
	\begin{split}
		&\sum_{k\ge -1}2^{(2 +\mu -18.8\delta) k}\|(1+s+|x|)^{\f{\mu}{2}-11\delta} P_{k}\p^{\le 1} u\|_{L_t^{2+8\delta}L^{2+\f{1}{2\delta}}_x}\\
		&\ls
		\sum_{k\ge -1}2^{(3+2\mu -2\delta) k}
		\|\w{x}^{\mu}P_k\partial u(0)\|_{L^2_x}
		+
		\big(\sum_{k\ge -1}\|(1+s+|x|)^{\f{\mu}{2}-10\delta}P_{k}\p^{\le 1} u\|_{L^2_tL^\infty_x}\big)^2
		\|\p u\|_{L^\infty_tH^{4+2\mu}_x}.
	\end{split}
\end{equation*}
By Lemma \ref{lem:weighted bernstein} and $\w{x}^{2\mu}\in A_2(\R^3)$, one has
\begin{equation*}
	\sum_{k\ge -1}2^{(3+2\mu-2\delta) k}
	\|\w{x}^{\mu}P_k\partial u(0)\|_{L^2_x}
	\ls
	\sum_{|a|\le 5}\|\w{x}^{\mu}\p^a_x\partial u(0)\|_{L^2_x}\ls \ve.
\end{equation*}
Therefore,
\begin{equation}\label{5.11}
	\sum_{k\ge -1}2^{(2 +\mu -18.8\delta) k}\|(1+s+|x|)^{\f{\mu}{2}-11\delta} P_{k}\p^{\le 1} u\|_{L_t^{2+8\delta}L^{2+\f{1}{2\delta}}_x}\ls \ve+\ve^3_2.
\end{equation}
It is pointed out that the condition $N \ge \lceil 4+2\mu \rceil$ is employed in \eqref{5.11}.
From \eqref{improve21}, \eqref{5.7}, \eqref{5.10} and \eqref{5.11}, there is $C_2 \ge 1$ such that
\begin{equation}\label{improve24}
\|\p u\|_{L^\infty([0,t];H^N)}\le C_2(\varepsilon +\varepsilon_2^3),
\end{equation}
\begin{equation}\label{improve25}
\sum_{k\ge-1}2^{(1+\delta) k}\|(1+s+|x|)^{\f{\mu}{2}-10\delta}P_k\p^{\le 1} u\|_{L^{2}([0,t];L^\infty)}
\le C_2(\varepsilon +\varepsilon_2^3),
\end{equation}
\begin{equation}\label{improve26}
\sum_{k\ge-1}2^{(3 +\mu -18.8\delta) k} \|(1+s+|x|)^{2.2\delta}P_{k}\p u\|_{L^{1+\f{1}{4\delta}}([0,t];L^{\f{2+8\delta}{1-4\delta}})}
\le C_2(\varepsilon +\varepsilon_2^3),
\end{equation}
\begin{equation}\label{improve27}
	\sum_{k\ge-1}2^{(2 +\mu -18.8\delta) k} \|(1+s+|x|)^{\f{\mu}{2}-11\delta}P_{k}\p^{\le 1} u\|_{L^{2+8\delta}([0,t];L^{2+\f{1}{2\delta}})}
	\le C_2(\varepsilon +\varepsilon_2^3).
\end{equation}
Choosing $\varepsilon_2 = 4C_2 \varepsilon_0$ and $\ve_0 = \left(\f{1}{64C^3_2}\right)^{\f12}$ in \eqref{improve24}--\eqref{improve27}, we can obtain
\begin{equation*}
\|\p u\|_{L^\infty([0,t];H^N)}\le \f{\varepsilon_2}2,	
\end{equation*}
\begin{equation*}
\sum_{k\ge-1}2^{(1+\delta) k}\|(1+s+|x|)^{\f{\mu}{2}-10\delta}P_k\p^{\le 1} u\|_{L^{2}([0,t];L^\infty)}
\le \f{\varepsilon_2}2,	
\end{equation*}
\begin{equation*}
\sum_{k\ge-1}2^{(3 +\mu -18.8\delta) k} \|(1+s+|x|)^{2.2\delta}P_{k}\p u\|_{L^{1+\f{1}{4\delta}}([0,t];L^{\f{2+8\delta}{1-4\delta}})}
\le \f{\varepsilon_2}2,	
\end{equation*}
\begin{equation*}
\sum_{k\ge-1}2^{(2 +\mu -18.8\delta) k} \|(1+s+|x|)^{\f{\mu}{2}-11\delta}P_{k}\p^{\le 1} u\|_{L^{2+8\delta}([0,t];L^{2+\f{1}{2\delta}})}
\le\f{\varepsilon_2}2.
\end{equation*}
This, together with the local existence of classical solution to \eqref{eq:wave} and the continuity argument, ensures that
\eqref{eq:wave} admits a unique solution $u\in C([0,+\infty);H^{N+1}(\R^3))\cap C^1([0,+\infty);H^{N}(\R^3))$.

\subsection{Scattering of solution $u$}

From \eqref{improve24}--\eqref{improve27}, we have
\begin{equation}\label{YHCCC-103}
\|\p u\|_{L^\infty([0,\infty);H^N)}\le C\varepsilon,	
\end{equation}
\begin{equation}\label{5.14}
\sum_{k\ge-1}2^{(1+\delta) k}\|(1+t+|x|)^{\f{\mu}{2}-10\delta} P_k\p^{\le 1} u\|_{L^{2}([0,\infty);L^\infty)}\le C\varepsilon,
\end{equation}
\begin{equation}\label{6.25}
\sum_{k\ge-1}2^{(3 +\mu -18.8\delta) k} \|(1+s+|x|)^{2.2\delta}P_{k}\p u\|_{L^{1+\f{1}{4\delta}}([0,\infty);L^{\f{2+8\delta}{1-4\delta}})}
\le C\varepsilon,
\end{equation}
\begin{equation}\label{6.26}
\sum_{k\ge-1}2^{(2 +\mu -18.8\delta) k} \|(1+s+|x|)^{\f{\mu}{2}-11\delta}P_{k}\p^{\le 1} u\|_{L^{2+8\delta}([0,\infty);L^{2+\f{1}{2\delta}})}
\le C\varepsilon.
\end{equation}
Define
\begin{equation*}
\begin{split}
u_0^{\infty,i}(x) &= u_0^i(x) + \int_0^\infty \frac{\sin(-c_is|\nabla|\big)}{c_i|\nabla|} G^i(u,\p u,\p_x\p u)(s)ds, \\
u_1^{\infty,i}(x) &= u_1^i(x) + \int_0^\infty \cos(-c_is|\nabla|\big) G^i(u,\p u,\p_x\p u)(s)ds. \\
\end{split}
\end{equation*}
Then
\begin{equation*}
	u^{\infty,i}(t,x) = \cos(c_it|\nabla|)u_0^i(x) + \frac{\sin(c_it|\nabla|)}{c_i|\nabla|}u_1^i(x) + \int_0^\infty \frac{\sin\big(c_i(t-s)|\nabla|\big)}{c_i|\nabla|} G^i(u,\p u,\p_x\p u)(s)ds
\end{equation*}
is a solution of $\square_{c_i} u^i = 0$ for $i = 1,\cdots,m$ with initial data $(u_0^\infty,u_1^\infty)$ at $t = 0$.
By applying \eqref{eq:partial Pku} without $P_k$, we have
\begin{equation*}
\begin{split}
&\|\p(u(t)-u^\infty(t))\|_{H^{N-1}(\R^3)}\\
\ls &\|\int_t^\infty  e^{\pm \sqrt{-1}c(t-s)|\nabla|} (\mathrm{Id},R)  G(u,\partial u,\p_x\p u)(s)ds\|_{H^{N-1}(\R^3)}\\
\ls &\int_t^\infty  \|G(u,\partial u,\p_x\p u)(s)\|_{H^{N-1}(\R^3)}ds\\
\ls & \|\p u\|_{L^\infty([t,\infty);H^N)}\big(\sum_{k\ge-1}2^{(1+\delta) k}\|(1+t+|x|)^{\f{\mu}{2}-10\delta}P_k\p^{\le 1} u\|_{L^{2}([t,\infty); L^\infty)}\big)^2.
\end{split}	
\end{equation*}
Together with \eqref{5.14}, this yields
\begin{equation*}
\lim_{t \to +\infty} \|\partial (u(t,\cdot) - u^\infty(t,\cdot))\|_{H^{N-1}(\mathbb{R}^3)} = 0.	
\end{equation*}

\subsection{Decay properties of solution $u$ as $t \to \infty$}

By \eqref{eq:Pku} and \eqref{eq:partial Pku}, one arrives at
\begin{equation}\label{YHCCC-41}
\begin{split}
&\|P_k\p^{\le 1} u(t)\|_{L^\infty_x}\\
&\ls
\sum_{\iota=0,1}\int_{0}^{t}\big\|P_{[[k]]}|\nabla|^{-\iota}e^{\pm \sqrt{-1}ct|\nabla|}e^{\mp \sqrt{-1}c\tau|\nabla|} (\mathrm{Id},R)
P_kG(u,\p u,\p_x\p u)(\tau)\big\|_{L^\infty_x}d\tau\\
&\quad+\sum_{\iota=0,1}\|P_{[[k]]}|\nabla|^{-\iota}e^{\pm \sqrt{-1}ct|\nabla|} (\mathrm{Id},R)P_k\partial u(0)\|_{L^\infty_x}.
\end{split}
\end{equation}
Applying Lemma \ref{cor:3.8} with $\delta=10^{-6}\mu$ and $\theta = \frac{\mu-21\delta}{1+2\delta}$, it follows from \eqref{YHCCC-41} that
\begin{equation}\label{6.34}
\begin{split}
&\|P_k\p^{\le 1} u(t)\|_{L^\infty_x}\\
&\ls
2^{(\f32+\mu) k}(1+t)^{-\theta}
\int_{0}^{t}\big\|\w{x}^{\mu-21\delta}P_ke^{\mp \sqrt{-1}c\tau|\nabla|} (\mathrm{Id},R) G(u,\p u,\p_x\p u)(\tau)\big\|_{L^2_x}d\tau\\
&\quad+2^{(\f32+\mu) k}(1+t)^{-\theta}
\|\w{x}^{\mu}P_k(\mathrm{Id},R)\partial u(0)\|_{L^2_x}.
\end{split}
\end{equation}
Note that by Lemma \ref{lem:riesz L2},
\begin{equation}\label{6.35}
\|\w{x}^{\mu}(\mathrm{Id},R)P_k\partial u(0)\|_{L^2_x}
\ls
\|\w{x}^{\mu}P_k\partial u(0)\|_{L^2_x}.
\end{equation}
On the other hand, by Corollary \ref{cor:Technical lemma}, one has
\begin{equation*}
\begin{split}
&\int_{0}^{t}\big\|\w{x}^{\mu-21\delta}P_ke^{\mp \sqrt{-1}c\tau|\nabla|} (\mathrm{Id},R) G(u,\partial u,\p_x\p u)(\tau)\big\|_{L^2_x}d\tau\\
&\ls\int_{0}^{t}(1+c\tau)^{\mu-20\delta}\big\|(\mathrm{Id},R) P_kG(u,\partial u,\p_x\p u)(\tau)\big\|_{L^2_x}d\tau\\
&\quad+\int_{0}^{t}\big\|\w{x}^{\mu-20\delta}(\mathrm{Id},R)P_k G(u,\partial u,\p_x\p u)(\tau)\big\|_{L^2_x}d\tau.
\end{split}
\end{equation*}
Together with Lemma \ref{lem:riesz L2}, this yields
\begin{equation}\label{6.37}
\begin{split}
&\int_{0}^{t}\big\|\w{x}^{\mu-21\delta}P_ke^{\mp \sqrt{-1}c\tau|\nabla|} (\mathrm{Id},R) G(u,\partial u,\p_x\p u)(\tau)\big\|_{L^2_x}d\tau\\
&\ls\int_{0}^{t}\big\|A_{\mu-20\delta}(\tau,x)P_k G(u,\partial u,\p_x\p u)(\tau)\big\|_{L^2_x}d\tau.
\end{split}
\end{equation}
Collecting \eqref{6.34}, \eqref{6.35} and \eqref{6.37}, we have
\begin{equation}\label{6.38}
\begin{split}
&(1+t)^{\mu^-}\|P_k\p^{\le 1} u(t)\|_{L^\infty_x}\\
&\ls 2^{(\f32+\mu) k}\big[\|\w{x}^{\mu}P_k\partial u(0)\|_{L^2_x}+
\int_{0}^{t}\big\|A_{\mu-20\delta}(\tau,x)P_k G(u,\partial u,\p_x\p u)(\tau)\big\|_{L^2_x}d\tau\big].
\end{split}
\end{equation}
By summing over $k$ in \eqref{6.38}, utilizing Lemma \ref{lem:weighted bernstein}, Lemma \ref{lem:weight and Exponent Allocation for the Cubic Nonlinear Term} and \eqref{YHCCC-103}--\eqref{5.14}, one arrives at
\begin{equation*}
\begin{split}
&(1+t)^{\mu^-}\|\p^{\le 1} u(t)\|_{L^\infty_x}\\
&\ls\sum_{k\ge -1}2^{(\f32+\mu) k}\big[
\|\w{x}^{\mu}P_k\partial u(0)\|_{L^2_x}
+\int_{0}^{t}\big\|A_{\mu-20\delta}(\tau,x)P_k G(u,\partial u,\p_x\p u)(\tau)\big\|_{L^2_x}d\tau\big]\\
&\ls\sum_{|a|\le 3}\|\w{x}^{\mu}\p^a_x\partial u(0)\|_{L^2_x}+\ve^3\\
&\ls\ve.
\end{split}	
\end{equation*}
Then the proof of Theorem \ref{thm:2} is completed.

\appendix
\section{Introduction to $A_2$ weights}\label{section a}
For the reader's convenience, in this section we give a brief introduction to some $A_2$ weight inequalities.
For further properties of $A_p$ weights, we refer the reader to \cite[Chapter 7]{Grafakos} and \cite[Chapter V]{Stein}
\begin{definition}
A non-negative weight function $w$ is said to be of class $A_2$ if
\begin{equation}\label{def:A_2 weight}
[w]_{A_2}=\sup_{\text{all cubes $Q$ in } \mathbb{R}^3}\big(\frac{1}{|Q|} \int_Q w(x) dx\big)\big(\frac{1}{|Q|} \int_Q w(x)^{-1} dx\big)<\infty. 	
\end{equation}

\end{definition}
\begin{lemma}\label{lem:w{x} in A2}
	$|x|^\alpha,\ \w{x}^{\alpha} \in A_2(\R^3)$ if and only if $\alpha\in(-3,3)$.
\end{lemma}
\begin{proof}
 For $|x|^\alpha$, the related result is established in \cite[Chapter V, p.~218]{Stein}.

For $\langle x \rangle^\alpha$,
let $Q$ be a cube with center $x_0$ and diagonal length $2r$ (so $Q \subseteq B(x_0, r)$), we consider the following two cases.

 \vskip 0.2 true cm
 \noindent\textbf{Case 1.} $r \leq |x_0|/2$
 \vskip 0.1 true cm
 For all $x \in Q$, we have $\frac{1}{2}|x_0| \leq |x| \leq \frac{3}{2}|x_0|$, hence $\langle x \rangle \approx \langle x_0 \rangle$. Thus
 \[
 \f{1}{|Q|}\int_Q \langle x \rangle^\alpha \, dx \approx \langle x_0 \rangle^\alpha , \quad \f{1}{|Q|}\int_Q \langle x \rangle^{-\alpha} \, dx \approx \langle x_0 \rangle^{-\alpha},
 \]
 and their product is bounded by a uniform constant, which fulfills the $A_2$ condition.

 \vskip 0.2 true cm
 \noindent\textbf{Case 2.} $r \geq |x_0|/2$
\vskip 0.1 true cm
 In this case, $Q \subset B(x_0, r) \subset B(0, 3r)$ holds.
 Due to $\alpha \in (-3, 3)$, one then has
 \[
 \int_Q \langle x \rangle^\alpha \, dx \lesssim (3r)^{\alpha+3}, \quad \int_Q \langle x \rangle^{-\alpha} \, dx \lesssim (3r)^{-\alpha+3}.
 \]
 Multiplying these gives an upper bound $C r^6$ and subsequently being divided by $|Q|^2 \approx r^6$ yield the desired uniform bound.

 If $\alpha \notin (-3, 3)$, it is easy to verify $[\w{x}^\alpha]_{A_2} = \infty$, which completes the proof of Lemma \ref{lem:w{x} in A2}.
\end{proof}

\begin{definition}\label{def:7.4.1}
Let $0 < \delta, A < \infty$. A function $K(x,y)$ defined for $x,y \in \mathbb{R}^n$ with $x \neq y$ is called a standard kernel (with constants $\delta$ and $A$) if
\begin{equation}\label{eq:7.4.1}
|K(x,y)| \leq \frac{A}{|x-y|^n}, \quad x \neq y,
\end{equation}
and whenever $|x-x'| \leq \frac12\max\left(|x-y|,|x'-y|\right)$, it holds
\begin{equation}\label{eq:7.4.2}
|K(x,y) - K(x',y)| \leq \frac{A|x-x'|^\delta}{\left(|x-y|+|x'-y|\right)^{n+\delta}};
\end{equation}
when $|y-y'| \leq \frac12\max\left(|x-y|,|x-y'|\right)$, one has
\begin{equation}\label{eq:7.4.3}
|K(x,y) - K(x,y')| \leq \frac{A|y-y'|^\delta}{\left(|x-y|+|x-y'|\right)^{n+\delta}}.
\end{equation}
The class of all kernels that satisfy \eqref{eq:7.4.1}-\eqref{eq:7.4.3} is denoted by $SK(\delta,A)$.
\end{definition}

\begin{definition}\label{def:7.4.2}
Let $0 < \delta, A < \infty$ and $K\in SK(\delta,A)$. A Calder\'{o}n-Zygmund operator $T$ associated with $K$ is defined as
\begin{equation}\label{eq:7.4.5}
T(f)(x) = \int_{\mathbb{R}^n} K(x,y)f(y)\,dy,
\end{equation}
which fulfills
\begin{equation}\label{eq:7.4.4}
\|T(f)\|_{L^2} \leq B\|f\|_{L^2}.
\end{equation}
All Calder\'{o}n--Zygmund operators $T$ form the space $CZO(\delta,A,B)$.
\end{definition}
\begin{lemma}\label{lem:A2}
Let $A,B,\beta > 0$ and $T\in CZO(\beta,A,B)$. Then there is a constant $C = C(\beta,[w]_{A_2})$ such that
for all $w \in A_2$ and $f \in L^2(w)$,
\[
\|T(f)\|_{L^2(w)} \leq C\,(A+B)\|f\|_{L^2(w)}.
\]
\end{lemma}
\begin{proof}
See \cite[Theorem 7.4.6]{Grafakos}.
\end{proof}
\begin{lemma}\label{lem:riesz L2}
Let $w \in A_2$ and $\Rj{j}$ be the $j$-th Riesz transformation for $j = 1, 2, 3$. There exists a constant $C = C (\Atwo{w})> 0$ such that
for any $f \in L^2(w)$,
\begin{equation}\label{YHCCC-42}
\begin{aligned}
\Ltw{\Rj{j}f} \leq C (\Atwo{w}) \Ltw{f}.
\end{aligned}
\end{equation}
\end{lemma}
\begin{proof}
 Since the Riesz transformation is a Calder\'on-Zygmund operator, \eqref{YHCCC-42} is shown by Lemma \ref{lem:A2}.
\end{proof}

\begin{lemma}\label{lem:weighted bernstein}
	Let $w \in A_2$, and $k\in \Z$. Then there exists a constant $C = C (\Atwo{w})> 0$ such that for any $f \in L^2(w)$,
	\[
	\Ltw{\pk f} \leq C (\Atwo{w}) 2^{-k}\Ltw{\pk \nabla f}
	\]
	and
	\begin{equation}\label{YHCCC-43}
	\begin{split}
	\Ltw{\pk \nabla f} \leq C (\Atwo{w}) 2^{k}\Ltw{\pk f}.
	\end{split}
	\end{equation}
\end{lemma}
\begin{proof}
	When $k=0$, it follows from Bernstein's inequality that
	\begin{equation*}
		\|\dot{P}_0 f\|_{L^2(\R^3)}=\|\dot{P}_0 |\nabla|^{-1}R\cdot\nabla f\|_{L^2(\R^3)}\ls  \|\dot{P}_0\nabla f\|_{L^2(\R^3)}.
	\end{equation*}
Since $\dot{P}_0|\nabla|^{-1}R$ is a Calder\'{o}n-Zygmund operator, then by Lemma \ref{lem:A2}, one has

	\[
	\Ltw{\dot{P}_0 f} \leq C_1(\Atwo{w}) \cdot \Ltw{\dot{P}_0 \nabla f}.
	\]
In addition, due to $\Atwo{w(2^{-k}\cdot)}=\Atwo{w}$, then we arrive at
	\begin{equation}\label{YHCCC-44}
	\begin{split}
		\Ltw{\pk f}  = &2^{-\f32 k}\|\dot{P}_0(f(2^{-k}\cdot))\|_{L^2(w(2^{-k}\cdot))}\\
		\le& 2^{-\f32 k}C_1(\Atwo{w(2^{-k}\cdot)})\|\dot{P}_0\nabla(f(2^{-k}\cdot))\|_{L^2(w(2^{-k}\cdot))}\\
		\le& C (\Atwo{w}) 2^{-k}\Ltw{\pk \nabla f}.
	\end{split}
	\end{equation}
Analogously, the proof of \eqref{YHCCC-43} can be completed as for \eqref{YHCCC-44}.
\end{proof}

\vskip 0.2 true cm
{\bf \color{blue}{Conflict of interest}}
\vskip 0.2 true cm

{\bf On behalf of all authors, the corresponding author states that there is no conflict of interest.}

\vskip 0.2 true cm

{\bf \color{blue}{Data availability}}

\vskip 0.2 true cm

{\bf Data sharing is not applicable to this article as no new data were created.}

\end{document}